\documentclass[12pt]{article}

\usepackage{amsmath,amssymb,amsthm,eucal,amscd,tikz,tikz-cd,mathtools}
\usepackage[pagebackref,colorlinks=true,allcolors=black,bookmarksopen,bookmarksdepth=3]{hyperref}
\usepackage[margin=1in]{geometry}

\usepackage{epigraph}

\usepackage{enumitem}
\setenumerate{label=(\roman*),topsep=1pt,itemsep=1pt,partopsep=0pt,parsep=0pt}

\newtheorem{theorem}{Theorem}[section]
\newtheorem{lemma}[theorem]{Lemma}
\newtheorem{corollary}[theorem]{Corollary}
\newtheorem{proposition}[theorem]{Proposition}

\theoremstyle{definition}
\newtheorem{definition}[theorem]{Definition}
\newtheorem{construction}[theorem]{Construction}

\theoremstyle{remark}
\newtheorem{remark}[theorem]{Remark}

\newtheorem{example}[theorem]{Example}

\newtheorem*{example*}{Example}

\numberwithin{equation}{section}

\renewcommand{\Re}{\operatorname{Re}}
\renewcommand{\Im}{\operatorname{Im}}
\newcommand{\HHH}{\mathcal H}
\newcommand{\1}{\mathbf 1}

\newcommand{\A}{\mathcal A}
\newcommand{\B}{\mathcal B}
\newcommand{\C}{\mathcal C}
\newcommand{\E}{\mathcal E}
\newcommand{\ainf}{A_\infty}
\newcommand{\SSS}{\mathcal S}

\newcommand{\sL}{\mathcal L}
\newcommand{\QQ}{\mathbb Q}
\newcommand{\ZZ}{\mathbb Z}
\newcommand{\CC}{\mathbb C}
\newcommand{\CCC}{\mathbf C}
\newcommand{\DDD}{\mathbf D}
\newcommand{\TTT}{\mathbf T}
\newcommand{\RR}{\mathbb R}

\newcommand{\Q}{\mathcal Q}
\newcommand{\M}{\mathcal M}
\newcommand{\N}{\mathcal N}

\newcommand{\T}{\mathcal T}
\newcommand{\G}{\mathcal G}
\newcommand{\X}{\mathcal X}
\newcommand{\V}{\mathcal V}
\newcommand{\PPP}{\mathcal P}
\newcommand{\cc}{\mathfrak c}

\newcommand{\OO}{\mathcal O}
\newcommand{\D}{\mathcal D}
\newcommand{\W}{\mathcal W}

\newcommand{\Op}{\operatorname{Op}}
\newcommand{\id}{\operatorname{id}}
\newcommand{\Tw}{\operatorname{Tw}}
\newcommand{\Ann}{\operatorname{Ann}}
\newcommand{\F}{\mathcal F}

\newcommand{\cone}{\operatorname{cone}}
\newcommand{\cones}{\operatorname{cones}}

\newcommand{\fiber}{[\operatorname{fiber}]}

\newcommand{\op}{\mathrm{op}}
\newcommand{\pt}{\mathrm{pt}}
\newcommand{\Hom}{\operatorname{Hom}}
\newcommand{\Ext}{\operatorname{Ext}}
\newcommand{\Fun}{\operatorname{Fun}}

\newcommand{\Pin}{\operatorname{Pin}}
\newcommand{\colim}{\operatornamewithlimits{colim}}

\newcommand{\sm}{\mathrm{sm}}
\newcommand{\Perf}{\operatorname{Perf}}

\newcommand{\Coh}{\operatorname{Coh}}
\newcommand{\Mod}{\operatorname{Mod}}
\newcommand{\Prop}{\operatorname{Prop}}
\newcommand{\LGr}{\operatorname{LGr}}

\newcommand{\Sh}{\operatorname{Sh}}
\newcommand{\muSh}{\operatorname{\mu sh}}
\renewcommand{\star}{\operatorname{star}}
\renewcommand{\ss}{\operatorname{ss}}
\newcommand{\Set}{\mathrm{Set}}

\newcommand{\pre}{\mathrm{pre}}
\newcommand{\crit}{\mathrm{crit}}
\newcommand{\subcrit}{\mathrm{subcrit}}
\newcommand{\Pic}{\operatorname{Pic}}
\newcommand{\tildetimes}{\mathbin{\tilde\times}}

\newcommand{\tildecup}{\mathbin{\tilde\cup}}

\begin{document}

\title{Microlocal Morse theory of wrapped Fukaya categories}

\author{
Sheel Ganatra,
John Pardon,
and
Vivek Shende
}

\date{August 27, 2023}

\maketitle

\begin{abstract}
The Nadler--Zaslow correspondence famously identifies the finite-dimensional Floer homology groups between Lagrangians
in cotangent bundles with the finite-dimensional Hom spaces between corresponding constructible sheaves.
We generalize this correspondence to incorporate the infinite-dimensional spaces of morphisms `at infinity', given on the Floer side by Reeb trajectories (also known as ``wrapping'') and on the sheaf side by
allowing unbounded infinite rank sheaves which are categorically compact.
When combined with existing sheaf theoretic computations, our results confirm many new instances of homological mirror symmetry.

More precisely, given a real analytic manifold $M$ and a subanalytic isotropic subset $\Lambda$ of its co-sphere bundle $S^*M$, 
we show that the partially wrapped Fukaya category of $T^*M$ stopped at $\Lambda$
is equivalent to the category of compact objects in the 
unbounded derived category of sheaves on $M$ with microsupport inside $\Lambda$.
By an embedding trick, we also deduce a sheaf theoretic description of the wrapped Fukaya category of any Weinstein sector admitting a stable polarization.
\end{abstract}

\setcounter{tocdepth}{3}

\newpage
\tableofcontents
\newpage

\section{Introduction}

\epigraph{I saw the angel in the marble, and carved until I set him free.}{---Michelangelo (attributed)}

\noindent
The calculation of Fukaya categories of symplectic manifolds has emerged as a question of central interest in geometry.
Within symplectic geometry itself, many questions may be phrased in terms of intersections of Lagrangian submanifolds; as the Fukaya category is built from these, it is a natural tool for their study.
In low-dimensional topology, a number of invariants of smooth manifolds and smooth knots can
be extracted from Fukaya categories of associated symplectic manifolds.
Homological mirror symmetry, a largely conjectural correspondence arising from non-rigorous reasoning in mathematical physics, further predicts that many categories of interest in algebraic geometry and representation theory also arise as Fukaya categories.

Beyond the intrinsic interest in confirming or explaining these predictions, 
knowing that a category of interest arises as a 
Fukaya category suggests the existence of additional structures. For one example, morphism
spaces in the Fukaya category are Floer homology chain complexes, and as such come with a 
natural basis; this often `explains' the existence of previously known `canonical' bases of these 
Hom spaces. For another, 
the relative ease of constructing symplectomorphisms (which act on the relevant Fukaya
categories) gives a natural source of automorphisms of these categories that are far less
apparent from other points of view.
The difficulty in calculating Fukaya categories, which is present in all of the aforementioned settings, stems from the global and analytic nature of the pseudo-holomorphic disks appearing in the definition.

\emph{In this paper, we obtain a combinatorial description of the partially wrapped Fukaya categories
of all stably polarized Weinstein manifolds (more generally, sectors), by showing that they are isomorphic to certain corresponding categories of microlocal sheaves.}

\subsection{Weinstein manifolds and partially wrapped Fukaya categories}

A vector field $Z$ on a symplectic manifold $(X, \omega)$ is said to be \emph{Liouville} when $L_Z \omega = \omega$.
Recall that in this case $\lambda := \omega(Z, \cdot)$ is a primitive for $\omega$, and that 
such symplectic manifolds are necessarily non-compact.
Such a triple $(X, \omega, Z)$ (equivalently $(X,\lambda)$) is
called a \emph{Liouville manifold} if, in addition,
the non-compact ends of $X$ are identified (necessarily uniquely) by $Z$ with the positive end of the symplectization of a contact manifold \cite{eliashberggromovconvexsymplectic}.  
The \emph{core} of a Liouville manifold is the locus of points $\cc_X$ which do not escape to infinity under the
Liouville flow.  The inclusion $\cc_X\subseteq X$ is a homotopy equivalence, and in some sense $\cc_X$ carries all of the symplectic topology of $X$ as well.

A Liouville manifold $(X, \omega, Z)$ is said to be \emph{Weinstein} if $Z$ is gradient-like \cite{cieliebakeliashberg}.  The key feature of such manifolds is that
the core $\cc_X$ is a union of isotropic submanifolds, and moreover admits transverse Lagrangian disks (`cocores') at its smooth Lagrangian points.
Prototypical examples include cotangent bundles,
affine algebraic varieties, and more generally (finite
type) Stein complex manifolds. Many examples of interest in geometric
representation theory, such as conical symplectic resolutions, quiver
varieties, moduli of Higgs bundles, and cluster varieties are in this class. 
Moreover, any compact symplectic manifold whose symplectic form has rational periods can be presented as a compactification of a Weinstein manifold \cite{donaldson,giroux}, and hence, through a strategy introduced in \cite{seideldeformations}, understanding the Fukaya categories of Weinstein manifolds serves as a stepping stone to studying the Fukaya categories of closed symplectic manifolds.

While various analytic difficulties in Floer theory are simplified
in the Liouville setting (due to strong topological and geometric control on pseudo-holomorphic disks),
there is a significant new layer of complexity possible, thanks to the
non-compactness of the target space. Namely, as has been understood for some time, it is desirable in this context to enlarge the Fukaya category of compact Lagrangians by adding certain non-compact (properly embedded, conical at infinity) Lagrangians as well.
These larger Fukaya categories often have better formal properties due to there being an ample supply of non-compact Lagrangians,
and they are also required by mirror symmetry, where non-compact Lagrangians in non-compact targets arise as mirror objects to
sheaves on non-compact or non-Calabi--Yau manifolds, whose $\Ext$ groups could
be of infinite rank or fail to satisfy Poincar\'e duality.
Fukaya categories of non-compact Lagrangians also have bearing on questions about the Reeb dynamics at infinity.
There are many different ways to add non-compact Lagrangians to the Fukaya category, with substantially different results; the basic parameters are (1) in which directions at
infinity to allow non-compact Lagrangians and (2) in what direction and by how
much to perturb (`wrap') Lagrangians at infinity when computing Floer homology.

The framework of \emph{partially wrapped Fukaya categories} \cite{aurouxborderedfloericm,aurouxborderedfloergokova,sylvanthesis,gpssectorsoc,gpsdescent}
has emerged as a way to describe and relate different prescriptions for asymptotics 
and wrapping.  One specifies a subset at
infinity which Lagrangians cannot limit to or wrap past, called the
\emph{stop}.
In the resulting category, Lagrangians which are isotopic in the complement of the
stop induce isomorphic objects, and symplectomorphisms preserving the stop induce
autoequivalences.
The resulting category is also invariant under isotopies of the complement of the stop.
A \emph{stopped Liouville manifold} $(X,\Lambda)$ consists of a Liouville manifold $X$ and a stop $\Lambda$.
The \emph{relative core} $\cc_{X,\Lambda}$ of $(X,\Lambda)$ is the set of those which do not escape to the complement of $\Lambda$ at infinity under the positive Liouville flow.

The partially wrapped setting includes variants of many previous constructions:
\begin{itemize}
\item
When the stop is empty one obtains the (fully) wrapped Fukaya category of Abouzaid--Seidel \cite{abouzaidseidel}.
\item Given a smooth Legendrian $\Lambda$,  there is a naive `infinitesimally wrapped' Fukaya category with asymptotics along $\Lambda$ given equivalently by either: (1) take the stop to be the complement of a small regular neighborhood of $\Lambda$, or (2) take the stop to be $\Lambda$ and consider just the full subcategory of Lagrangians which admit wrappings converging to $\Lambda$ (see Section \ref{NZ} for further discussion, including a comparison with \cite{nadlerzaslow}).\footnote{This naive category embeds fully faithfully into the category of `proper modules' over the partially wrapped Fukaya category stopped at $\Lambda$, which should be regarded as the more correct category.  It is an open and likely hard geometric question to determine when this embedding is an equivalence, already for $\Lambda = \emptyset$.}
More generally, we can take $\Lambda$ to be the core of a Liouville hypersurface.
\item
The Fukaya category of a \emph{Landau--Ginzburg model} $\mathrm{w}: Y \to \CC$, also known as the \emph{Fukaya--Seidel category} when $\mathrm{w}$ is a Lefschetz fibration \cite{seidelbook}, can be modeled as the partially wrapped Fukaya category of $Y$ stopped at (the core of) the Weinstein hypersurface $\mathrm{w}^{-1}(-\infty)$.\footnote{From our point of view, this should just be taken as the \emph{definition} of the Fukaya--Seidel category.  However, we note there are some technical differences between this definition and the standard definition, and a careful proof of their equivalence is, as far as we know, a folk result whose proof has no available reference (though a special case is treated in \cite[Sec.\ 8.6]{gpsdescent}).}
\end{itemize}
Of course, partially wrapped Fukaya categories form a significantly broader class than infinitesimal or Fukaya--Seidel categories (e.g.\ they can have infinite-dimensional morphism spaces).
An illustrative example: categories of coherent complexes on arbitrary (not necessarily compact or smooth) $n$-dimensional toric stacks
can be shown equivalent to partially wrapped Fukaya categories of $(\CC^*)^n = T^* T^n$, 
by combining the sheaf theoretic work of \cite{kuwagaki} with the main theorem of the present article.

Of particular importance are (possibly singular) isotropic stops.
Isotropicity of the stop plays the same role as the Weinstein condition on the symplectic manifold itself; for instance, the core of a \emph{Weinstein hypersurface at infinity} is a typical singular isotropic stop of interest to mirror symmetry.

\subsection{Topological and sheaf theoretic interpretations}

Despite their analytic origins, Fukaya categories have often been found to admit topological
interpretations.  The first prototype is the fact that the Lagrangian Floer homology of an exact Lagrangian is nothing
other than its ordinary cohomology.  The work of Nadler and Zaslow \cite{nadlerzaslow, nadlermicrolocal} provides a 
sweeping generalization of this, identifying infinitesimally wrapped Floer homologies of exact Lagrangians in cotangent bundles
with morphisms of sheaves on the base.  In a seemingly different direction, it was observed that \emph{wrapped} Floer homology
also has a topological interpretation.  Indeed, work of Abbondondalo and Schwarz \cite{abbondandoloschwarz-floer, abbondandoloschwarz-conormal,
abbondandoloschwarz} (see also Cieliebak and Latschev \cite{cieliebak-latschev}) found many instances where wrapped Floer homology is isomorphic to
the homology of spaces of paths and loops.
Building on these, Abouzaid showed that in fact the wrapped Fukaya category of 
a cotangent bundle is naturally identified with perfect modules over chains on the based loop space of the base \cite{abouzaidcotangent} (see also \cite{abouzaidplumbing}).  This last
result may be restated (by the $\infty$ version of the van Kampen theorem): the wrapped Fukaya category of a cotangent bundle is the global
sections of the constant cosheaf of categories on the zero section with costalk $\Perf\ZZ$.  This formulation exhibits an instance of a more general conjecture of Kontsevich 
\cite{kontsevichnotes}: the wrapped Fukaya category of any Weinstein manifold $X$ should 
be the global sections of a cosheaf of categories on its core $\cc_X$.

Nadler's work \cite{nadler-wrapped} unified these points of view, by proposing that while infinitesimally wrapped Fukaya categories are modeled by
(micro)sheaves with perfect (micro)stalks (as in \cite{nadlerzaslow}), the partially wrapped category should be modeled by compact objects 
in the category of all (micro)sheaves with appropriate microsupport conditions.  For essentially formal reasons, these categories of compact objects may be organized into a cosheaf of categories, so Nadler's proposal is a strengthening of Kontsevich's conjecture.  
At the time of Nadler's original proposal, microsheaves were only defined for subsets of cotangent bundles, 
but the high codimension embedding trick from \cite{shende-microlocal} has now defined a category of microsheaves 
on the core of any Liouville manifold \cite{nadler-shende}.

Since microlocal sheaf categories are entirely combinatorial/topological in nature, Nadler's proposal is a (conjectural) computation of (partially) wrapped Fukaya categories.
An illustrative example calculation is given in
\cite{nadler-wrapped}, where it was shown that the relevant category of microlocal sheaves on the skeleton of the (higher dimensional) symplectic pairs of pants 
matched the mirror-symmetric prediction for the wrapped Fukaya category (which was more recently verified directly \cite{lekili-polishchuk-pants}).

\begin{example*}
Remarkably, while wrapped and infinitesimal Floer homologies are rather different creatures, the same does not appear to be the case on the sheaf side 
of Nadler's proposal.  Notably, there is nothing like `wrapping' on the sheaf side; instead, purely categorical operations.  
As it is difficult to appreciate the depth of the distinction at first, let
us return to the example of the cotangent bundle, with empty stop.  The core is just the zero section, and sheaves microsupported in the zero section
are (almost by definition) nothing other than locally constant sheaves.  If we require these to have finite stalks, then we are studying finite
rank local systems (which coincides with the infinitesimally wrapped Fukaya category, i.e.\ in this case the Fukaya category of compact Lagrangians).
On the other hand, a typical compact object in this category is the `tautological' local system, with fiber given by chains on the based loop
space.  Usually of infinite rank, this object is best understood in terms of the functor it co-represents: taking the stalk at a point.   Note that no `wrapping' 
appears in this purely categorical procedure proposed by Nadler, yet it does in fact correctly recover the wrapped Fukaya category.
\end{example*}

\subsection{Main results}

We now fix notation and state our main results more precisely.
Theorem \ref{sheaffukayaequivalence} concerns the special case of cotangent bundles, and its proof comprises the bulk of the paper.
Theorem \ref{microsheaffukayaequivalence} is derived from Theorem \ref{sheaffukayaequivalence} and concerns more general stably polarized Liouville manifolds.

For a Liouville symplectic manifold $X$ and closed subset $\Lambda \subseteq \partial_\infty X$, we write $\W(X, \Lambda)$ for the (partially) wrapped Fukaya category 
of $X$, stopped at $\Lambda$.  Its objects are Lagrangians $L\subseteq X$ which are eventually conical and disjoint from $\Lambda$ at infinity, 
and its morphism complexes are Floer cochains after wrapping Lagrangians in the complement of $\Lambda$.   
It is an $\ainf$ category defined in \cite{gpssectorsoc,gpsdescent} (see also \cite{abouzaidseidel,abouzaidseidelunpublished,sylvanthesis}); we 
review its definition at the beginning of Section \ref{sec:fukaya}.   Particularly important objects of $\W(X,\Lambda)$ include: the Lagrangian linking disks to the smooth Legendrian points of $\Lambda$ \cite[Sec.\ 5.3]{gpsdescent}, the Lagrangian cocore disks when $X$ is Weinstein, and the cotangent fibers when $X=T^*M$ (which may be viewed as a special case of cocore disks).
We write $\Perf\W(X,\Lambda)$ for the idempotent-completed pre-triangulated closure of $\W(X,\Lambda)$.
Whenever $\Lambda \subseteq \Lambda'$, there is a tautological functor $\W(X, \Lambda') \to \W(X, \Lambda)$.

For a smooth manifold $M$, we write $\Sh(M)$ for the dg category of sheaves of dg $\ZZ$-modules on $M$.
The microsupport of a sheaf $\F$ is a closed conical locus $\ss(\F) \subseteq T^*M$ whose role is to encode, infinitesimally, which restriction maps are quasi-isomorphisms.
We write $\Sh_\Lambda(M)$ for the full subcategory of $\Sh(M)$ spanned by those sheaves whose microsupport at infinity is contained in $\Lambda$.
Particularly nice functors $\Sh_\Lambda(M)\to\Mod\ZZ$ include the stalk functors at points of $M$ and the microstalk functors at smooth Legendrian points of $\Lambda$.
These definitions are reviewed in Section \ref{sec:sheaves}.
We denote by $\Sh_\Lambda(M)^c$ the category of compact objects in $\Sh_\Lambda(M)$.
The reader is cautioned that the compact objects of $\Sh_\Lambda(M)$ do not necessarily have perfect stalks or bounded homological degree; that is, they need not be constructible sheaves in the usual sense.
The relevance of such objects on the sheaf side was pointed out in \cite{nadler-wrapped}, where Theorem \ref{sheaffukayaequivalence} was implicitly conjectured by the terminology `wrapped sheaves' for compact objects of $\Sh_\Lambda(M)$ and `wrapped (microlocal) skyscrapers' for co-representatives of (micro)stalk functors.

\begin{theorem}\label{sheaffukayaequivalence}
Let $M$ be a real analytic manifold, and let $\Lambda \subseteq S^*M$ be a subanalytic closed isotropic subset.
There there is a canonical equivalence of categories
\begin{equation}\label{sheaffukayaequivalenceeqn}
\Perf\W(T^*M,\Lambda)^\op=\Sh_\Lambda(M)^c
\end{equation}
which carries the linking disk at any smooth Legendrian point $p\in\Lambda$ to a co-representative of the microstalk functor at $p\in\Lambda$ and carries the cotangent fiber at a point $p\in M$ not in the image of $\Lambda$ to a co-representative of the stalk functor.
\end{theorem}

\begin{remark}
Rather than passing to the opposite category of the Fukaya category, we could equivalently negate either $\Lambda$ or the Liouville form on $T^*M$.
\end{remark}

To prove Theorem \ref{sheaffukayaequivalence}, we do not calculate a single $\W(T^*M,\Lambda)$ on its own.  Instead, we calculate the functor $\Lambda\mapsto\Perf\W(T^*M,\Lambda)^\op$ (functoriality is with respect to inclusions $\Lambda\subseteq\Lambda'$) which is a much more rigid object.
In fact, we formulate a pair of axioms which uniquely charaterize such a system of categories $\Lambda\mapsto\C(\Lambda)$ (Section \ref{sec:axioms}), so the proof of Theorem \ref{sheaffukayaequivalence} then reduces to verifying these axioms for both $\Lambda\mapsto\Perf\W(T^*M,\Lambda)^\op$ (Section \ref{sec:fukaya}) and $\Lambda\mapsto\Sh_\Lambda(M)^c$ (Section \ref{sec:sheaves}).

The underlying reason this strategy can succeed is that there are special stops $\Lambda$ (specifically $\Lambda=N^*_\infty\SSS$, the union of conormals to the strata of a Whitney triangulation $\SSS$) for which the Reeb dynamics in the complement of $\Lambda$ are simple, thus making it tractable to show \eqref{sheaffukayaequivalenceeqn} by direct calculation.
Since every $\Lambda$ is a subset of some $N^*_\infty\SSS$, it then suffices to show that both sides of \eqref{sheaffukayaequivalenceeqn} transform in the same way when $\Lambda$ gets smaller.
On the Fukaya side, the functor $\W(T^*M,\Lambda')\to\W(T^*M,\Lambda)$ for $\Lambda'\supseteq\Lambda$ is the quotient by the linking disks to $\Lambda'\setminus\Lambda$; this was established recently in \cite{gpsdescent}.
On the sheaf side, one quotients by co-representatives of microstalks; this is ultimately a consequence of co-isotropicity of the microsupport \cite[Thm.\ 6.5.1]{kashiwara-schapira}.
The identification of linking disks with microstalks matches the wrapping exact triangle of \cite{gpsdescent} with the microlocal Morse description of sheaf cohomology from \cite{goresky-macpherson,kashiwara-schapira}.
The conclusion is then that choosing a Whitney triangulation $\SSS$ of $M$ whose conormal $N^*_\infty\SSS$ contains $\Lambda$ yields a description of both categories $\Perf\W(T^*M,\Lambda)^\op$ and $\Sh_\Lambda(M)^c$ as the same localization of the category $\Perf\W(T^*M,N^*_\infty\SSS)^\op=\Perf\SSS=\Sh_{N^*_\infty\SSS}(M)^c$.

This approach to the proof of Theorem \ref{sheaffukayaequivalence} is rather different from the previous computations of Fukaya categories of cotangent bundles \cite{nadlerzaslow, nadlermicrolocal,abouzaidcotangent}.
In particular, we rely on no results from these articles.
In fact, our ability to add geometry to simplify the situation is such that the only Floer cohomology calculations which need to be made in this entire article are between Lagrangians which intersect in at most one point, obviating, in particular, the need to ever compute a holomorphic disc.
We therefore expect that the proofs of the results in this paper would apply to the case of more general (e.g.\ sphere spectrum) coefficients, provided one has access to the definitions of the sheaf and Fukaya categories in these settings.

The equivalence of Theorem \ref{sheaffukayaequivalence} is also functorial under open inclusions:

\begin{proposition}\label{sfecompatibility}
Let $M'\hookrightarrow M$ be an analytic open inclusion of analytic manifolds.
For subanalytic isotropics $\Lambda'\subseteq S^*M'$ and $\Lambda\subseteq S^*M$ with $\Lambda'\supseteq\Lambda\cap S^*M'$, the following diagram commutes
\begin{equation}
\begin{tikzcd}
\Perf\W(M',\Lambda')^\op\ar[equal]{d}{Thm\ \ref{sheaffukayaequivalence}}\ar{r}&\Perf\W(M,\Lambda)^\op\ar[equal]{d}{Thm\ \ref{sheaffukayaequivalence}}\\
\Sh_{\Lambda'}(M')^c\ar{r}&\Sh_\Lambda(M)^c
\end{tikzcd}
\end{equation}
where the bottom horizontal arrow is (the restriction to compact objects of) the left adjoint of restriction.
\end{proposition}

In Section \ref{sec:corollaries}, we detail a number of applications and corollaries of Theorem \ref{sheaffukayaequivalence}.
These include a version of the original Nadler--Zaslow correspondence, translations of the microsheaf theoretic work on mirror symmetry for toric varieties and toric boundaries, and a sheaf theoretic description of the augmentation category of the partially wrapped Floer cochains of the linking disk to a Legendrian (expected to be equivalent to the Legendrian DGA), valid for the jet bundle of a manifold of any dimension.
After the present work, a host of sheaf theoretic calculations \cite{FLTZ-Morelli, shende-treumann-zaslow, STWZ, shende-treumann-williams, kuwagaki, nadler-mirrorLG, nadler-wrapped, gammage-shende} can now be understood as computations of wrapped Fukaya categories.

Finally, in Section \ref{microsection}, we turn from cotangent bundles to the general setting of stably polarized Weinstein manifolds.  
We proceed by combining the `doubling trick' of \cite[Ex.\ 10.7 and 13.4]{gpsdescent} and the `antimicrolocalization' of \cite[Sec.\ 7]{nadler-shende}
to reduce to the cotangent bundle case.  We arrive the following sheaf theoretic description of partially wrapped Fukaya categories: 

\begin{theorem}\label{microsheaffukayaequivalence}
Let $X$ be a real analytic Liouville manifold, and let $\Lambda\subseteq\partial_\infty X$ be a stop whose relative core $\cc_{X,\Lambda}:=\cc_X\cup(\Lambda\times\RR)\subseteq X$ is subanalytic singular isotropic.
For any stable polarization\footnote{A stable polarization of a symplectic manifold is the expression of its tangent bundle plus $\CC^k$ (some finite $k$) as the complexification of a real vector bundle; a choice of stable polarization controls the `twisting' of the categories on both sides of \eqref{microsheaffukayaequivalenceeqn}.
Many examples of interest are stably polarized, such as all cotangent bundles and all complete intersections in $\CC^n$.}
of $X$, there is a fully faithful functor
\begin{equation}\label{microsheaffukayaequivalenceeqn}
\Perf\W(X,\Lambda)^\op\hookrightarrow\muSh_{\cc_{X,\Lambda}}(\cc_{X,\Lambda})^c,
\end{equation}
where $\muSh_{\cc_{X,\Lambda}}(\cc_{X,\Lambda})$ denotes the category of microlocal sheaves on $\cc_{X,\Lambda}$.
This functor sends a homological cocore at a smooth Lagrangian point $p$ of $\cc_{X,\Lambda}$ to a co-representative of the microstalk at $p$.

In particular, if $X$ is Weinstein, or more generally admits homological cocores, then \eqref{microsheaffukayaequivalenceeqn} is an equivalence.
\end{theorem}

Let us comment on the hypotheses of Theorem \ref{microsheaffukayaequivalence} and, in particular, argue that it applies to all stably polarized Weinstein sectors.
The analyticity assumptions hold in most concrete cases of interest, and abstractly speaking, any Weinstein manifold (more generally, sector) may be perturbed so as to be real analytic and to have subanalytic relative core (see Corollary \ref{analyticrelativeskeleton}).
A homological cocore at a smooth Lagrangian point $p\in\cc_{X,\Lambda}$ is an object of $\Perf\W(X,\Lambda)$ whose image in $\Perf\W((X,\Lambda)\times(\CC_{\Re\geq 0},\infty))$ is the linking disk at $p\times\infty$.
Admitting homological cocores means that every smooth Lagrangian point of $\cc_{X,\Lambda}$ has a homological cocore; this condition turns out to depend only on $X$ up to deformation, and holds whenever $X$ is Weinstein.
For $(X,\Lambda)$ as in Theorem \ref{microsheaffukayaequivalence}, the stabilization $X\times\CC$ always admits homological cocores, so there is always an equivalence $\Perf\W((X,\Lambda)\times(\CC,\pm\infty))^\op=\muSh_{\cc_{X,\Lambda}\times\RR}(\cc_{X,\Lambda}\times\RR)^c=\muSh_{\cc_{X,\Lambda}}(\cc_{X,\Lambda})^c$.

The embedding of Theorem \ref{microsheaffukayaequivalence} depends \emph{a priori} on a choice of analytic Liouville hypersurface embedding of $X\times\CC^k$ into $S^*M$ for some auxiliary analytic manifold $M$, compatible with stable polarizations (part of the proof is to show such data always exists).
We expect our methods could be extended to show that the embedding of Theorem \ref{microsheaffukayaequivalence} is independent of this choice.

The embedding of Theorem \ref{microsheaffukayaequivalence} associated to a given analytic Liouville hypersurface embedding $X\hookrightarrow S^*M$ is, by construction, compatible with the equivalence of Theorem \ref{sheaffukayaequivalence} in the sense that the following diagram commutes:
\begin{equation} \label{theoremscommute}
\begin{tikzcd}
\Perf\W(X,\Lambda)^\op\ar{r}\ar[hook]{d}[swap]{\text{Thm \ref{microsheaffukayaequivalence}}}&\Perf\W(T^*M,\overline{\cc_{X,\Lambda}})^\op\ar[equal]{d}{\text{Thm \ref{sheaffukayaequivalence}}}\\
\muSh_{\cc_{X,\Lambda}}(\cc_{X,\Lambda})^c\ar{r}{\mu^*}&\Sh_{\overline{\cc_{X,\Lambda}}}(M)^c
\end{tikzcd}
\end{equation}
(see Proposition \ref{sheafdescmicrolocalwithremoval}), where $\mu^*$ denotes the left adjoint to microlocalization.
Using this compatibility, it is proven in \cite{gammage-shende} that for a Fano toric stack $Y$ with toric divisor $D$, there is a commutative diagram
\begin{equation}\label{fsmirror}
\begin{tikzcd}
\Perf \W(W^{-1}(-\infty)) \ar{r} \ar[equal]{d} & \Perf \W(X,W) \ar{r} \ar[equal]{d}  & \Perf \W(X) \ar{r} \ar[equal]{d} & 0 \\
\Coh(D) \ar{r} & \Coh(Y) \ar{r} & \Coh(Y \setminus D) \ar{r} & 0
\end{tikzcd}
\end{equation}
where $W:X\to\CC$ is the mirror Landau--Ginzburg model (see Example \ref{veryaffine}).

\medskip

\textbf{Convention.} Throughout this document, we work in the setting of dg and, equivalently, $A_\infty$ categories over $\ZZ$ (or more generally any commutative ring).
We only ever
consider ``derived'' functors, we only ever mean ``homotopy'' limits or colimits,
and we systematically omit the word ``quasi''.  By modules, 
we mean dg or $A_\infty$ modules, e.g.\ by $\ZZ$-modules we mean the category of chain complexes of abelian $\ZZ$-modules, localized at quasi-isomorphisms,
except when, as in this sentence, we qualify it with the word `abelian'. In Section \ref{categoricalsection} we detail our assumptions about these categories and collect relevant categorical notions which will appear throughout the paper.

\subsection{Acknowledgements}

We thank Mohammed Abouzaid, Roger Casals, Alexander Efimov, Tobias Ekholm, Benjamin Gammage, David Nadler, Amnon Neeman, and Lenhard Ng 
for helpful discussions, some of which took place during visits to the American Institute of Mathematics. 

S.G.\ was partially supported by NSF grant DMS--1907635, and would also like to thank the Simons Laufer Mathematical Sciences Institute (previously known as MSRI) for its
hospitality during visits in Spring 2018 (supported by NSF grant DMS--1440140) and Fall 2022 (supported by NSF grant DMS--1928930)
during which some of this work was completed.
This research was conducted during the period J.P.\ served as a Clay Research Fellow and was partially supported by a Packard Fellowship and by the National Science Foundation under the Alan T.\ Waterman Award, Grant No.\ 1747553.
V.S.\ was supported by the NSF CAREER grant DMS--1654545.

\section{Stratifications}

\subsection{Generalities} 

Let $X$ be a topological space.  By a stratification $\SSS$ of $X$, we mean a locally finite decomposition into disjoint locally closed subsets $\{X_\alpha\}_{\alpha\in\SSS}$, called
strata, such that each boundary $\overline{X_\alpha}\setminus X_\alpha$ is a union of other strata $X_\beta$.
The collection of strata $\SSS$ is naturally a poset, in which there is a map $\beta \to \alpha$ iff $X_\alpha \subseteq \overline{X_\beta}$.

\begin{remark}\label{stratvsposet}
The poset $\SSS$ does not generally capture the homotopy type of the space $X$.
Conditions under which it does (contractibility of various strata/stars) are well known and recalled below.
\end{remark}

We will say a subset $Y \subseteq X$ is $\SSS$-constructible when it is a union of strata of the stratification $\SSS$ of $X$.  
We say that a stratification $\T$ refines a stratification $\SSS$ when the strata of $\SSS$ are  $\T$-constructible.  

We recall that an abstract simplicial complex on a vertex set $\V$ is a collection $\Sigma$ 
of nonempty finite subsets of $\V$, containing all singletons and all subsets of elements of $\Sigma$.  By a simplicial complex, 
we mean the geometric realization $\left|\Sigma\right|$ of an abstract simplicial complex $\Sigma$;  it comes with a stratification by the `open simplices' (which, of course, are locally closed, not necessarily open, subsets of $\left|\Sigma\right|$). 
We say a stratification $\SSS$ on $X$ is a triangulation when there exists a homeomorphism $\left|\Sigma\right|\xrightarrow\sim X$ identifying stratifications.
We never impose any sort of regularity condition (differentiability, smoothness, analyticity, etc.)\ on this homeomorphism, even in the context of stratifications of a given regularity class.
Note that the following are \emph{not} triangulations: 
a stratification of a circle into single point and its complement, or into two points and their complement; the stratification into three points and their complement is a triangulation.

The open star of a stratum is the union of strata whose closures contain it.  Taking stars reverses the inclusion: we have 
$X_\alpha \subseteq \overline{X_\beta} \iff \star(X_\beta) \subseteq \star(X_\alpha)$.
Note that $\star(X_\alpha) \cap \star(X_\beta) = \bigcup_{\alpha\leftarrow\gamma\to\beta}\star(X_\gamma)$.
For triangulations, we can do better: $\star(X_\alpha) \cap \star(X_\beta) = \star(X_\gamma)$ where $\gamma$ is the simplex spanned by the vertices of $\alpha$ union the vertices of $\beta$ (if this simplex is present), and otherwise $\star(X_\alpha) \cap \star(X_\beta) =\emptyset$.

For a $C^p$ manifold $M$, we say a stratification $\SSS$ is $C^p$ if each stratum $M_\alpha$ is a 
(locally closed) $C^p$ submanifold.

A $C^1$ stratification $\SSS$ of a $C^1$ manifold $M$ 
is called a \emph{Whitney stratification} iff it satisfies Whitney's conditions (a) and (b). 
These are usually stated as the following conditions on pairs of strata $X$ and $Y$ of $\SSS$:
\begin{itemize}
\item[(a)]For any sequence $x_i\in X$ converging to $y\in Y$ such that $T_{x_i}X$ converges to a subspace $V\subseteq T_yM$, we have $T_yY\subseteq V$.
\item[(b)]For any pair of sequences $x_i\in X$ and $y_i\in Y$ both converging to $y\in Y$ such that $T_{x_i}X$ converges to a subspace $V\subseteq T_yM$ and the secant directions from $y_i$ to $x_i$ converge to a line $L\subseteq T_yM$, we have $L\subseteq V$.
\end{itemize}
By compactness of flag varieties, we may pass to convergent subsequences, and hence conditions (a) and (b) may be reformulated as follows:
\begin{itemize}
\item[(a)]For strata $Y\subseteq\overline X$, as $X\ni x\to y\in Y$, the tangent spaces $T_xX$ become arbitrarily close to containing $T_yY$, uniformly over compact subsets of $Y$.
\item[(b)]For strata $Y\subseteq\overline X$, as $X\ni x\to y\in Y$, the secant lines between $x$ and $y$ become arbitrarily close to being contained in $T_xX$, uniformly over compact subsets of $Y$.
\end{itemize}
Whitney's condition (a) is equivalent to the assertion that the union of conormals $N^*\SSS:=\bigcup_\alpha N^*M_\alpha\subseteq T^*M$ is closed.
In fact, it is not hard to see that Whitney's condition (b) implies Whitney's condition (a) \cite[Prop.\ 2.4]{mathernotes}.
A Whitney stratification is, by definition, at least $C^1$; it makes sense to consider $C^p$ Whitney stratifications for any $p\geq 1$, including $p=\infty$.

\medskip

In order to guarantee the existence of Whitney stratifications, we will 
ultimately restrict to the setting of (real) analytic manifolds and subanalytic stratifications.
We recall that a set is defined to be subanalytic when locally (i.e.\ in a neighborhood of every point of its closure) it is the analytic image of a relatively compact semianalytic set (i.e.\ locally defined by finitely many analytic inequalities).
The canonical modern reference for subanalytic geometry is \cite{bierstone-milman}.
By a subanalytic stratification, we mean a stratification in which all strata are subanalytic.
Every subanalytic stratification admits a subanalytic refinement in which all strata are locally closed analytic submanifolds.
It is a fundamental result that for any locally finite collection of subanalytic subsets of an analytic manifold, there exists a subanalytic stratification with respect to which all the subsets are constructible.
For proofs of these results, see \cite{bierstone-milman,shiotabook}.
We also require the result that every subanalytic stratification admits a refinement to a subanalytic Whitney triangulation \cite{shiotawhitney,czapla,czaplapawlucki}.

\begin{remark}
Wherever we have written `subanalytic', one could substitute `defineable' with respect to any fixed analytic-geometric category \cite{vandendries-miller,vandendries}.
Every defineable stratification has a defineable refinement to a $C^p$ Whitney triangulation for any given $p<\infty$ \cite{shiotawhitney,czapla,czaplapawlucki}.
The fact that this is not known to hold for $p=\infty$ does not create any difficulties.
\end{remark}

\begin{lemma}\label{whitneyintersect}
Let $M$ be a manifold with Whitney stratification $\SSS$.
If $N\subseteq M$ is a locally closed submanifold transverse to every stratum of $\SSS$, then the intersected stratification $\SSS\cap N$ is a Whitney stratification of $N$.
\end{lemma}

\begin{proof}
It suffices to verify that $\SSS\cap N$ satisfies Whitney (b).
Thus consider a pair of strata $X\cap N$ and $Y\cap N$ with $Y\subseteq\overline X$.
Whitney (b) for the stratification $\SSS$ guarantees that the secant line from $x\in X\cap N$ to $y\in Y\cap N$ becomes arbitrarily close to being contained in $T_xX$ as $x\to y$, uniformly over compact subsets of $Y\cap N$.
On the other hand, Whitney (b) for the stratification for $\SSS\cap N$ requires this secant line to become arbitrarily close to being contained in $T_xX\cap T_xN$, a stronger condition.
Since $N$ is a submanifold, the secant line certainly becomes arbitrarily close to being contained in $T_xN$.
Our task is thus to pass from being close to $T_xX$ and $T_xN$ to being close to their intersection $T_xX\cap T_xN$.
It thus suffices to show that $T_xX$ and $T_xN$ are uniformly transverse as $x\to y$.
Let us see how this follows from Whitney (a) for $\SSS$.
Since $N$ is transverse to $Y$, we have $T_yY+T_yN=T_yM$.
Whitney (a) for $\SSS$ means that $T_xX$ is arbitrarily close to containing $T_yY$ as $x\to y$.
Thus $T_xX+T_xN=T_xM$ uniformly as $x\to y$, as desired.
\end{proof}

\begin{remark}\label{addsubmanifoldremark}
Here is a typical application of Lemma \ref{whitneyintersect}.
Let $M$ be a manifold with a Whitney stratification $\SSS$, and let $N\subseteq M$ be a closed (as a subset) submanifold transverse to every stratum of $\SSS$.
Then the stratification $\SSS_N$ of $M$ with strata $M_\alpha\cap N$ and $M_\alpha\setminus N$ for $\alpha\in\SSS$ (the poset of strata $\SSS_N$ is thus $\SSS\times\{(M\setminus N)>N\}$) is Whitney.
Indeed, Whitney (b) for $\SSS_N$ is a special case of Whitney (b) for $\SSS$ and $\SSS\cap N$.
A slight modification of this example will also come up later.
Let $B\subseteq M$ be a closed ball whose boundary is transverse to a Whitney stratification $\SSS$.
The stratification of $M$ by $M_\alpha\cap\partial B$, $M_\alpha\cap B^\circ$, and $M_\alpha\setminus B$, for strata $M_\alpha\subseteq M$, is now Whitney by the same argument using Lemma \ref{whitneyintersect}.
\end{remark}

\subsection{Microlocal approximation of constructible sets}\label{microlocalapprox}

A constructible set $X$ with respect to a Whitney stratification is in general quite singular.
Our goal in this section is to show how such sets can be \emph{microlocally} approximated by manifolds-with-corners $X^\eta$ parameterized by small $\eta>0$, in the sense that $X^\eta\to X$ and the conormal to $X^\eta$ converges to (being contained in) the conormal of $X$ as $\eta\to 0$.
This result will be used in the proof of Proposition \ref{better841} and in Section \ref{fukayacalculationssec}.
We fix an integer $p\geq 1$, possibly $p=\infty$, and work throughout with stratifications of class $C^p$.

A $C^p$ \emph{radius function} for a locally closed $C^p$ submanifold $Y\subseteq M$ is a pair $(U,\rho)$ where $U\subseteq M$ is an open set containing $Y$ and $\rho:U\to\RR_{\geq 0}$ is of class $C^p$ on $U\setminus Y$ and satisfies the following three conditions:
\begin{itemize}
\item$\rho^{-1}(0)=Y$;
\item$\rho$ is Lipschitz
on a neighborhood of any compact subset of $Y$;
\item the lim inf of the evaluation of $d\rho(x)$ on the secant direction from $y\in Y$ to $x\in M$ is bounded below by some $\epsilon>0$ as $x\to y$, provided $y$ is constrained to a compact subset of $Y$ and the ratio $\frac{d(x,y)}{d(x,Y)}$ is bounded by some fixed $N<\infty$ (this condition is well defined since $\rho$ is assumed Lipschitz; it implies $\left|d\rho(x)\right|$ is bounded away from zero over neighborhoods of compact subsets of $Y$).
\end{itemize}
The standard radius function for $\RR^n\times 0\subseteq\RR^n\times\RR^m$ is of course $(a_1,\ldots,a_n,b_1,\ldots,b_m)\mapsto(b_1^2+\cdots+b_m^2)^{1/2}$.
Every locally closed $C^p$ submanifold admits a $C^p$ radius function, as can be seen by choosing a collection of local coordinate patches and summing together standard radius functions via a partition of unity (a convex combination of radius functions is a radius function).

The following (trivial) restriction property for radius functions will be important: if $(U,\rho)$ is a radius function for $Y\subseteq M$, and $N\subseteq M$ is transverse to $Y$, then $(U\cap N,\rho|_{U\cap N})$ is a radius function for $Y\cap N\subseteq N$.

When $Y\subseteq M$ is relatively compact, a radius function $(U,\rho)$ for $Y$ will be called \emph{proper} when for every open set $V$ containing $\overline Y$, there exists $\epsilon>0$ such that $\rho^{-1}([0,\epsilon])\subseteq V$.
It is easy to produce proper radius functions: for any radius function $(U,\rho)$, there exists an open set $U'\subseteq U$ containing $Y$ such that $(U',\rho|_{U'})$ is proper.

The purpose of a radius function is to define tubular neighborhoods $\rho^{-1}([0,\epsilon])$.
The conditions in the definition of a radius function are chosen so as to be able to prove the following two key assertions:

\begin{lemma}\label{tubeconormal}
Let $(U,\rho)$ be a radius function for $Y\subseteq M$.
We have $d\rho(x)\to N^*Y$ as $x\to Y$, uniformly over compact subsets of $Y$ (equivalently, $N^*\rho^{-1}(r)$ approaches being contained in $N^*Y$ as $r\to 0$, uniformly over compact subsets of $Y$).
\end{lemma}

\begin{proof}
Suppose for sake of contradiction that there is a sequence $x_i\to y\in Y$ with $d\rho(x_i)\to\xi\in T^*_yM\setminus N^*_yY$.
Since $\xi\notin N^*_yY$, there exist $y_i\in Y$ with $\frac{d(x_i,y_i)}{d(x_i,Y)}$ uniformly bounded and the secant direction from $y_i$ to $x_i$ converging to the kernel of $\xi$.
This contradicts the final axiom of a radius function.
\end{proof}

\begin{lemma}\label{tubetransverse}
Let $Y\subseteq M$ be a stratum of a Whitney stratification $\SSS$, and let $(U,\rho)$ be a radius function for $Y$.
There exists an open set $V\subseteq U$ containing $Y$ such that the submanifold $\rho^{-1}(r)$ is transverse to $\SSS$ over $V$ for all $r>0$.
\end{lemma}

\begin{proof}
The secant line to $x\in\rho^{-1}(r)$ from nearby $y\in Y$ pairs positively with $d\rho(x)$, whereas Whitney (b) requires that this secant line approach the tangent space to the stratum containing $x$ in the limit $x\to y$.
Since $d\rho$ is bounded, this gives a positive lower bound on the restriction of $d\rho$ to any stratum of $\SSS$ in a neighborhood of any compact subset of $Y$.
\end{proof}

%%%%%%%%%%% IMPORTANT NOTE TO COPYEDITORS %%%%%%%%%%%%%
%%%%%%%%%%%%%%%%%%%%%%%%%%%%%%%%%%%%%%%%%%%%%%%%%%%%%%%
%%% If any of the figures needs to be resized, by far the simplest way to do this is
%%% to adjust the "scale" parameter in the \begin{tikzpicture} command.  That is,
%%% \begin{tikzpicture}[scale=\textwidth/6cm]
%%% should be replaced with
%%% \begin{tikzpicture}[scale=\textwidth/???cm]
%%% where ??? is whatever number you want to make the figure(s) the correct size.
%%% A *bigger* number means the figure will be *smaller*.

\begin{figure}[htbp]
\centering
\begin{tikzpicture}[scale=\textwidth/6cm]
\draw(0,0)to(0,1)to(1,1)to(1,0)to cycle;
\draw(0,0)to(1,1);
\begin{scope}[even odd rule]
\clip(1.15,-0.15)rectangle(3.15,1.15) (1.985,-0.015)to(1.985,1.015)to(2.015,1.015)to(2.015,-0.015)to cycle (2.985,-0.015)to(2.985,1.015)to(3.015,1.015)to(3.015,-0.015)to cycle (1.985,-0.015)to(1.985,.015)to(3.015,.015)to(3.015,-0.015)to cycle (1.985,0.985)to(1.985,1.015)to(3.015,1.015)to(3.015,0.985)to cycle (2-.010606,.010606)to(3-.010606,1.010606)to(3+.010606,1-.010606)to(2+.010606,-.010606)to cycle;
\draw(2,0) circle (0.1);
\draw(3,0) circle (0.1);
\draw(2,1) circle (0.1);
\draw(3,1) circle (0.1);
\end{scope}
\begin{scope}[even odd rule]
\clip(1.1,-0.1)rectangle(3.1,1.1) (2,0) circle (0.1) (3,0) circle (0.1) (2,1) circle (0.1) (3,1) circle (0.1);
\draw(1.985,-0.015)to(1.985,1.015)to(3.015,1.015)to(3.015,-0.015)to cycle;
\draw(2.015,0.015)to(2.015,0.985)to(2.985,0.985)to(2.985,0.015)to cycle;
\draw(2-.010606,.010606)to(3-.010606,1.010606);
\draw(2+.010606,-.010606)to(3+.010606,1-.010606);
\end{scope}
\end{tikzpicture}
\caption{A compact constructible set $X$ (left) and its outward cornering $X^{\underline\epsilon}$ (right).}\label{figureblowup}
\end{figure}

We now turn to the setting of a compact set $X\subseteq M$ constructible with respect to a chosen Whitney stratification $\SSS$.
Given a proper radius function for each stratum of $\SSS$ contained in $X$, we define the `outward cornering' of $X$ with respect to $\SSS$ (see Figure \ref{figureblowup}) to be
\begin{equation}\label{outwardcornering}
X^{\underline\epsilon}=\bigcup_{M_\alpha\subseteq X}\rho_\alpha^{-1}([0,\epsilon_\alpha])
\end{equation}
for $\underline\epsilon=(\epsilon_\alpha)_{M_\alpha\subseteq X}$, where it is tacitly required that $\epsilon_\alpha>0$ be sufficiently small as an unspecified function of $(\epsilon_\beta)_{M_\beta\subsetneqq\overline{M_\alpha}}$.
Note that, no matter our notion of sufficiently small, we can always find a parameterization $\underline\epsilon(\eta)$ where each $\epsilon_\alpha>0$ is an increasing function of $\eta>0$, limiting to zero as $\eta\to 0$, such that $\epsilon_\alpha(\eta)>0$ is sufficiently small in terms of $(\epsilon_\beta(\eta))_{M_\beta\subsetneqq\overline{M_\alpha}}$ for all $\eta>0$ (proof: by induction on strata).

The significance of properness of the given radius functions is that it (along with compactness of $X$) ensures that the part of $\partial X^{\underline\epsilon}$ coming from $\rho_\alpha^{-1}(\epsilon_\alpha)$ is contained in a neighborhood of a \emph{compact} subset of $M_\alpha\subseteq X$ (depending on $(\epsilon_\beta)_{M_\beta\subsetneqq\overline{M_\alpha}}$), hence falls within the scope of Lemma \ref{tubetransverse}.

Here is the first key property of $X^{\underline\epsilon}$:

\begin{corollary}\label{neighborhoodcorners}
Fix a Whitney stratification $\SSS$ of $M$.
Let $X\subseteq M$ be a compact constructible subset, with a choice of proper radius function for each of its strata.
Then $M\setminus X^{\underline\epsilon}$ is a manifold-with-corners, all of whose corner strata are transverse to $\SSS$.
\end{corollary}

\begin{proof}
We proceed by induction on the number of strata of $X$.
Thus let $X=X_0\cup M_\alpha$, and suppose the result is known for $X_0$.
This means, in particular, that $M\setminus X_0^{\underline\epsilon}$ is a manifold-with-corners, all of whose corner strata are transverse to $M_\alpha$.
It follows that further removing the small regular neighborhood $\rho_\alpha^{-1}([0,\epsilon_\alpha])$ produces a manifold-with-corners $M\setminus X^{\underline\epsilon}$, provided $\epsilon_\alpha>0$ is sufficiently small as a function of $(\epsilon_\beta)_{M_\beta\subsetneqq\overline{M_\alpha}}$.

Now let us show that the corner strata of $M\setminus X^{\underline\epsilon}$ are transverse to $\SSS$.
The boundary stratum of $M\setminus X^{\underline\epsilon}$ coming from $\rho_\alpha^{-1}(\epsilon_\alpha)$ is transverse to $\SSS$ by Lemma \ref{tubetransverse}.
A general corner stratum of $M\setminus X^{\underline\epsilon}$ is the intersection $C\cap\rho_\alpha^{-1}(\epsilon_\alpha)$ where $C$ is a corner stratum of $M\setminus X_0^{\underline\epsilon}$.
To see that such an intersection is also transverse to $\SSS$, apply Lemma \ref{tubetransverse} to the stratum $C\cap M_\alpha\subseteq C$ of the restriction to $C$ of $\SSS$ (which is Whitney by Lemma \ref{whitneyintersect}, locally uniformly in $\underline\epsilon\setminus\epsilon_\alpha$) equipped with the restriction of $\rho_\alpha$ (which remains a radius function as noted earlier).
\end{proof}

Note that the conclusion of Corollary \ref{neighborhoodcorners} (transversality of $\partial X^{\underline\epsilon}$ and $\SSS$) is equivalent to saying that $N^*\SSS$ (the union of the conormals of all strata) and $N^*X^{\underline\epsilon}$ (the union of the conormals of all corner strata) are disjoint at infinity.

Given transversality of $\partial X^{\underline\epsilon}$ and $\SSS$, we can define the `big conormal' $N^*(\SSS|(M\setminus X^{\underline\epsilon}))$ to be the union of conormals of intersections of strata of $\SSS$ and corner strata of $M\setminus X^{\underline\epsilon}$.
The second key property of $X^{\underline\epsilon}$ is the following convergence result:

\begin{corollary}\label{cornerconormalconvergence}
Fix a Whitney stratification $\SSS$ of $M$.
Let $X\subseteq M$ be a compact constructible subset, with a choice of proper radius function for each of its strata.
Then the big conormal $N^*(\SSS|(M\setminus X^{\underline\epsilon}))$ lies in arbitrarily small neighborhoods of $N^*\SSS$ as $\underline\epsilon\to 0$.
\end{corollary}

\begin{proof}
We proceed by induction on the number of strata of $X$.
Thus let $X=X_0\cup M_\alpha$, and suppose the result is known for $X_0$.
It thus suffices to show that $N^*(\SSS|(M\setminus X^{\underline\epsilon}))$ lies in an arbitrarily small neighborhood of $N^*(\SSS|(M\setminus X_0^{\underline\epsilon}))$ as $\epsilon_\alpha\to 0$ (uniformly over compact subsets of $\underline\epsilon\setminus\epsilon_\alpha$).

The conormal of the intersection of a stratum $M_\gamma$ of $\SSS$ with $\rho_\alpha^{-1}(\epsilon_\alpha)$ is the sum of the conormals of $M_\gamma$ and $\rho_\alpha^{-1}(\epsilon_\alpha)$.
These individually approach being contained in $N^*M_\alpha$ in the limit $\epsilon_\alpha\to 0$ (the first by Whitney (a) and the second by Lemma \ref{tubeconormal}), and they are quantitatively transverse by the axioms of a radius function (as was the main point of the proof of Lemma \ref{tubetransverse}).
It follows that their sum also approaches $N^*M_\alpha$ as $\epsilon_\alpha\to 0$.

The general case is that of the conormal of $M_\gamma\cap\rho_\alpha^{-1}(\epsilon_\alpha)\cap C$ where $C$ is a corner stratum of $M\setminus X_0^{\underline\epsilon}$.
It follows by applying the same argument to the intersection of the situation with $C$, as in the proof of Corollary \ref{neighborhoodcorners}.
\end{proof}

\begin{remark}
In the context of Corollary \ref{cornerconormalconvergence}, note that 
any subanalytic family of Legendrians inside $S^*M$, whose projections converge to $X$, 
will themselves converge to a subset of the conormal of $X$ \emph{with respect to some refinement of $\SSS$}.
Corollary \ref{cornerconormalconvergence} provides a stronger convergence result (we do not need to refine the stratification) for the particular family of outward cornerings defined in \eqref{outwardcornering}.
\end{remark}

\subsection{Cornering and conormals of constructible open sets}\label{constructibleconormals}

Here we develop some finer properties of the microlocal approximations constructed in the previous subsection.
They will not be used until Section \ref{fukayacalculationssec}.

Let $\SSS$ be a Whitney stratification of $M$ by locally closed smooth submanifolds.
For any $\SSS$-constructible relatively compact open set $U\subseteq M$, we define its \emph{inward cornering}
\begin{equation}
U^{-\underline\epsilon}:=U\setminus\overline{(\partial U)^{\underline\epsilon}}.
\end{equation}
where $(\partial U)^{\underline\epsilon}$ denotes the outward cornering $\partial U$ defined in \eqref{outwardcornering}.
Thus $U^{-\underline\epsilon}$ is (the interior of) a codimension zero submanifold-with-corners (Corollary \ref{neighborhoodcorners}), depending smoothly on $\underline\epsilon$, such that as $\underline\epsilon\to 0$, its conormal $N^*U^{-\underline\epsilon}$ remains disjoint from $N^*\SSS$ at infinity (Corollary \ref{neighborhoodcorners}) yet limits inside it (Corollary \ref{cornerconormalconvergence}).
Strictly speaking, $U^{-\underline\epsilon}$ also depends on the choices of tubular neighborhoods of the strata comprising $\partial U$, however we will leave this choice out of the notation as it is never particularly relevant (note that it is a convex, hence contractible, choice).
When even the choice of $\underline\epsilon$ is not relevant, we will simply write $U^-$.

Taking $\underline\epsilon\to 0$, we learn that:

\begin{lemma}\label{perturbationdiffeo}
$U$ and $U^-$ are diffeomorphic, and the diffeomorphism may be chosen to be the identity on any fixed compact subset of $U$.
\qed
\end{lemma}

When $\SSS$ is a triangulation, we may consider for any simplex $s\in\SSS$ its open star, $\star(s)$.

\begin{lemma}\label{bdrywarmup}
For $\SSS$ a triangulation and any $s \in \SSS$, there is a homotopy equivalence $\partial\star(s)^-\simeq\partial\star(s)$.
\end{lemma}

\begin{proof}
Both $\star(s)$ and $\star(s)^-$ are the interiors of compact (topological) manifolds-with-boundary (namely their closures).
They are also diffeomorphic: $\star(s)\cong\star(s)^-$ by Lemma \ref{perturbationdiffeo}.
It therefore suffices to recall the standard fact that the interior of a compact manifold-with-boundary remembers its boundary, up to homotopy equivalence.

Indeed, let $\overline M$ be a compact manifold-with-boundary and let $M=\overline M\setminus\partial\overline M$ denote its interior.
The `end space' $e(M)$ is the space of proper maps $\RR_{\geq 0}\to M$ (this is a model for the homotopy inverse limit of $M\setminus K$ over compact subsets $K\subseteq M$).
A choice of collar $\partial\overline M\times[0,1)\hookrightarrow\overline M$ determines a homotopy equivalence
\begin{equation}
e(M)\xleftarrow\sim e(\partial\overline M\times(0,{\textstyle\frac 12}])=C(\RR_{\geq 0},\partial\overline M)\times e((0,{\textstyle\frac 12}]),
\end{equation}
and we have homotopy equivalences $C(\RR_{\geq 0},\partial\overline M)\simeq\partial\overline M$ and $e((0,\frac 12])\simeq *$, so we have $e(M)\simeq\partial\overline M$.
(Compare \cite[\S 1]{hughesranicki}.)
\end{proof}

\begin{lemma}\label{strangesimplexintersectionisstableball}
For $\SSS$ a triangulation and for simplices $s,t\in\SSS$ with $\star(t)\cap\star(s)\ne\emptyset$ and $t\nrightarrow s$, the intersection $\star(t)^{-\underline\epsilon}\cap\partial\star(s)^{-\underline\delta}$ is contractible for $\underline\epsilon$ sufficiently small and $\underline\delta$ sufficiently small in terms of $\underline\epsilon$.
\end{lemma}

\begin{proof}
As in the proof of Lemma \ref{bdrywarmup}, the intersection ${\star(t)^{-\underline\epsilon}}\cap{\partial\star(s)^{-\underline\delta}}$ is the end space of $\overline{\star(t)^{-\underline\epsilon}}\cap\star(s)^{-\underline\delta}$.
As in Lemma \ref{perturbationdiffeo}, this intersection $\overline{\star(t)^{-\underline\epsilon}}\cap\star(s)^{-\underline\delta}$ is diffeomorphic to $\overline{\star(t)^{-\underline\epsilon}}\cap\star(s)$, whose end space is in turn given by $\star(t)^{-\underline\epsilon}\cap\partial\star(s)$.
Now we may take $\underline\epsilon\to 0$ again mimicking the proof of Lemma \ref{perturbationdiffeo} to see that this is homotopy equivalent to $\star(t)\cap\partial\star(s)$.
The assumptions on $s$ and $t$ now imply that this is contractible.
\end{proof}

\begin{lemma}\label{smoothcornering}
If $U$ has smooth boundary, then there is a $C^0$-small isotopy between $U^-$ and a small inward pushoff of $U$.
\end{lemma}

\begin{proof}
The definition of $U^-$ depends on choice of radius functions for the strata comprising $\partial U$.
Fix coordinates $\partial U\times\RR\subseteq M$ near $\partial U$, and choose radius functions whose inward derivative in the $\RR$-direction is positive inside $U$.
Such radius functions exist locally, hence can be patched together using a partition of unity pulled back from $\partial U$ (i.e.\ independent of the $\RR$-coordinate), which preserves the property of having positive inward derivative inside $U$.
Now using these radius functions, each vertical line $p\times\RR$ for $p\in\partial U$ intersects $\partial U^-$ exactly once, transversally, which provides the desired isotopy.
\end{proof}

\begin{remark}
It is \emph{not} asserted that the isotopy in Lemma \ref{smoothcornering} will ensure that the conormal remains disjoint from $N^*\SSS$ at infinity.
This will require us to exercise some care when applying it.
\end{remark}

For the next result, consider a Whitney stratification $\SSS$ and a point $q\in M$ lying in a stratum $M_\chi$.
Let $\SSS_q$ denote the Whitney stratification obtained from $\SSS$ by replacing $M_\chi$ with $M_\chi\setminus q$ and $\{q\}$.
Given a compact $\SSS$-constructible set $X\subseteq M$ containing $q$, we can consider its outward cornerings $X^{\underline\epsilon}$ and $X^{\underline\epsilon,\delta}$ with respect to $\SSS$ and $\SSS_q$, respectively, where $\underline\epsilon=(\epsilon_\alpha)_{M_\alpha\subseteq X}$ and $\delta>0$ is associated to $q$.
Evidently
\begin{equation}
X^{\underline\epsilon,\delta}=X^{\underline\epsilon}\cup B_\delta(q),
\end{equation}
and according to the definition of `outward cornering' above, there is the implicit requirement that $\epsilon_\chi>0$ be sufficiently small as a function of $(\epsilon_\beta)_{M_\beta\subsetneqq\overline{M_\alpha}}$ and $\delta$.
The next result concerns the behavior of $X^{\underline\epsilon,\delta}$ when we \emph{remove} the dependence of $\epsilon_\chi$ on $\delta$.
The resulting neighborhoods are illustrated in Figure \ref{figureXepsdel}, which should be contrasted with Figure \ref{figureblowup}.

\begin{proposition}\label{wrapnothitballnew}
Let $X\subseteq M$ be a compact $\SSS$-constructible subset, and fix $q\in X$ living in stratum $\chi$.
Let $\tilde X^{\underline\epsilon,\delta}=X^{\underline\epsilon}\tildecup B_\delta(q)$ where $X^{\underline\epsilon}$ is the outward cornering \eqref{outwardcornering} and the notation $\tilde\cup$ indicates that the boundary of the union $\rho_\chi^{-1}([0,\epsilon_\chi])\cup B_\delta(q)$ is smoothed along its potential corner locus $\rho_\chi^{-1}(\epsilon_\chi)\cup\partial B_\delta(q)$, and then the remaining tubes are added.
Then for suitable choices of radius functions near $q$, these modified outward cornerings $X^{\underline\epsilon,\delta}$ satisfy the conclusion of Corollary \ref{neighborhoodcorners}.
\end{proposition}

%%%%%%%%%%% IMPORTANT NOTE TO COPYEDITORS %%%%%%%%%%%%%
%%%%%%%%%%%%%%%%%%%%%%%%%%%%%%%%%%%%%%%%%%%%%%%%%%%%%%%
%%% If any of the figures needs to be resized, by far the simplest way to do this is
%%% to adjust the "scale" parameter in the \begin{tikzpicture} command.  That is,
%%% \begin{tikzpicture}[scale=\textwidth/10cm]
%%% should be replaced with
%%% \begin{tikzpicture}[scale=\textwidth/???cm]
%%% where ??? is whatever number you want to make the figure(s) the correct size.
%%% A *bigger* number means the figure will be *smaller*.

\usetikzlibrary{intersections}
\begin{figure}[htbp]
\centering
\begin{tikzpicture}[scale=\textwidth/10cm]
\path[name path = LUhoriz](1,0.2)--(-1,0.2);
\path[name path = LUcircleright](0.55,0)arc(0:90:0.55);
\path[name path = LUcircleleft]((-0.55,0)arc(180:90:0.55);
\path[name intersections={of = LUhoriz and LUcircleright, by=a}];
\path[name intersections={of = LUhoriz and LUcircleleft, by=b}];
\draw(-1,0)to(1,0);
\draw(a)--(1,0.2);
\draw(b)--(-1,0.2);
\begin{scope}
\clip(-1,0.2)--(1,0.2)--(1,1)--(-1,1)--cycle;
\draw(0,0)circle(0.55);
\end{scope}
\draw[fill](0,0)circle(0.03);
\node at(0.1,-0.15){$q$};
\end{tikzpicture}
\quad
\begin{tikzpicture}[scale=\textwidth/10cm]
\path[name path = LUhoriz](1,0.2)--(-1,0.2);
\path[name path = LUcircleright](0.4,0)arc(0:90:0.4);
\path[name path = LUcircleleft]((-0.4,0)arc(180:90:0.4);
\path[name intersections={of = LUhoriz and LUcircleright, by=a}];
\path[name intersections={of = LUhoriz and LUcircleleft, by=b}];
\draw(-1,0)to(1,0);
\draw(a)--(1,0.2);
\draw(b)--(-1,0.2);
\begin{scope}
\clip(-1,0.2)--(1,0.2)--(1,1)--(-1,1)--cycle;
\draw(0,0)circle(0.4);
\end{scope}
\draw[fill](0,0)circle(0.03);
\node at(0.1,-0.15){$q$};
\end{tikzpicture}
\quad
\begin{tikzpicture}[scale=\textwidth/10cm]
\path[name path = LUhoriz](1,0.2)--(-1,0.2);
\path[name path = LUcircleright](0.25,0)arc(0:90:0.25);
\path[name path = LUcircleleft]((-0.25,0)arc(180:90:0.25);
\path[name intersections={of = LUhoriz and LUcircleright, by=a}];
\path[name intersections={of = LUhoriz and LUcircleleft, by=b}];
\draw(-1,0)to(1,0);
\draw(a)--(1,0.2);
\draw(b)--(-1,0.2);
\begin{scope}
\clip(-1,0.2)--(1,0.2)--(1,1)--(-1,1)--cycle;
\draw(0,0)circle(0.25);
\end{scope}
\draw[fill](0,0)circle(0.03);
\node at(0.1,-0.15){$q$};
\end{tikzpicture}
\quad
\begin{tikzpicture}[scale=\textwidth/10cm]
\draw(-1,0)to(1,0);
\draw(1,0.2)--(-1,0.2);
\draw[fill](0,0)circle(0.03);
\node at(0.1,-0.15){$q$};
\end{tikzpicture}
\caption{The neighborhoods $X^{\underline\epsilon,\delta}$.}\label{figureXepsdel}
\end{figure}

\begin{proof}
Choose a local Euclidean chart near $q$ in which $q\in M_\chi\subseteq M$ is locally modelled on $0\in\RR^k\subseteq\RR^n$, and let us choose the usual radius (i.e.\ distance) functions in this chart.
Note that we thus have a completely explicit picture of how these tubes intersect near $q$; in particular, their union with smoothed boundary is well behaved.
We regard the union of tubes (with smoothed boundary) $\rho_\chi^{-1}([0,\epsilon_\chi])\tildecup B_\delta(q)$ as a single object associated to the stratum $\chi$, albeit depending on two parameters $\epsilon_\chi$ and $\delta$.
This single object satisfies the conclusion of Lemma \ref{tubetransverse} by Whitney (b), which is all that is used in the inductive proof of Corollary \ref{neighborhoodcorners}.
\end{proof}

\section{Microlocal Morse categories} \label{sec:axioms}

\subsection{Strata poset categories and refinement functors}\label{posetcategories}

Let $\SSS$ be a stratification.
We fix the following notation for the Yoneda embedding:
\begin{align}
\SSS &\to \Fun(\SSS^\op,\Set),\\
\alpha&\mapsto\Hom(\cdot,\alpha)=:1_{\star(\alpha)}.
\end{align}
Note that
\begin{equation}
\Hom(1_{\star(\alpha)},1_{\star(\beta)}) = 1_{\star(\beta)}(\alpha) = \Hom(\alpha, \beta) =  \begin{cases} \{1\} & \star(\alpha) \subseteq \star(\beta)  \\ \emptyset & \textrm{otherwise.} \end{cases}
\end{equation}
For any $\SSS$-constructible open set $U$, we introduce the functor $1_U\in\Fun(\SSS^\op,\Set)$ defined by the analogous formula
\begin{equation}\label{indicatorUdef}
\Hom(1_{\star(\alpha)}, 1_{U}) = 1_{U}(\alpha) :=  \begin{cases} \{1\} & \star(\alpha) \subseteq U  \\ \emptyset & \textrm{otherwise.} \end{cases}
\end{equation}
(The action of $1_U$ on morphism sets is in fact uniquely determined by the above, since when $\alpha \to \beta$, the set $\Hom_{\Set}(1_U(\beta), 1_U(\alpha))$ always consists of one element.)
Note that $\star(\alpha) \subseteq U$ iff $\alpha\subseteq U$.
More generally, for $U \subseteq V$ there is a unique natural transformation
\begin{equation}\label{uniquenattrans}
1_{U} \to 1_{V}.
\end{equation}
It sends $1 \in 1_{U}(\alpha)$ to $1 \in 1_V(\alpha)$ for any $\star(\alpha) \subseteq U\subseteq V$.

Now let $\SSS'$ be a stratification refining $\SSS$.  There is a natural map $r:\SSS'\to\SSS$, sending a stratum in $\SSS'$ to the 
unique stratum in $\SSS$ containing it.
We write
\begin{equation}
r^*:\Fun(\SSS^\op,\Set)\to\Fun(\SSS^{\prime\op},\Set)
\end{equation}
for the pullback of functors along this map $r$.
For $\tau' \in \SSS'$ and an $\SSS$-constructible open set $U$, we have 
\begin{equation}
\Hom(1_{\star(\tau')}, r^*1_U) = (r^*1_U)(\tau') = 1_U(r(\tau')) =
\begin{cases} \{1\} & \star(r(\tau')) \subseteq U \\ \emptyset & \textrm{otherwise.} \end{cases}
\end{equation}
Since $U$ is open and $\SSS$-constructible, we have $\star(r(\tau')) \subseteq U$ iff $\star(\tau')\subseteq U$, so we conclude that $r^*1_U=1_U$.

We now linearize.   We write $\ZZ[\SSS]$ for the linearization of a poset $\SSS$.
We write $\Mod\SSS$ for the category of modules $\Fun(\SSS^\op,\Mod\ZZ) = \Fun(\ZZ[\SSS]^\op, \Mod \ZZ)$, and we use $r^*:\Mod\SSS\to\Mod\SSS'$ for pullback of modules as above.
As with any pullback of modules, this functor has a left adjoint, typically termed extension of scalars or induction,
which we write as $r_!:\Mod\SSS'\to\Mod\SSS$, which fits into the commuting diagram
\begin{equation}
\begin{tikzcd}[column sep = large]
\SSS'\ar{r}{s\mapsto 1_{\star(s)}}\ar{d}[swap]{r}& \Mod\SSS'\ar{d}{r_!}\\
\SSS\ar{r}{s\mapsto 1_{\star(s)}}& \Mod\SSS.
\end{tikzcd}
\end{equation}
Restriction of scalars $r^*$ is co-continuous, so its left adjoint $r_!$ extension of scalars preserves compact objects, giving a map $r_!:\Perf\SSS'\to\Perf\SSS$ (which can also be viewed as the canonical extension of $r:\SSS'\to\SSS$ to the idempotent-completed pre-triangulated hulls).

\subsection{A category for any \texorpdfstring{$\Lambda$}{Lambda}}

We now wish to define a \emph{microlocal Morse category} $\C(\Lambda)$ for any subanalytic (possibly) singular isotropic $\Lambda\subseteq S^*M$, together with functors $\C(\Lambda')\to\C(\Lambda)$ for inclusions $\Lambda'\supseteq\Lambda$.
We define this system of categories $\Lambda\mapsto\C(\Lambda)$, the \emph{microlocal Morse theatre}, by formulating axioms which characterize it uniquely.
(Recall that $S^*M:=(T^*M\setminus M)/\RR_{>0}$ denotes the co-sphere bundle of $M$, and a closed subanalytic set $\Lambda\subseteq S^*M$ is called isotropic iff for some, or, equivalently, every, cover of $\Lambda$ by locally closed $C^1$ submanifolds, all of them are isotropic.)

The previous subsection defined categories $\Perf\SSS$ together with functors $r_!:\Perf\SSS'\to\Perf\SSS$ whenever $\SSS'$ is a refinement of $\SSS$.
For our current purpose, these categories do not have the correct significance for general stratifications $\SSS$ (compare Remark \ref{stratvsposet}).
As such, we will consider these categories only for triangulations $\SSS$.\footnote{In fact, there are weaker conditions on a stratification $\SSS$ (which are satisfied if $\SSS$ is a triangulation) implying that $\Perf\SSS$ is the correct category to associate to $\SSS$.}
The microlocal Morse theatre is an extension of this functor $\SSS\mapsto\Perf\SSS$ on triangulations.

\begin{definition}\label{categorysystemextension}
A \emph{microlocal Morse pre-theatre} $\Lambda\mapsto\C(\Lambda)$ is a functor from the category of subanalytic singular isotropics inside $S^*M$ to the category of dg categories over $\ZZ$.
A \emph{normalized} microlocal Morse pre-theatre is one equipped with an isomorphism of functors $(\SSS\mapsto\C(N^*_\infty\SSS))=(\SSS\mapsto\Perf\SSS)$ on Whitney triangulations $\SSS$.
\end{definition}

\begin{remark}
Any isomorphism of functors $(\SSS\mapsto H^*\C(N^*_\infty\SSS))=(\SSS\mapsto H^*\Perf\SSS)$ automatically lifts to an isomorphism $(\SSS\mapsto\C(N^*_\infty\SSS))=(\SSS\mapsto\Perf\SSS)$ by Proposition \ref{ffhenceainf}.
This will be crucial when discussing Fukaya categories (specifically Theorem \ref{Frefinement}).
\end{remark}

We will characterize the microlocal Morse theatre in terms of microlocal Morse theory.\footnote{More conventionally \cite{goresky-macpherson}, 
this is called stratified Morse theory.  We find the term `microlocal' more descriptive, and also the word stratified would otherwise
take on too many meanings in this article.}
Let $f:M\to\RR$ be a function and $\SSS$ a stratification.
An intersection of $\Gamma_{df}$ with $N^*\SSS$ is called an $\SSS$-critical point, which is said to be Morse if it is a transverse intersection at a smooth point of $N^*\SSS$.
The function $f$ is said to be $\SSS$-Morse when all its $\SSS$-critical points are Morse.
When $\SSS$ is subanalytic, such functions are plentiful, and can be chosen analytic.
(See \cite[Thm.\ 2.2.1]{goresky-macpherson} for this assertion, which is collected there from various results in the literature.)

More generally, for any singular isotropic $\Lambda\subseteq S^*M$, a $\Lambda$-critical point of $f$ is by definition an intersection of $\Gamma_{df}$ with the union of the zero section and $\RR_{>0}\times\Lambda$.   That is, 
\begin{equation}
\crit_\Lambda(f) = \Gamma_{df}\cap(0_M\cup(\RR_{>0}\times\Lambda))\subseteq T^*M
\end{equation}
Such a $\Lambda$-critical point is said to be Morse if the intersection is transverse and occurs at a smooth point of $0_M\cup(\RR_{>0}\times\Lambda)$, and any $f$ whose $\Lambda$-critical points are all Morse is called $\Lambda$-Morse.

\begin{definition} \label{morsecharacter}
In any normalized microlocal Morse pre-theatre $\Lambda\mapsto\C(\Lambda)$, the \emph{Morse characters} $\X_{\Lambda,p}(f,\epsilon,\SSS)\in\C(\Lambda)$ are defined as follows for smooth Legendrian points $p\in\Lambda$.

Let $f:M\to\RR$ be an analytic function with a Morse $\Lambda$-critical point at $p$ (i.e.\ somewhere in $\RR_{>0}\times\{p\}\subseteq T^*M$)
with critical value $0$, no other $\Lambda$-critical points with critical values in the interval $[-\epsilon,\epsilon]$, and with relatively compact sublevel set $f^{-1}(-\infty,\epsilon)$.
Let $\SSS$ be a Whitney triangulation for which $\Lambda\subseteq N^*_\infty\SSS$ and for which both $f^{-1}(-\infty, -\epsilon)$ and $f^{-1}(-\infty, \epsilon)$ are $\SSS$-constructible.

The Morse character $\X_{\Lambda,p}(f,\epsilon,\SSS)$ is then defined as the image of
\begin{equation}
\cone(1_{f^{-1}(-\infty, -\epsilon)} \to 1_{f^{-1}(-\infty, \epsilon)}) \in \Perf\SSS=\C(N^*_\infty\SSS).
\end{equation}
under the map $\C(N^*_\infty\SSS)\to\C(\Lambda)$, where $1_{f^{-1}(-\infty, -\epsilon)} \to 1_{f^{-1}(-\infty, \epsilon)}$ is (the linearization of) the unique map \eqref{uniquenattrans}.
\end{definition}

The Morse character $\X_{\Lambda,p}(f,\epsilon,\SSS)\in\C(\Lambda)$ depends \emph{a priori} on the `casting directors' $(f,\epsilon,\SSS)$.
Casting directors $(f,\epsilon)$ exist at any smooth Legendrian point $p\in\Lambda$ by general position, and $\SSS$ exists by the following argument. First, by \cite[Prop.\ 8.3.10]{kashiwara-schapira} every closed subanalytic singular isotropic $\Lambda\subseteq S^*M$ is contained in $N^*_\infty\SSS$ for some subanalytic stratification $\SSS$ of $M$.
Next, by refining $\SSS$ the subanalytic subsets $f^{-1}(-\infty, \pm \epsilon)$ can be made constructible \cite{bierstone-milman,shiotabook}.
Finally, $\SSS$ can be made a subanalytic Whitney triangulation by \cite{shiotawhitney,czapla,czaplapawlucki}.

\begin{definition}\label{microlocalmorsetheatre}
A \emph{microlocal Morse theatre} is a normalized microlocal Morse pre-theatre $\Lambda\mapsto\C(\Lambda)$ satisfying the \emph{localization property}: for any inclusion $\Lambda\subseteq\Lambda'$ and any collection of Morse characters $\X_{\Lambda',p}(f,\epsilon,\SSS)\in\C(\Lambda')$ at smooth Legendrian points $p\in\Lambda'\setminus\Lambda$ with at least one in every component of the smooth Legendrian locus of $\Lambda'\setminus\Lambda$, the functor $\C(\Lambda')\to\C(\Lambda)$ is the idempotent-completed quotient by these Morse characters.
\end{definition}

The definition of a microlocal Morse theatre allows one to readily compute any particular microlocal Morse category $\C(\Lambda)$: embed $\Lambda$ into some $N^*_\infty\SSS$, cast Morse characters in $\C(N^*_\infty\SSS)=\Perf\SSS$ for all Legendrian components of $N^*_\infty\SSS\setminus\Lambda$, and take the quotient of $\Perf\SSS$ by these characters and idempotent complete.
It follows that:

\begin{proposition}\label{microlocalmorseuniqueness}
Any two microlocal Morse theatres $\Lambda\mapsto\C(\Lambda)$ are uniquely isomorphic.
\end{proposition}

\begin{proof}
For any normalized microlocal Morse pre-theatre $\C$, let $\X_{\Lambda'\setminus\Lambda}\subseteq\C(\Lambda')$ denote the collection of all Morse characters at all smooth Legendrian points of $\Lambda'\setminus\Lambda$.
Now for any microlocal Morse theatre $\C$, we have a canonical quasi-equivalence (functorial in $\Lambda$):
\begin{equation}
\varinjlim_{\begin{smallmatrix} N^*\SSS \supseteq \Lambda\\\SSS\textrm{ Whitney triangulation}\end{smallmatrix}}(\C(N^*\SSS)/\X_{N^*\SSS\setminus\Lambda})^\pi\xrightarrow\sim\C(\Lambda),
\end{equation}
and the left hand side is independent of $\C$ since $\C$ is normalized.
\end{proof}

A \emph{dramatic realization} is a particular construction of the microlocal Morse theatre $\Lambda\mapsto\C(\Lambda)$.
We give two dramatic realizations, namely via sheaves and via Lagrangians in Sections \ref{sec:sheaves} and \ref{sec:fukaya}, respectively.
Both these dramatic realizations \emph{cast} the Morse characters as certain familiar objects.
They moreover show that the Morse characters in fact depend only on $p$ (up to shifts) and are independent of the casting directors.

\begin{theorem}\label{microlocalmorseexistence}
The microlocal Morse theatre $\Lambda\mapsto\C(\Lambda)$ exists, and the Morse characters $\X_{\Lambda,p}\in\C(\Lambda)$ are independent of the casting directors and form a local system over the smooth Legendrian locus of $\Lambda$.
\end{theorem}

\begin{proof}
This follows from either Theorem \ref{sheafrealization} or Theorem \ref{fukcasting}.
\end{proof}

\begin{proof}[Proof of Theorem \ref{sheaffukayaequivalence}]
Combine Theorem \ref{sheafrealization} and Theorem \ref{fukcasting} with Proposition \ref{microlocalmorseuniqueness}.
\end{proof}

In fact, both dramatic realizations show that $\C(\Lambda)$ is invariant under contact isotopy of $S^*M$, something which is not apparent from the present combinatorial prescription.
This is immediate on the Fukaya side, and on the sheaf side it is `sheaf quantization' \cite{guillermou-kashiwara-schapira}.
In fact, there are even stronger invariance statements: it is shown in \cite{gpsdescent} that in fact $\C(\Lambda)$ is invariant under contact isotopy of $S^*M \setminus \Lambda$ inside $S^*M$;
meanwhile, it is shown in \cite{nadler-shende} that $\C(\Lambda)$ is invariant under `gapped' deformations of $\Lambda$.

\begin{remark}
The construction of this subsection makes sense in any stable setting, e.g.\ over the sphere spectrum.
To show existence of the microlocal Morse theatre in such a more general setting, one could set up either microlocal sheaf theory or the Fukaya category over the sphere spectrum.
In principle, one could also show existence directly from the stratified Morse theory of \cite{goresky-macpherson}, as it already establishes results about homotopy types of spaces (not just their cohomologies).
A more interesting question is whether any symplectically invariant statement can be made beyond the stable setting.
\end{remark}

\subsection{Open inclusions}\label{morsemorphism}

We now discuss functoriality of the microlocal Morse theatre under open inclusions.

Given any analytic open inclusion of analytic manifolds $M'\hookrightarrow M$, a microlocal Morse pre-theatre $\C'$ on $M'$ determines a microlocal Morse pre-theatre $\Lambda\mapsto\C'(\Lambda\cap S^*M')$ on $M$; let us call this the extension of $\C'$ to $M$.
For microlocal Morse pre-theatres $\C'$ on $M'$ and $\C$ on $M$, a morphism $\C'\to\C$ means a morphism to $\C$ from the extension of $\C'$ to $M$.
In other words, such a morphism consists of a coherent system of maps $\C'(\Lambda\cap S^*M')\to\C(\Lambda)$ for subanalytic singular isotropics $\Lambda\subseteq S^*M$.
Equivalently, we may view such a morphism as a coherent system of maps
\begin{equation}
\C'(\Lambda')\to\C(\Lambda)
\end{equation}
for subanalytic singular isotropics $\Lambda'\subseteq S^*M'$ and $\Lambda\subseteq S^*M$ with $\Lambda\cap S^*M'\subseteq\Lambda'$.
If $\C$ and $\C'$ are normalized, then restricting these maps to the case that $\Lambda$ and $\Lambda'$ are both conormals of Whitney triangulations, we obtain a coherent collection of maps
\begin{equation}\label{correctonperf}
\Perf\SSS'=\C'(N^*_\infty\SSS')\to\C(N^*_\infty\SSS)=\Perf\SSS
\end{equation}
for every pair of Whitney triangulations $\SSS$ of $M$ and $\SSS'$ of $M'$ such that $\SSS'$ refines $\SSS\cap M'$.
A normalized morphism $\C'\to\C$ is one equipped with an isomorphism between \eqref{correctonperf} and the extension to $\Perf$ of the tautological maps of posets $\SSS'\to\SSS$ (sending a stratum of $\SSS'$ to the unique stratum of $\SSS$ containing it).
A morphism of microlocal Morse theatres $\C'\to\C$ is, by definition, a normalized morphism of underlying normalized microlocal Morse pre-theatres.

We now have the following refinement of Proposition \ref{microlocalmorseuniqueness}:

\begin{proposition}\label{morphismunique}
For an analytic open inclusion of analytic manifolds $M'\hookrightarrow M$ and microlocal Morse theatres $\C'$ on $M'$ and $\C$ on $M$, there exists a unique morphism of microlocal Morse theatres $\C'\to\C$.
\end{proposition}

\begin{proof}
It follows from \eqref{correctonperf} that the maps $\C'(\Lambda')\to\C(\Lambda)$ send Morse characters to Morse characters (choose casting directors on $M$ which are `supported inside $M'$', and appeal to the independence of Morse characters of the casting directors from Theorem \ref{microlocalmorseexistence}).
Now we have the following commutative diagram, functorial in $\Lambda$ and $\Lambda'$:
\begin{equation}
\begin{tikzcd}
\varinjlim{}(\C(N^*_\infty\SSS')/\X_{N^*_\infty\SSS'\setminus\Lambda'})^\pi\ar{d}{\sim}\ar{r}&\varinjlim{}(\C(N^*_\infty\SSS)/\X_{N^*_\infty\SSS\setminus\Lambda})^\pi\ar{d}{\sim}\\
\C'(\Lambda')\ar{r}&\C(\Lambda)
\end{tikzcd}
\end{equation}
where both direct limits take place over pairs of Whitney triangulations $\SSS$ and $\SSS'$ for which $\SSS'$ refines $M'\cap\SSS$ and for which $\Lambda\subseteq N^*_\infty\SSS$ and $\Lambda'\subseteq N^*_\infty\SSS'$.
Hence the maps $\C'(\Lambda')\to\C(\Lambda)$ are determined uniquely.
\end{proof}

\begin{proof}[Proof of Proposition \ref{sfecompatibility}]
Combine Propositions \ref{sheaffunctorial} and \ref{wfunctorial} with Proposition \ref{morphismunique}.
\end{proof}

\section{Sheaf categories} \label{sec:sheaves}

We recall the general formalism of sheaves, and properties of stratifications.  We then recall from \cite{kashiwara-schapira} 
the notion of microsupport, and the category $\Sh_\Lambda(M)$ of sheaves on $M$ whose microsupport at infinity is contained in $\Lambda$. 
We show that the assignment $\Lambda \mapsto \Sh_\Lambda(M)$ is a microlocal Morse theatre in the sense of Definition \ref{microlocalmorsetheatre}.

\subsection{Categories of sheaves and functors between them}

Here we give a brief review of the general formalism of sheaves.  Our presentation is somewhat modern in that we \emph{never} discuss 
sheaves of abelian groups, rather we work at the dg level and with unbounded complexes from the beginning, but it is essentially the same as 
any standard account such as \cite{iversen, kashiwara-schapira, schapira-notes}, complemented by \cite{spaltenstein} in order to work 
with unbounded complexes, and in particular for the proper base change theorem in this setting.  Some
discussion about working in the unbounded setting can be found in \cite{kscategoriessheaves}.

Given a topological space $T$, we write $\Op(T)$ for the category whose objects are open sets and morphisms are inclusions.  
A ($\ZZ$-module valued) presheaf on $T$ is by definition a functor $\Op(T)^\op \to \Mod\ZZ$.
In particular, a presheaf $\F$ takes a value $\F(U) \in \Mod\ZZ$ on an open set $U \subseteq T$, termed its sections; given open sets $U \subseteq V$ it gives
a morphism $\F(V) \to \F(U)$, termed the restriction, etcetera.  Given any subset $X\subseteq T$, we write $\F(X) = \varinjlim_{X \subseteq U} \F(U)$; 
when $X$ is a point, this is termed the stalk and is written $\F_x$. 

The category of sheaves is the full subcategory of presheaves on objects $\F$ taking covers to limits:
\begin{equation}
\F\left(\bigcup_{i \in I} U_i\right) \xrightarrow{\sim} \lim_{\emptyset\ne J \subseteq I} \F\left( \bigcap_{j \in J} U_j \right)
\end{equation}
The inclusion of sheaves into presheaves has a left adjoint termed ``sheafification'',
giving, for any presheaf $\F$, a sheaf $\F^{\mathrm{sh}}$ such that any map from
$\F$ to a sheaf factors uniquely through $\F^{\mathrm{sh}}$.

We write $\Sh(T)$ for the (dg) category of sheaves of (dg) $\ZZ$-modules on $T$.  It is complete and co-complete.
Its homotopy category is what was classically called the unbounded derived category of sheaves on $T$.  

For any continuous map $f: S \to T$, there is an adjoint pair
$f^*: \Sh(T) \leftrightarrow  \Sh(S): f_*$.  The pushforward $f_*$ is given by the formula 
$(f_* \F)(U) = \F(f^{-1}(U))$, while the pullback $f^*$ is the sheafification of the presheaf given by 
$(f^*\G)(V) = \G(f(V))$.  

\begin{example}
Consider $f: S \to \pt$, and the constant sheaf $\ZZ := f^* \ZZ$.  Note that in 
our conventions, $\ZZ(U)$ is a chain complex computing the cohomology of $U$.  This should illustrate
where, in this account of sheaf theory, is hiding the usual homological algebra of resolutions: it is in the sheafification.
\end{example}

Being a left adjoint, $f^*$ is  co-continuous
(preserves colimits, in particular, sums).  When $j: U \to T$ is the inclusion of an open set, $j^*$ is given by 
the  simpler formula $(j^*\F)(V) = \F(V)$, no sheafification required, and hence preserves limits
as well.  In particular, it must also be a right adjoint.  The corresponding left adjoint $j_!$ is easy to describe: it is the sheafification of
\begin{equation*}
V\mapsto\begin{cases} \F(V) & V \subseteq U \\ 0 & \textrm{otherwise.} \end{cases}
\end{equation*}
The sheaf $j_! \F$ is termed the extension by zero, since its stalks in $U$ are isomorphic to the corresponding stalks
of $\F$, and its stalks outside of $U$ are zero.   
For a sheaf $\F$ on $T$, we write $\F_U := j_! j^* \F$.  By adjunction there is a canonical morphism
$\F_U \to \F$.  The object $\ZZ_U$ co-represents the functor of sections over $U$, i.e. $\Hom(\ZZ_U,\F)=\F(U)$.

Being a right adjoint, $f_*$ is  continuous. When $f$ is proper, it is in addition co-continuous.  More generally, 
for a morphism of locally compact spaces $f: S \to T$, one defines\footnote{This particular way of defining the 
$!$ pushforward is taken from \cite{schapira-notes}.  It has the virtue of making the co-continuity of $f_!$ obvious.}  
\begin{eqnarray*}
f_!: \Sh(S) & \to &  \Sh(T) \\
\F & \mapsto & \varinjlim_{U \subset \subset S} f_* \F_U
\end{eqnarray*} 
Here the notation $U \subset \subset S$ means that the closure of $U$ is compact.  When $S$ is an open subset,
this recovers the original definition.  When $f$ is proper, then $f_! = f_*$. When $f$ is the map to a point, then 
$f_! f^* \ZZ$ is the compactly supported cohomology.  

As $f_!$ is built from colimits, left adjoints, and pushforwards from compact sets, it is co-continuous.  
As such it has a right adjoint, denoted $f^!$.  When $f$ is the inclusion
of an open subset, we already had the right adjoint $f^*$, so in this case $f^* = f^!$. 

For any locally closed subset $v: V \subseteq T$, we extend the notation $\F_V := v_! v^* \F$.  This sheaf 
has the same stalks as $\F$ at points in $V$, and has vanishing stalks outside.  

For an open-closed decomposition $U \xhookrightarrow{j} T \xhookleftarrow{i} V$ ($j$ open, $i$ closed), the functors $j_{*}, j_!$ and $i_*, i_!$ are fully 
faithful, and there is an exact triangle
\begin{equation}\label{triangle}
j_! j^! \to \id \to i_* i^* \xrightarrow{[1]}
\end{equation}

Denoting by $\Op(M)$ the poset of open sets, there are functors
\begin{align}\label{generatingsheaves}
\Op(M)&\xrightarrow{!}\Sh(M)&\Op(M)^\op&\xrightarrow{*}\Sh(M)\cr
U&\mapsto u_!\ZZ&U&\mapsto u_*\ZZ
\end{align}
where $u:U\to M$ denotes the inclusion.
We have the following criterion for when (pullbacks of) these functors are fully faithful:

\begin{lemma} \label{lem:posetsheaf}
Let $\Pi$ be a poset with a map to $\Op(M)$, and let $\ZZ[\Pi]$ denote its dg linearization.
The following are equivalent:
\begin{itemize}
\item $H^*(U) \cong \ZZ$ for all $U \in \Pi$ and $H^*(U) \xrightarrow{\sim} H^*(U \setminus V)$ whenever $U \nsubseteq V$.
\item The composition $\ZZ[\Pi]\to\ZZ[\Op(M)]\xrightarrow{!}\Sh(M)$ is fully faithful.
\end{itemize}
\end{lemma}

\begin{proof}
We have
\begin{align}
\Hom_M(\ZZ_U, \ZZ_V) &= \Hom_M(u_! \ZZ, v_! \ZZ) = \Hom_U (\ZZ, u^! v_! \ZZ) = \Hom_U(\ZZ, u^* v_! \ZZ)\cr
&=H^*(U, \ZZ_{V \cap U})=\cone(H^*(U) \to H^*(U \setminus V))[-1],
\end{align}
where we have used the exact triangle \eqref{triangle}.
The second condition asks that this be $\ZZ$ when $U \subseteq V$ and zero otherwise, which is exactly what is asserted in the first condition.
\end{proof}

\begin{lemma} \label{lem:sheafrefine}
Let $\Pi$ be a poset with a map to $\Op(M)$ satisfying the equivalent conditions in Lemma \ref{lem:posetsheaf},
and suppose that $W\subseteq M$ is an open set such that 
$H^*(U) \xrightarrow{\sim} H^*(U \setminus W)$ is an isomorphism whenever $U\nsubseteq W$.
Then the pullback of the module $\Hom(-, \ZZ_W)$ along $\ZZ[\Pi] \xrightarrow{!} \Sh(M)$ is the indicator functor
\begin{equation}
1_{W}: U \mapsto \begin{cases} \ZZ & U \subseteq W, \\ 0 & \textrm{otherwise.} \end{cases}
\end{equation}
\end{lemma}

\begin{proof}
This is true by the same calculation as above.
\end{proof}

\subsection{Constructible sheaves}

Let $T$ be a topological space and $\SSS: T = \coprod T_\alpha$ a stratification.  Write $T_\SSS$ for the topological 
space with underlying set $T$ and base given by the stars of strata in $\SSS$ (note that the intersection of any two stars is expressible as a union of stars).
Note the continuous map $T \to T_\SSS$. 

\begin{remark}\label{weakremark}
Let $\pi: T \to T'$ be any continuous bijection.
For any open set $U$ of $T'$, and any sheaf $\F$ on $T$, one
has by definition $\pi_*\F(U) = \F(U)$.  It follows that $\pi^*\ZZ_U = \ZZ_U$, as this sheaf co-represents
the functor of sections over $U$. 
\end{remark}

\begin{lemma} \label{posetmeaning}
Pulling back sheaves under $\SSS\xrightarrow\star\Op(T_\SSS)$ defines an equivalence
\begin{align}
\Sh(T_\SSS)&\xrightarrow{\sim}\Fun(\SSS^\op,\Mod\ZZ)=\Mod\SSS\\
\F&\mapsto(s\mapsto\F(\star(s)))=\Hom_{T_\SSS}(\ZZ_{\star(-)},\F)
\end{align}
which sends $\ZZ_{\star(s)}$ to $\Hom_{\SSS}(\cdot,s)=1_{\star(s)}$.
\end{lemma}

\begin{proof}
The functor in question is simply restricting a sheaf on $T_\SSS$ to the base consisting of stars of strata.
This functor is fully faithful because a map of sheaves is determined by its restriction to a base for the topology.
It is essentially surjective because
there are no nontrivial covers of stars of strata by stars of strata.  The behavior on objects is as asserted because $\ZZ_{\star(s)}$ and $s$ 
are the co-representatives of the functors of sections over $s$ and the value of the module at $s$, respectively.
\end{proof}

\begin{lemma}\label{posetmeaningpushforward}
If $\SSS'$ refines $\SSS$, then the following diagram commutes:
\begin{equation}
\begin{tikzcd}
\Sh(T_{\SSS'})\ar{r}&\Mod\SSS'\\
\Sh(T_\SSS)\ar{r}\ar{u}{\pi^*}&\Mod\SSS\ar{u}[swap]{r^*}
\end{tikzcd}
\end{equation}
where $\pi^*$ denotes pullback of sheaves under the continuous map $\pi:T_{\SSS'}\to T_\SSS$ and $r^*:\Mod\SSS\to\Mod\SSS'$ denotes the pullback along the natural map $r:\SSS'\to\SSS$.
\end{lemma}

\begin{proof}
By Remark \ref{weakremark} and the characterization of the horizontal functors as $\ZZ_U\mapsto 1_U$.
\end{proof}

A sheaf is said to be constant when it is isomorphic to the star pullback of a sheaf on a point, 
and locally constant when this is true after restriction to an open cover.
For a stratification $\SSS$ of $M$, we say a sheaf is $\SSS$-constructible\footnote{Some sources, such as \cite{kashiwara-schapira}, also 
ask that the word constructible should mean that sheaves should have perfect stalks and bounded cohomological degree.  We do not.}
if it is locally constant when star restricted to each stratum of $\SSS$.  We write $\Sh_\SSS(T)$ for the full subcategory of $\Sh(T)$ 
on the $\SSS$-constructible sheaves.  

Note that the image of the pullback map $\Sh(T_{\SSS}) \to \Sh(T)$ is contained in $\Sh_\SSS(T)$.

\begin{lemma} \label{sheafposetequivalence}
For a triangulation $\SSS$, the map $\Sh(T_{\SSS}) \to \Sh_{\SSS}(T)$ is an equivalence.
\end{lemma}

\begin{proof}
To show full faithfulness, in view of the equivalence of Lemma \ref{posetmeaning} it is enough to check that $\Hom_{T_\SSS}(\ZZ_{\star(s)},\ZZ_{\star(t)})=\Hom_T(\ZZ_{\star(s)},\ZZ_{\star(t)})$.
The former is the indicator of $\star(s)\subseteq\star(t)$ again by Lemma \ref{posetmeaning}.
To show that $\Hom_T(\ZZ_{\star(s)},\ZZ_{\star(t)})$ is as well, by Lemma \ref{lem:posetsheaf} it is enough to show that $H^*(\star(s))\to H^*(\star(s)\setminus\star(t))$ is an isomorphism for $\star(s)\nsubseteq\star(t)$.
If $\star(s)\nsubseteq\star(t)$, then $\star(s)\setminus\star(t)$ is the join of something with $s$, and is hence contractible.

Regarding essential surjectivity, note that the exact triangle of \eqref{triangle} serves to decompose
any sheaf into an iterated extension of (extensions by zero of) sheaves on the strata; hence any constructible sheaf into 
(extensions by zero of) locally constant sheaves on the strata.  Since the strata are all contractible, these sheaves are in fact constant.
This shows that the $\ZZ_s$ generate.  To conclude that the $\ZZ_{\star(s)}$ generate, 
use the exact triangle $\ZZ_{\star(s) \setminus s} \to \ZZ_{\star(s)} \to \ZZ_s \xrightarrow{[1]}$ and induction on dimension of strata (noting that the first term is in the span of $\ZZ_t$ for $\dim(t) < \dim(s)$).
\end{proof}

\subsection{Microsupport}\label{subsec:microsupport}

The notion of microsupport is developed in  \cite{kashiwara-schapira}.\footnote{In \cite{kashiwara-schapira}, the authors
work in the bounded derived category.   As noted in \cite{robalo-schapira}, the only real dependence on this 
was in the proof of one lemma, which is extended to the unbounded setting in that reference.}  We recall some basic facts here.

For what follows, let $M$ denote an analytic manifold.
Given a sheaf $\F$ and 
a smooth function $\phi: M \to \RR$, consider a point $m$ in a level set $\phi^{-1}(t)$.  We say that $m \in M$ is a cohomological $\F$-critical 
point of $\phi$ if, for inclusion of the superlevelset $i: \phi^{-1}(\RR_{\ge t}) \hookrightarrow M$, one has $(i^! \F)_m \ne 0$.

The \emph{microsupport} $\ss(\F) \subseteq T^*M$ is by definition the closure of the locus of differentials of functions at their cohomological $\F$-critical points \cite{kashiwara-schapira}.
It is conical.

If $\F$ is locally constant, then a cohomological $\F$-critical point can only occur where the function in question has zero derivative.  Thus 
the microsupport of a locally constant sheaf is contained in the zero section (and is equal to it where the sheaf is not locally zero).  
If $U \subseteq M$ is an open set and $m$ is a point in the smooth locus of $\partial U$, then over $m$, the locus $\ss(\ZZ_U)=\ss(u_!\ZZ)$ is the half-line of outward
conormals to $\partial U$.  The locus $\ss(u_*\ZZ)$ is the inward conormal.

For a subset $X \subseteq T^*M$, we write $\Sh_X(M)$ for the full subcategory of $\Sh(M)$ spanned by objects with microsupport contained in $X$.
Similarly, for $X \subseteq S^*M$, we write $\Sh_X(M)$ for the full subcategory of $\Sh(M)$ with microsupport at infinity contained in $X$.
Evidently if $0_M \subseteq X$, then $\Sh_X(M) = \Sh_{\partial_{\infty}X}(M)$.  

The following result is a strengthening of \cite[Prop.\ 8.4.1]{kashiwara-schapira}:

\begin{proposition}\label{better841}
For a Whitney stratification $\SSS$ of a $C^1$ manifold $M$, we have $\Sh_\SSS(M)=\Sh_{N^*\SSS}(M)$ (i.e.\ having microsupport contained in $N^*\SSS$ is equivalent to being $\SSS$-constructible).
\end{proposition}

\begin{proof}
We first show the inclusion $\Sh_\SSS(M)\subseteq\Sh_{N^*\SSS}(M)$.
Let us first show that $\ss(\ZZ_X)\subseteq N^*\SSS$.
When $X$ is relatively compact, express $X$ as the ascending union of locally closed submanifolds-with-corners $X_i = X\setminus(\partial X)^{\underline\epsilon_i}$ where $(\partial X)^{\underline\epsilon_i}$ denotes the outward cornering of $\partial X=\overline X\setminus X$ in the sense of \eqref{outwardcornering} and $\underline \epsilon_i \to 0$.
Corollary \ref{neighborhoodcorners} implies $X\setminus(\partial X)^{\underline\epsilon_i}$ are indeed locally closed submanifolds-with-corners, and Corollary \ref{cornerconormalconvergence} implies that their conormals limit inside $N^*\SSS$ as $\underline\epsilon_i\to 0$, and hence that $\ss(\ZZ_X)=\ss(\varinjlim\ZZ_{X_i})\subseteq N^*\SSS$.
The case of general $X$ may be reduced to the relatively compact case by refining the stratification as in Remark \ref{addsubmanifoldremark} (the assertion $\ss(\ZZ_X)\subseteq N^*\SSS$ is local).
The same argument shows that for any locally constant sheaf on $X$, its lower shriek pushforward to $M$ has microsupport contained in $N^*\SSS$.
Since any $\SSS$-constructible sheaf is (locally) a finite iterated extension of such sheaves, we conclude that $\Sh_\SSS(M)\subseteq\Sh_{N^*\SSS}(M)$.

We now show that the inclusion $\Sh_\SSS(M)\subseteq\Sh_{N^*\SSS}(M)$ implies the reverse inclusion $\Sh_{N^*\SSS}(M)\subseteq\Sh_\SSS(M)$ by a straightforward d\'evissage argument.
Suppose $\ss(\F)\subseteq N^*\SSS$, and let us show that $\F$ is $\SSS$-constructible.
Let $X$ be a maximal stratum over which $\F$ is nonzero, and let $U\subseteq M$ be an open set containing $X$ so that $X\subseteq U$ is the support of $F|_U$ (hence, in particular, $X\subseteq U$ is closed).
Since $\ss(\F)\subseteq N^*\SSS$, there exists a (derived) local system on $X$ whose lower shriek pushforward $\F_0$ (which is $\SSS$-constructible) agrees with $\F$ over $U$.
Since $\F_0\in\Sh_\SSS(M)\subseteq\Sh_{N^*\SSS}(M)$, it suffices to show that the cone of $\F_0\to\F$ is $\SSS$-constructible.
We have thus reduced to a sheaf with smaller support.
Iterating, we eventually reduce to the case of $\F=0$ which is obviously $\SSS$-constructible.
\end{proof}

\subsection{Microstalks}

Recall that if $\F$ is a sheaf and $\phi$ is a smooth function with $\phi(x) = t$ and $d\phi_x = \xi$, and we denote 
the inclusion $i: \phi^{-1}(\RR_{\ge t}) \to M$, 
then if $(i^! F)_x \ne 0$, we have $(x, \xi) \in \ss(\F)$ (though not conversely).
Given this, one wants to assign the complex $(i^! \F)_x$ itself as an invariant of $\F$ at $(x, \xi)$.  This is not generally
possible, but it can be done when $\xi$ is a point in the smooth Lagrangian locus of $\ss(\F)$ \cite[Prop.\ 7.5.3]{kashiwara-schapira}.
Namely, at any smooth Lagrangian point $(x,\xi)\in X\subseteq T^*M$, there is a `microstalk' functor
\begin{equation}
\mu_{(x, \xi)}:\Sh_X(M) \to \Sh(\pt).
\end{equation}
It is given by a shift of $\F \mapsto (i^! \F)_x$ for any $\phi$ with $d_x \phi = \xi$ with the graph of $d \phi$ transverse to $X$.
The shift can be fixed using the index of the three transverse Lagrangians $(\ss(\F), T^*_x M, \Gamma_{d \phi})$.
When $\xi = 0$, the microstalk functor is simply the stalk functor.

\begin{lemma}\label{microstalkcocontinuous}
The microstalk functors are co-continuous.
\end{lemma}

\begin{proof}
Every stalk functor ($i^*$ for $i$ the inclusion of a point) is a left adjoint, hence is co-continuous.
To show co-continuity of the microstalk at a point $(x,\xi)\in\Lambda$ with $\xi\ne 0$, argue as follows.
By applying a contact transformation, we may reduce to the case that $\Lambda$ is (locally near $(x,\xi)$) the conormal of a smooth hypersurface $N\subseteq M$.
Let $B\subseteq M$ be an open ball with smooth boundary whose inward conormal at $x\in\partial B$ is $\xi$.
Moreover, choose $B$ so that $N^*\partial B$ and $\Lambda=N^*N$ intersect cleanly at $(x,\xi)$ (that is, $\partial B$ and $N$ are tangent at $x$, differing by a non-degenerate quadratic form).
Define $B_-$ and $B_+$ from $B$ by pushing $\partial B$ inward/outward near $x$.
Now the cone of the map $\Gamma_c(B_-,-)\to\Gamma_c(B_+,-)$ is (up to a shift) the microstalk functor $\mu_{(x,\xi)}$.
The compactly supported sections functor $\Gamma_c$ is co-continuous, since it is the composition of the restriction and lower shriek pushforward functors, both of which are co-continuous.
\end{proof}

\begin{proposition}\label{removemicrosupport}
Let $X \subseteq T^*M$ be closed and conical, and let $\Lambda\subseteq T^*M \setminus X$ be closed, conical, and stratified by isotropic submanifolds.
Then $\Sh_X(M) \subseteq \Sh_{X \cup \Lambda}(M)$ is the kernel of all microstalks at Lagrangian points of $\Lambda$.
\end{proposition}

\begin{proof}
If $\ss(\F)\subseteq X$, then the microstalks of $\F$ at Lagrangian points of $\Lambda$ vanish by definition of microsupport.
To prove the converse, suppose that $\ss(\F)\subseteq X\cup\Lambda$ and that the microstalks of $\F$ vanish at all Lagrangian points of $\Lambda$, and let us show that $\ss(\F)\subseteq X$.
By the fundamental result \cite[Thm.\ 6.5.4]{kashiwara-schapira} that the microsupport is co-isotropic, it is enough to show that $p\notin\ss(\F)$ for every Lagrangian point $p\in\Lambda$.
It is not quite immediate from the definitions that vanishing of the microstalk implies there is no microsupport,
since the microsupport is defined in terms of arbitrary test functions, whereas microstalks are defined in terms of microlocally transverse test functions.
To see it is true, and that moreover the microstalk is locally
constant along $\Lambda$, one can apply a contact transformation so that $\Lambda$ becomes locally the conormal
to a smooth hypersurface; for details see \cite[Chap.\ 7]{kashiwara-schapira}.
\end{proof}

It will be central to our discussion to find co-representatives of the microstalk functors.
Here is a first step:

\begin{theorem}[{\cite[Cor.\ 5.4.19, Prop.\ 5.4.20, Prop.\ 7.5.3]{kashiwara-schapira} or \cite{goresky-macpherson, schmid-vilonen}}]\label{thm:sheafwex}
Let $X \subseteq T^*M$ be a closed conical subset, let 
$\phi: M \to \RR$ be a proper function, and assume that over $\phi^{-1}([a, b))$, one has $\Gamma_{d \phi} \cap X = (x, \xi)$,
where $(x, \xi)$ is a smooth Lagrangian point of $X$.

Let $A: \phi^{-1}((-\infty, a)) \to M$, $A': \phi^{-1}((a, \infty)) \to M$, $B: \phi^{-1}((-\infty, b)) \to M$, and $B': \phi^{-1}((b, \infty)) \to M$ be the inclusions.
Then (up to a shift), the following functors $\Sh_X(M) \to \Sh(\pt)$ are isomorphic:
\begin{itemize}
\item The microstalk functor $\mu_{(x,\xi)}$.
\item $\Hom(\cone(A_! \ZZ \to B_! \ZZ), -)$.
\item $\Hom(\cone(A'_* \ZZ \to B'_* \ZZ), -)$.
\end{itemize}
Here the maps are the canonical ones coming from restriction of sections.
\end{theorem}

We do not say that $\cone(A_! \ZZ \to B_! \ZZ)$ co-represents the microstalk because it is not an element of $\Sh_X(M)$.
As observed in \cite{nadler-wrapped}, such co-representatives do exist, for categorical reasons, as we now explain.
First, we need to know that the categories in question are well generated in the sense of Neeman \cite{neeman-book,krausewellgenerated}.

\begin{lemma}\label{wellgeneration}
The category $\Sh_X(M)$ is well generated.
\end{lemma}

\begin{proof}
The category of all sheaves $\Sh(M)$ is the derived category of a Grothendieck abelian category, hence is well generated \cite{neeman-sheaves}.
A sheaf $\F$ having singular support inside $X$ is equivalent to the restriction maps $\F(U_\alpha)\to\F(V_\alpha)$ being isomorphisms for some list of pairs $(U_\alpha,V_\alpha)_\alpha$ depending on $X$.
This condition is equivalent to $\F$ being right-orthogonal to the cone of the map $\ZZ_{V_\alpha}\to\ZZ_{U_\alpha}$.
Now the right-orthogonal complement of a set of objects in a well generated category is well generated \cite{neeman-book}\cite[Thm.\ 4.9]{portaquotient}.
\end{proof}

Now note that the microsupport of a sum or product is contained in the closure of the union of the microsupports.  Thus if $X$ is closed, 
then the subcategory 
$\Sh_X(M) \subseteq \Sh(M)$ is closed under sums and products.  In particular, $\Sh_X(M)$ is complete and co-complete, 
and the inclusion $\Sh_X(M) \to \Sh(M)$ is continuous and co-continuous.  More generally, if $X \subseteq X'$ are closed, then 
the inclusion $\iota: \Sh_X(M) \to \Sh_{X'}(M)$ is continuous and co-continuous.
It follows that:

\begin{lemma}\label{adjoints}
For closed $X\subseteq X'\subseteq T^*M$, the inclusion $\iota:\Sh_X(M)\to\Sh_{X'}(M)$ has both adjoints: $(\iota^*, \iota, \iota^!)$.
\end{lemma}

\begin{proof}
Since $\Sh_X(M)$ is well generated and $\iota$ is co-continuous, it has a right adjoint $\iota^!$ by Brown representability for well generated categories \cite{neeman-book}.

For the left adjoint $\iota^*$, it would be sufficient to know Brown representability for the opposite of $\Sh_X(M)$.
However according to Neeman \cite{neemanexcogenerators}, it is an open problem to establish Brown representability for the opposites of well generated categories.
Instead, we may argue as follows.
The categories in question are presentable (co-complete and accessible \cite[Def.\ 5.4.2.1 and 5.5.0.1]{luriehttpub}; this is a version of well generation) and so a functor has a left adjoint iff it is continuous and accessible (preserves $\kappa$-filtered colimits \cite[Def.\ 5.4.2.5 and 5.3.4.5]{luriehttpub}) by Lurie \cite[Corollary 5.5.2.9]{luriehttpub}.

Another proof of the existence of the left adjoint $\iota^*$ has been given by Efimov \cite{efimovleftadmissible}.
\end{proof}

For example, if $V \subseteq M$ 
is a closed subset, then taking $X = T^*M|_V$ and $X' = T^* M$ recovers the adjoint triple for the pushforward along $V \to M$,
because $\Sh_{T^*M|_V}(M)=\Sh(V)$.

Using the left adjoint and Theorem \ref{thm:sheafwex}, we can obtain a co-representative for the microstalk as follows.  Take any $X' \supseteq \ss(\cone(A_! \ZZ \to B_! \ZZ))$, e.g.\ $X' = T^*M$.  Then $\iota^*\cone(A_! \ZZ \to B_! \ZZ) \in \Sh_X(M)$ co-represents the microstalk.

We do not generally have a good understanding of $(\iota^*, \iota, \iota^!)$, but when $X' \setminus X$ is isotropic we have the following (special cases of which have appeared in \cite{nadler-wrapped,ikekuwagaki}):

\begin{theorem} \label{thm:sheafstopremoval}
Let $X \subseteq T^*M$ be closed and conical, and let $\Lambda\subseteq T^*M \setminus X$ be closed, conical, and stratified by isotropic submanifolds.
The left adjoint $\iota^*$ to the inclusion $\iota:\Sh_X(M)\to\Sh_{X\cup\Lambda}(M)$ realizes the quotient
\begin{equation}\label{sheafstopremovalqe}
\Sh_{X \cup \Lambda}(M)/\D\xrightarrow\sim \Sh_X(M),
\end{equation}
where $\D$ denotes co-representing objects for the microstalks at Lagrangian points of $\Lambda$.
\end{theorem}

\begin{proof}
According to Proposition \ref{removemicrosupport}, the full subcategory $\Sh_X(M)\subseteq\Sh_{X\cup\Lambda}(M)$ is precisely the right-orthogonal to $\D$.
The left adjoint $\iota^*$ to the inclusion $\iota$ is thus termed the quotient by $\D$.
\end{proof}

\begin{remark}\label{noneedwellgeneration}
For our purposes in this paper, we do not need Lemma \ref{adjoints} and Theorem \ref{thm:sheafstopremoval} in their general formulations given above, rather only in the special case of subanalytic singular isotropic singular supports.
In this setting, we give an elementary derivation (i.e.\ without appealing to general Brown representability type statements) in the next subsection.
\end{remark}

\subsection{Compact objects}

Here we elaborate upon some assertions of \cite{nadler-wrapped}.

We write $\Sh_X(M)^c$ for the compact objects in the category $\Sh_X(M)$. Be warned: 

\begin{proposition}[\cite{neeman-sheaves}]
When $M$ is connected and non-compact, $\Sh(M)^c = 0$.\qed
\end{proposition}

There are not many more compact objects in the compact case.
However, for sheaves with prescribed isotropic microsupport, the situation is different: 

\begin{proposition} \label{isotropicimpliescompact} 
For $\Lambda \subseteq T^*M$ a conic subset Whitney stratifiable by isotropics, the category 
$\Sh_\Lambda(M)$ is compactly generated by the (co-representatives of the) microstalk functors at the smooth Lagrangian points of $\Lambda$.
\end{proposition}

\begin{proof}
It was shown immediately after the proof of Lemma \ref{adjoints} that the microstalk functors at smooth Lagrangian points of $\Lambda$ are co-represented by objects of $\Sh_\Lambda(M)$.
Since the microstalk functors are co-continuous (Lemma \ref{microstalkcocontinuous}), these co-representatives are compact.
Any sheaf right-orthogonal to 
these co-representatives has by definition vanishing microstalks at all smooth Lagrangian points of $\Lambda$, hence has microsupport contained in the complement of the 
smooth Lagrangian locus (see the proof of Proposition \ref{removemicrosupport}).  This complement, being stratified by subcritical isotropics, has no co-isotropic subset; hence by 
the involutivity of microsupports \cite[Thm.\ 6.5.4]{kashiwara-schapira}, the microsupport
is in fact the empty set and the sheaf vanishes.
\end{proof} 

As promised in Remark \ref{noneedwellgeneration}, we now give arguments avoiding the use of  representability theorems in non-compactly-generated
categories.  This comes
at the cost of assuming subanalyticity in order to ensure the existence of triangulations, but for the main results
we will anyway need this hypothesis. 

\begin{lemma} \label{constructibleimpliescompact}
For $\SSS$ a triangulation, the category $\Sh_\SSS(M)$ is compactly generated, and 
the objects of $\Sh_\SSS(M)^c$ are the sheaves with perfect stalks and compact support.  
\end{lemma}

\begin{proof}
Under the identification (Lemma \ref{sheafposetequivalence}) $\Sh_\SSS(M) = \Mod \SSS$, the $\ZZ_{\star(s)}$ go to compact generators. 
The d\'evissage in the proof of the same Lemma shows that  $\ZZ_s$ also generate, and can be expressed using finitely many  
$\ZZ_{\star(s)}$, hence are compact.  The $\ZZ_s$ evidently generate the sheaves with perfect stalks and compact support. 
\end{proof}

\begin{remark}\label{noncompacttriangulation}
Note that while a non-compact manifold does not admit a finite triangulation, it can sometimes be a relatively compact constructible subset of a larger manifold.
\end{remark}

Recall that $\Sh_\SSS(M)=\Sh_{N^*\SSS}(M)$ for any Whitney stratification by Proposition \ref{better841}.

\begin{proposition} \label{compactmicrostalks}
For any subanalytic Whitney triangulation $\SSS$, the category $\Sh_\SSS(M)$ is compactly generated by co-representatives of the microstalks at smooth points of $N^*\SSS$.
\end{proposition}

\begin{proof}
Consider the microstalk at some smooth point $(x, \xi)\in N^*\SSS$.  It is possible to choose 
real analytic $\phi$ as in Theorem \ref{thm:sheafwex}, see \cite[Thm.\ 2.2.1]{goresky-macpherson} or
\cite[Prop.\ 8.3.12]{kashiwara-schapira}.  We keep the notation of Theorem \ref{thm:sheafwex}.
Refine $\SSS$ to a subanalytic Whitney triangulation $\SSS'$ for which $A_!\ZZ$ and $B_!\ZZ$ are constructible.

By Lemma \ref{constructibleimpliescompact}, $\cone(A_! \ZZ \to B_! \ZZ)$ is a compact object in $\Sh_{\SSS'}(M)$.
Lemmas \ref{posetmeaning} and \ref{sheafposetequivalence} give $\Sh_\SSS(M)=\Mod\SSS$, and Lemma \ref{posetmeaningpushforward} states that the inclusion $\iota:\Sh_\SSS(M)\to\Sh_{\SSS'}(M)$ corresponds to the map $r^*:\Mod\SSS\to\Mod\SSS'$ from Section \ref{posetcategories}.
It was observed in Section \ref{posetcategories} that $r^*$ has a left adjoint $r_!:\Mod\SSS'\to\Mod\SSS$, thus giving us a left adjoint $\iota^*:\Sh_{\SSS'}(M)\to\Sh_\SSS(M)$.
These left adjoints $r_!$/$\iota^*$ preserve compact objects since $r^*$/$\iota$ are co-continuous.
Thus the object $\iota^*\cone(A_! \ZZ \to B_! \ZZ) \in \Sh_\SSS(M)$, which co-represents the desired microstalk, is compact.

Co-representatives of the microstalks at all smooth points of $N^*\SSS$ generate $\Sh_\SSS(M)$ by Proposition \ref{removemicrosupport}.
\end{proof}

\begin{remark}\label{compactstalks}
A similar argument shows that the stalk at any point of $M$ (not necessarily a smooth point of $N^*\SSS$) is co-representable by a compact object of $\Sh_\SSS(M)$.
Indeed, note that for any $x \in M$, the functor of taking stalks at $x$, which is by definition $\F_x := \varinjlim \F(B_\epsilon(x))$,
is in fact computed by some fixed $\F_x = \F(B_{\epsilon_x}(x))$.  Indeed, further shrinking of the ball will be non-characteristic with respect
to $N^*\SSS$, as follows from Whitney's condition (b) (or alternatively from microlocal Bertini--Sard \cite[Prop.\ 8.3.12]{kashiwara-schapira}).
We may now argue as above, choosing any analytic function with sublevelset $B_{\epsilon_x}(x)$.
\end{remark}

\begin{corollary}\label{microstalksgenerate}
For any closed conical subanalytic isotropic $\Lambda\subseteq T^*M$, the category $\Sh_\Lambda(M)$ is compactly generated by co-representatives of the microstalks at smooth points of $\Lambda$.
\end{corollary}

\begin{proof}
Fix a subanalytic Whitney triangulation $\SSS$ for which $\Lambda\subseteq N^*\SSS$.
Denote by $\D\subseteq\Sh_{N^*\SSS}(M)^c$ the co-representatives of the microstalks at smooth points of $N^*\SSS\setminus\Lambda$.
By Proposition \ref{removemicrosupport}, $\iota:\Sh_\Lambda(M)\subseteq\Sh_{N^*\SSS}(M)$ is precisely the inclusion of the right-orthogonal to $\D$.
Since the objects of $\D$ are compact by Proposition \ref{compactmicrostalks}, Lemma \ref{compactinquotient} applies to show that $\Sh_\Lambda(M)=\Sh_{N^*\SSS}(M)/\D$ is compactly generated by $\Sh_\Lambda(M)^c=(\Sh_{N^*\SSS}(M)^c/\D)^\pi$ and that the resulting functor $\iota^*:\Sh_{N^*\SSS}(M)\to\Sh_\Lambda(M)$ is left adjoint to $\iota$.
\end{proof}

\begin{corollary} \label{compactsheafstopremoval}
Let $X\subseteq T^*M$ and $\Lambda\subseteq T^*M\setminus X$ be closed conical subanalytic isotropics.
The inclusion $\iota:\Sh_X(M)\to\Sh_{X\cup\Lambda}(M)$ has a left adjoint $\iota^*:\Sh_{X\cup\Lambda}(M)\to\Sh_X(M)$ whose restriction to compact objects defines an equivalence
\begin{equation}
(\Sh_{X \cup \Lambda}(M)^c/\D)^\pi\xrightarrow\sim \Sh_X(M)^c,
\end{equation}
where $\D$ denotes co-representing objects for the microstalks at Lagrangian points of $\Lambda$.
\end{corollary}

\begin{proof}
By Corollary \ref{microstalksgenerate}, $\Sh_{X\cup\Lambda}(M)$ is compactly generated by the microstalks at smooth points of $X\cup\Lambda$.
Now argue as in the proof of Corollary \ref{microstalksgenerate}.
\end{proof}

The following result was shown in \cite{nadler-wrapped} using arborealization; here is a direct argument. 

\begin{corollary} \label{propersheaves}
The Yoneda embedding induces an equivalence between the full subcategory of $\Sh_\Lambda(M)$ of objects with perfect stalks
and the category $\Prop \Sh_\Lambda(M)^c$.
\end{corollary}

\begin{proof}
From the argument in Proposition \ref{compactmicrostalks}, we see that the microstalks are calculated by comparing sections over
precompact sets; it follows that a sheaf microsupported in $\Lambda$ (thus constructible) with perfect stalks has perfect microstalks.
The microstalk functors split-generate $\Sh_\Lambda(M)^c$ by Corollary \ref{microstalksgenerate}, so we see that a sheaf with perfect stalks defines a proper module over $\Sh_\Lambda(M)^c$.

To see the converse, recall from Remark \ref{compactstalks} that the stalk functors can be expressed 
in terms of sections over open sets constructible with respect to some $\SSS$ satisfying $N^*\SSS\supseteq\Lambda$.
The left adjoint to $\Sh_\Lambda(M)\hookrightarrow\Sh_\SSS(M)$ preserves compact objects as observed previously, hence proper over $\Sh_\Lambda(M)^c$ implies perfect stalks.
\end{proof}

For compact $M$, we  establish smoothness and/or properness for some of these categories. 

\begin{proposition} \label{sheafcompactsmoothproper}
If $M$ is compact and $\SSS$ is a triangulation, then $\Sh_\SSS(M)^c$ is smooth and proper.
\end{proposition}

\begin{proof}
The $\ZZ_{\star(s)}$ give a finite generating exceptional collection which is proper, and this implies smoothness by Lemma \ref{exceptionalpropersmooth}.
\end{proof}

More generally,

\begin{corollary}
If $M$ is compact and $\Lambda$ is closed conical subanalytic singular isotropic, then $\Sh_\Lambda(M)^c$ is smooth, and hence
$\Prop \Sh_\Lambda(M)^c \subseteq \Perf \Sh_\Lambda(M)^c$ and $\Prop \Sh_\Lambda(M)^c $ is proper. 
\end{corollary}

\begin{proof}
By Proposition \ref{sheafcompactsmoothproper} and Corollary \ref{compactsheafstopremoval}, the category 
$\Sh_\Lambda(M)^c$ is a quotient of a smooth category, hence smooth (Lemma \ref{smoothquotient}).  
Smoothness implies proper modules are perfect (Lemma \ref{propperf}) and that the category of proper modules is proper.
\end{proof}

\begin{remark}\label{noncompactsmoothness}
When $(M,\Lambda)$ are non-compact but finite-type in a suitable sense, the
same result is true. One can prove it by embedding into a compact manifold as
in Remark \ref{noncompacttriangulation}.
\end{remark}

\subsection{In conclusion}

Collecting the results of this section, we have shown:

\begin{theorem}\label{sheafrealization}
The functor $\Lambda \mapsto \Sh_\Lambda(M)^c$ is a microlocal Morse theatre in the sense of Definition \ref{microlocalmorsetheatre}, which casts the co-representatives of the microstalk functors at smooth points of $\Lambda$ as the Morse characters.
\end{theorem}

\begin{proof}
The most obvious functor $\Lambda \to \Sh_\Lambda(M)$ is the one which carries \emph{inclusions} $\Lambda \subseteq \Lambda'$ to 
\emph{inclusions} $\Sh_\Lambda(M)\hookrightarrow\Sh_{\Lambda'}(M)$; note that this is in fact a strict diagram of categories (as all are simply full subcategories of $\Sh(M)$) and takes values 
in the category whose objects are large dg categories and whose morphisms are continuous and co-continuous functors.
Passing to left adjoints and taking compact objects (see Corollary \ref{compactsheafstopremoval}), we obtain a microlocal Morse pre-theatre $\Lambda\mapsto\Sh_\Lambda(M)^c$.

For triangulations $\SSS$, the functors
\begin{equation}
\SSS\xrightarrow{s\mapsto\ZZ_{\star(s)}}\Sh_\SSS(M)\xrightarrow{\F\mapsto\Hom(\ZZ_{\star(-)},\F)}\Mod\SSS
\end{equation}
define an equivalence $\Perf\SSS=\Sh_\SSS(M)^c$ by Lemmas \ref{posetmeaning} and \ref{sheafposetequivalence}.
When $\SSS$ is a Whitney stratification, we have $\Sh_\SSS(M)^c=\Sh_{N^*_\infty\SSS}(M)^c$ by Proposition \ref{better841}.

Taking the commutative diagram in Lemma \ref{posetmeaningpushforward} and passing to the left adjoints of the vertical maps
shows that this equivalence respects refinement of Whitney triangulations.
This shows that $\Lambda\mapsto\Sh_\Lambda(M)^c$ is normalized.

By Theorem \ref{thm:sheafwex}, the Morse characters in $\Perf\SSS$ correspond, under this isomorphism, to co-representatives of the microstalks.
According to Corollary \ref{compactsheafstopremoval}, the functor $\Sh_{\Lambda'}(M)^c\to\Sh_\Lambda(M)^c$ is the quotient by co-representatives of the microstalks.
Thus $\Lambda\mapsto\Sh_\Lambda(M)^c$ satisfies the localization property, and is thus a microlocal Morse theatre.
\end{proof}

\begin{proposition}\label{sheaffunctorial}
For any analytic open inclusion of analytic manifolds $M'\hookrightarrow M$, the restriction functors $\Sh_\Lambda(M)\to\Sh_{\Lambda'}(M')$ for subanalytic singular isotropics with $\Lambda'\supseteq\Lambda\cap S^*M'$ have left adjoints whose restrictions to compact objects form a morphism of microlocal Morse theatres.
\end{proposition}

\begin{proof}
The categories $\Sh_\Lambda(M)$ are compactly generated by Corollary \ref{microstalksgenerate}, and Brown representability holds for the opposites of compactly generated categories by \cite{neeman-book,krause}.
Thus since the restriction functors $\Sh_\Lambda(M)\to\Sh_{\Lambda'}(M')$ are continuous, they admit left adjoints.
Since restriction is co-continuous, these left adjoints preserve compact objects.
Restricting these left adjoints to compact objects defines a morphism of microlocal Morse pre-theatres in the sense of Section \ref{morsemorphism}.

Let us show that this is a morphism of microlocal Morse theatres, i.e.\ that it is normalized.
For a stratification $\SSS$ of $M$ and a stratification $\SSS'$ refining $\SSS\cap M'$, we have the following commutative diagram
\begin{equation}
\begin{tikzcd}
\Sh_{\SSS'}(M')&\ar{l}\Sh(M'_{\SSS'})\ar{r}{\sim}&\Mod\SSS'\\
\Sh_\SSS(M)\ar{u}&\ar{l}\Sh(M_\SSS)\ar{u}\ar{r}{\sim}&\Mod\SSS\ar{u}
\end{tikzcd}
\end{equation}
(compare Lemmas \ref{posetmeaning} and \ref{posetmeaningpushforward}).
When $\SSS$ and $\SSS'$ are triangulations, the left horizontal maps are also equivalences by Lemma \ref{sheafposetequivalence}.
Finally, when $\SSS$ is Whitney, we have $\Sh_\SSS(M)=\Sh_{N^*\SSS}(M)$ (and the same for $\SSS'$) by Proposition \ref{better841}.
Thus passing to left adjoints of the vertical arrows and restricting to compact objects, we conclude.
\end{proof}

\section{Wrapped Fukaya categories} \label{sec:fukaya}

\subsection{Wrapped Floer cohomology}

Here we quickly fix notation and review basic facts (see, e.g., \cite[Sec.\ 3]{gpssectorsoc} for more details).
Fix a Liouville manifold or open Liouville sector $X$.

For a pair of exact Lagrangians $L, K \subseteq X$, conical and disjoint at infinity, 
we write $HF^*(L, K)$ for their Floer cohomology.  We write $HF^*(L, L)$ to mean $HF^*(L^+, L)$, where $L^+$ denotes an (unspecified) small positive (meaning positive at infinity) pushoff of $L$.
There is a isomorphism of groups $HF^*(L,L) = H^*(L)$ and 
the group $HF^*(L, L)=HF^*(L^+, L)$ is a unital algebra;\footnote{One expects (as is known for compact $L$) that the isomorphism $HF^*(L,L) = H^*(L)$ is further compatible with algebra structures; we are not aware of a reference for this.}
its unit is termed the \emph{continuation element}.
Composition of continuation elements associated to small pushoffs defines more generally
a continuation element in $HF^*(L^{++}, L)$ for $L^{++}$ any (not necessarily small) positive wrapping (i.e., isotopy)  of $L$. Composition with the continuation element associated to $L \leadsto L^{++}$ gives maps $HF^*(L,K) \to HF^*(L^{++}, K)$ and $HF^*(K,L^{++}) \to HF^*(K,L)$ for any $K$ disjoint at infinity from $L$ and $L^{++}$, which are termed \emph{continuation maps}. 
If the entire positive isotopy $L\leadsto L^{++}$ takes place in the complement of $\partial_\infty K$, 
then these continuation maps are isomorphisms.
More generally, if $L \leadsto L'$ is any isotopy taking place in the complement of $\partial_\infty K$ (for example any compactly supported isotopy), then there is an induced identification $HF^*(L,K) = HF^*(L',K)$ (see \cite[Lem.\ 3.21]{gpssectorsoc}) which coincides with the continuation isomorphism if $L \leadsto L'$ is positive at infinity (see \cite[Lem.\ 3.26]{gpssectorsoc}).
In particular, seeing as $HF^*(L,K) = 0$ tautologically when $K$ and $L$ are disjoint, Floer cohomology $HF^*(L,K)$ vanishes whenever $L$ is disjoinable from $K$ by an isotopy in the complement of $\partial_\infty K$.

The wrapped Floer cohomology $HW^*(L, K)_X$ is equivalently calculated by
\begin{equation}
\varinjlim_{L \leadsto L^{++}} HF^*(L^{++}, K) = \varinjlim_{\substack{L \leadsto L^{++} \\ K^{--}\leadsto K}} HF^*(L^{++}, K^{--}) =  \varinjlim_{K^{--}\leadsto K} HF^*(L, K^{--}).
\end{equation}
Here, the direct limits are taken using the continuation maps over positive-at-infinity isotopies of $L$ and negative-at-infinity isotopies of $K$.  
The freedom to wrap in only one factor is extremely useful in practice.  

Given any closed subset $\Lambda \subseteq \partial_\infty X$, and $L, K$ disjoint at infinity from $\Lambda$, 
we similarly define partially wrapped Floer cohomology $HW^*(L, K)_{X,\Lambda}$ by restricting wrappings to take place in the complement of $\Lambda$.

The following Lemma allows one to explicitly describe some cofinal wrapping sequences in a given $(X,\Lambda)$.  Its typical use is the following.  To compute $HW^*(L, K)_{X,\Lambda}$, if 
one can find a cofinal sequence $L_t$ such that the induced maps $HF^*(L_t, K) \to HF^*(L_{t+1}, K) $ are eventually 
all isomorphisms,  then $HW^*(L, K)_{X,\Lambda} = HF^*(L_t, K)$ for any $L_t$ in this stable range.

\begin{lemma}[{\cite[Lem.\ 3.29]{gpssectorsoc} \cite[Lem.\ 2.2]{gpsdescent}}]\label{cofinalitycriterion}
Let $L_t$ be a positive isotopy of Lagrangians in $X$ avoiding $\Lambda$ at infinity.
If $\partial_\infty L_t$ escapes to infinity in (i.e.\ is eventually disjoint from any given compact subset of) $\partial_{\infty} X \setminus \Lambda$ as $t\to\infty$, then it is a cofinal wrapping of $L_0$ in $(X,\Lambda)$.
\qed
\end{lemma}

\subsection{Wrapped Fukaya categories}

In \cite{gpssectorsoc,gpsdescent}, for any Liouville sector $X$ and any closed subset $\Lambda\subseteq(\partial_\infty X)^\circ$, we constructed $\ainf$ categories $\W(X,\Lambda)$ whose objects are exact Lagrangians in $X \setminus \Lambda$, 
conical at infinity (by convention $\W(X):= \W(X,\emptyset)$).  The cohomology-level morphisms are simply the wrapped Floer cohomology groups as defined above: $H^*\W(L, K)=HW^*(L, K)_{X,\Lambda}$.
For a compact manifold-with-boundary $M$, its cotangent bundle $T^*M$ is a Liouville sector \cite[Ex.\ 2.7]{gpssectorsoc}.

One main point of \cite{gpssectorsoc} was the construction of a covariant functor $\W(X) \to \W(Y)$ for an inclusion of Liouville sectors $X \subseteq Y$.  
In \cite{gpsdescent} we remarked that the same construction gives a functor $\W(X, \Lambda\cap(\partial_\infty X)^\circ) \to \W(Y, \Lambda)$.  
This covariance is a nontrivial result having to do with the fact that holomorphic disks can be made to not cross the boundary of a Liouville sector (if the Lagrangian boundary conditions do not).  By 
contrast, it is immediate from the definition that if $\Lambda\subseteq\Lambda'$ then there is a natural map $\W(X,\Lambda') \to \W(X,\Lambda)$: just wrap more.  Both
covariance statements allow one to calculate in a potentially simpler geometry, and push forward the result.

We wish to consider here categories $\W(T^*M,\Lambda)$ for (possibly non-compact) manifolds $M$ without boundary and closed subsets $\Lambda\subseteq S^*M=\partial_\infty T^*M$.
Such a cotangent bundle $T^*M$ is an \emph{open Liouville sector} in the sense of \cite[Rem.\ 2.8]{gpssectorsoc} (meaning, concretely, it admits an exhaustion by Liouville sectors, in this case $T^*M_0\subseteq T^*M_1\subseteq\cdots$ where $M_0\subseteq M_1\subseteq\cdots$ is an exhaustion of $M$ by compact codimension zero submanifolds-with-boundary).

The construction of the wrapped Fukaya category of an open Liouville sector is given in \cite[Sec.\ 3.8]{gpssectorsoc}.
The generalization to case with a stop following \cite[Sec.\ 2]{gpsdescent} is straightforward.
The result is the following definition.
We consider tuples $(P,\{M_p\}_{p\in P},\{L_p\}_{p\in P},\underline J,\underline\xi)$ where:
\begin{enumerate}
\item$P$ is a partially ordered set.
\item Each $M_p\subseteq M$ is a compact codimension zero submanifold with smooth boundary, equipped with a choice of projection from $T^*M_p$ to $\CC_{\Re\geq 0}$ as in \cite[Def.\ 2.26]{gpssectorsoc} defined near the boundary.
\item Each $L_p\subseteq T^*M_p$ is an exact Lagrangian, cylindrical at infinity, disjoint from $\Lambda$ at infinity, equipped with grading and orientation data as in Section \ref{gradorsec} below, such that for every totally ordered subset $p_0>\cdots>p_k\in P$, the Lagrangians $L_{p_0},\ldots,L_{p_k}$ are mutually transverse.
\item The pair $(\underline\xi,\underline J)$ is a choice of compatible Floer data (strip-like coordinates and almost complex structures) as in \cite[Eq.\ (3.31)--(3.33)]{gpssectorsoc} for every totally ordered subset $p_0>\cdots>p_k\in P$ (so $J_{p_0,\ldots,p_k}$ is an almost complex structure on $T^*M_{p_0}$), such that all moduli spaces of Fukaya $\ainf$ disks are cut out transversally.
\end{enumerate}
Any such tuple gives rise to a $\ainf$ category whose objects are the elements of $P$, whose morphism spaces from $p$ to $p'$ are $CF^*(L_p,L_{p'})$ for $p>p'$, are $\ZZ$ for $p=p'$, and otherwise vanish.
We may ask that such a tuple be \emph{cofinite} (meaning $P^{\leq p}$ is finite for every $p\in P$) and \emph{duplicate-free} (meaning $P^{\leq p}$ equipped with the restriction of the remaining data are pairwise non-isomorphic for $p\in P$).
There is a \emph{universal} cofinite duplicate-free tuple $(P,\{M_p\}_{p\in P},\{L_p\}_{p\in P},\underline J,\underline\xi)$ \cite[Lem.\ 3.42]{gpssectorsoc}, which thus gives a canonically defined $\ainf$ category $\OO(T^*M,\Lambda)$.
The wrapped category $\W(T^*M,\Lambda)$ is defined as the localization $\OO(T^*M,\Lambda)[C^{-1}]$ (refer to \cite[Sec.\ 3.1.3]{gpssectorsoc} for localizations of $\ainf$ categories) at the class $C$ of all continuation elements in $HF^0(L_p,L_{p'})$ for positive isotopies $L_{p'}\leadsto L_p$ inside $T^*M_p$ disjoint at infinity from $\Lambda$.
That this category deserves the name $\W(T^*M,\Lambda)$ is justified by \cite[Prop.\ 3.43, Prop.\ 3.39, Lem.\ 3.37]{gpssectorsoc} and \cite[Sec.\ 2]{gpsdescent}; in particular, these show that it has the correct cohomology category.

The resulting category $\W(T^*M,\Lambda)$ is moreover strictly functorial in $M$ and $\Lambda$: for any open inclusion of manifolds $M\hookrightarrow M'$ such that $\Lambda$ contains the inverse image of $\Lambda'$, there is an induced functor $\W(T^*M,\Lambda)\to\W(T^*M',\Lambda')$, and these functors respect compositions of inclusions $M\hookrightarrow M'\hookrightarrow M''$.

\subsection{Gradings and orientations}\label{gradorsec}

We briefly review the setup for defining gradings and orientations in Floer theory; for more details see Seidel \cite{seidelgraded} and \cite[(11e)--(11l)]{seidelbook}.
Our Floer cohomology groups and Fukaya categories are all $\ZZ$-graded and with $\ZZ$ coefficients.

Denote by $\LGr(V)$ the Grassmannian of Lagrangian subspaces of a given symplectic vector space $V$.
A map $\partial D^2=S^1\to\LGr(V)$ defines elliptic boundary conditions for the $\bar\partial$-operator on the trivial vector bundle with fiber $V$ over $D^2$ (choosing also a compatible complex structure on $V$, which is a contractible choice), and hence a virtual vector space, namely the index (kernel minus cokernel) of this operator, thus giving a map
\begin{equation}
\sL(\LGr(V))\to\ZZ\times BO.
\end{equation}
Identifying $U/O=\varinjlim_n\LGr(\CC^n)$ and restricting to the based loop space, the resulting map $\Omega(U/O)\to\ZZ\times BO$ is (almost \cite[Rmk.\ 11.8]{seidelbook}) the Bott periodicity homotopy equivalence.
For Floer theory with $\ZZ$-grading and $\ZZ$ coefficients, we care just about the dimension and orientation, i.e.\ we compose the above map with $(\id,w_1):\ZZ\times BO\to\ZZ\times K(\ZZ/2,1)$.
Now restricting to the based loop space $\Omega\LGr(V)$ and applying $B$, we obtain cohomology classes on $\LGr(V)$, which are (see the related \cite[Lem.\ 11.7]{seidelbook} or \cite[Prop.\ 4.2.8]{abouzaidviterbo}) the Maslov class $\mu\in H^1(\LGr(V),\ZZ)$ and $w_2\in H^2(\LGr(V),\ZZ/2)$ (the second Stiefel--Whitney class of the tautological bundle $L\to\LGr(V)$).
The class $w_2$ is represented by a map $\LGr(V)\to K(\ZZ/2,2)$ given by the pullback of $w_2:BO\to K(\ZZ/2,2)$ under the map $\LGr(V)\to BO$ (which is \emph{well-defined up to contractible choice}) classifying the tautological bundle.
In contrast, the map $\LGr(V)\to K(\ZZ,1)=S^1$ classified by the Maslov class $\mu$ is \emph{not} well-defined up to contractible choice.
Rather, given a compatible complex structure on $V$ (a contractible choice), the Maslov class is represented by the canonical map
\begin{equation}\label{thecanonicalpartofmaslov}
\LGr(V) \to 
((\wedge_\CC^\mathrm{top} V)^{\otimes 2} \setminus 0) / \RR_{>0}
\end{equation}
given by the composition of $\wedge^{\mathrm{top}}_{\RR}: \LGr(V) \to \LGr(\wedge_\CC^\mathrm{top}V)=\RR P(\wedge_{\CC}^{\mathrm{top}}V)$ with the squaring map $\RR P(\wedge_{\CC}^{\mathrm{top}}V)\to ((\wedge_{\CC}^{\mathrm{top}}V)^{\otimes 2} \setminus 0)/ \RR_{>0}$.
Given a `basepoint' $S\in\LGr(V)$ (so $V=S\otimes_\RR\CC$), we obtain canonical identifications $\LGr(V)=\LGr(S\otimes_\RR\CC)=U(S\otimes_\RR\CC)/O(S)$ and $((\wedge_\CC^\mathrm{top} V)^{\otimes 2})/\RR_{>0}=(\CC\setminus 0)/\RR_{>0}=U(1)$, under which \eqref{thecanonicalpartofmaslov} is given by $\det^2$.
Given a map $\LGr(V)\to K(\ZZ,1)\times K(\ZZ/2,2)$ representing $(\mu,w_2)$, we obtain a $(\ZZ\times\RR P^\infty)$-bundle $\LGr(V)^\#\to\LGr(V)$.

We now globalize.
Let $X$ be a symplectic manifold, and denote by $\LGr(X)$ the bundle of Lagrangian Grassmannians of $TX$ over $X$.
There is a canonical map $\LGr(X)\to K(\ZZ/2,2)$ restricting to $w_2$ on each fiber, namely the pullback of $w_2:BO\to K(\ZZ/2,2)$ under the map classifying the tautological bundle over $\LGr(X)$.
There need not be a map $\LGr(X)\to K(\ZZ,1)$ whose restriction to each fiber represents $\mu$; the obstruction to the existence of such a map is given by $2c_1(TX)\in H^2(X,\ZZ)$ and is represented geometrically by the complex line bundle $(\wedge_\CC^\mathrm{top}TX)^{\otimes 2}$.
\emph{Grading/orientation data for $X$} is, by definition, a choice of map $X\to K(\ZZ/2,2)$ and map $\LGr(X)\to K(\ZZ,1)$ whose restriction to each fiber represents $\mu$.
The choice of map $X\to K(\ZZ/2,2)$ induces a map $\LGr(X)\to K(\ZZ/2,2)$ by pulling back and adding the canonical map $\LGr(X)\to K(\ZZ/2,2)$ restricting to $w_2$ on each fiber.
Grading/orientation data on $X$ thus induces a map $\LGr(X)\to K(\ZZ,1)\times K(\ZZ/2,2)$ whose restriction to each fiber represents $(\mu,w_2)$.
The pullback of the tautological $(\ZZ\times\RR P^\infty)$-bundle over $K(\ZZ,1)\times K(\ZZ/2,2)$ thus defines a $(\ZZ\times\RR P^\infty)$-bundle $\LGr(X)^\#\to\LGr(X)$ associated to this choice of grading/orientation data.

We now introduce Lagrangians.
Fix a choice of grading/orientation for $X$, giving $\LGr(X)^\#\to\LGr(X)$.
Given a Lagrangian $L\subseteq X$, \emph{grading/orientation data for $L$} means a lift of the canonical section of $\LGr(X)|_L$ to $\LGr(X)^\#|_L$.
It is explained in Seidel \cite[(11e)--(11l)]{seidelbook} (also reviewed in \cite[Sec.\ 3.2]{gpssectorsoc}) how such data determines graded orientation lines associated to transverse intersections of ordered pairs of Lagrangians, as well as a recipe for orienting moduli spaces of pseudo-holomorphic disks relative to these orientation lines.

For our purposes in this paper, we will induce grading/orientation data from Lagrangian polarizations.
Recall that a (Lagrangian) \emph{polarization} of a symplectic manifold $X$ is a global section of $\LGr(X)$; equivalently (up to homotopy) it is a real vector bundle $B$ with an isomorphism $B\otimes_\RR\CC=TX$.
Given such a polarization, we obtain a map $\LGr(X)\to K(\ZZ,1)$ using the section as the fiberwise basepoint, and we obtain a map $X\to K(\ZZ/2,2)$ by pulling back $w_2:BO\to K(\ZZ/2,2)$ under the map classifying $B$.
By this very definition, any Lagrangian which is everywhere tangent to the polarization admits canonical grading/orientation data (i.e.\ section of $\LGr(X)^\#|_L$).
A stable polarization (a global section of $\LGr(TX\oplus\CC^k)$ for some $k<\infty$) also induces grading/orientation data by restriction from $\LGr(TX\oplus\CC^k)$ to $\LGr(X)=\LGr(TX)$).

In the specific case of cotangent bundles $T^*M$, there is a tautological polarization given by (the tangent space of) the tautological foliation by Lagrangian fibers of the projection $T^*M\to M$; the fibers are thus equipped with canonical grading/orientation data with respect to the grading/orientation data on $T^*M$ induced by this polarization.
Conormals to open sets with smooth (or cornered) boundary also have canonical grading/orientation data, see \S\ref{conormalssec}.
We will see in Remark \ref{whytwist} and Lemma \ref{continuationballinsomething} the point in the proof of Theorem \ref{sheaffukayaequivalence} where it matters to have chosen this particular grading/orientation data on $T^*M$.

\begin{remark}
The notion of grading/orientation data given above may be reformulated as follows, which connects it to the corresponding discussion of coefficient twisting in microlocal sheaf categories as it appears in \cite{guillermou,jin}.
The stable $J$-homomorphism sends a (stable) vector bundle to (the suspension spectrum of) its Thom space, which is a family of invertible modules over the sphere spectrum.
Applying cochains, we may obtain a family of invertible dg $\ZZ$-modules.
We thus have an infinite loop map
\begin{equation}
\ZZ\times BO\xrightarrow J\Pic\mathbb S\to\Pic\ZZ
\end{equation}
sending a vector bundle $V$ to the local system $C^*(V,V\setminus 0)$ (where $\Pic$ denotes the space of invertible modules).
The invertible module $\ZZ[1]$ and the automorphism $-1$ of the invertible module $\ZZ$ together define an isomorphism of infinite loop spaces $\ZZ\times B(\ZZ/2)\xrightarrow\sim\Pic\ZZ$.
The map $\ZZ\times BO\to\Pic\ZZ=\ZZ\times B(\ZZ/2)$ is then the evident projection to $\ZZ$ times the Stiefel--Whitney class $w_1$, as considered above.
Applying $B$ as before, we obtain a map
\begin{equation}\label{stablemuwtwo}
U/O=B(\ZZ\times BO)\to B\Pic\ZZ=B\ZZ\times B^2(\ZZ/2).
\end{equation}
Now the tangent bundle of a symplectic manifold $X$ is classified by a map $X\to BU$, which we may compose with $BU\to B(U/O)$ to obtain a map $X\to B(U/O)$ which classifies the (stable) Lagrangian Grassmannian of $X$.
Composing this with $B\eqref{stablemuwtwo}$ yields a map
\begin{equation}\label{Xglobaltwist}
X\to B^2\ZZ\times B^3(\ZZ/2).
\end{equation}
Now grading/orientation data on $X$ is equivalently a null-homotopy of this map.
Indeed, the map to the second factor is canonically null-homotopic (since it by definition factors through $BU\to B(U/O)\to B^2O$ which is canonically null-homotopic), so a choice of null-homotopy of it is the same as a choice of map $X\to\Omega B^3(\ZZ/2)=B^2(\ZZ/2)$.
The map to the first factor by definition classifies $(\wedge_\CC^\mathrm{top}TX)^{\otimes 2}$, a trivialization of which is the same as a map $\LGr(X)\to K(\ZZ,1)$ whose restriction to each fiber represents $\mu$.

On a Lagrangian $L \subseteq X$, there is a tautological section of $\LGr(X)|_L$ given by 
the tangent space to the Lagrangian.  That is, the restricted map $L \to B(U/O)$ has 
a canonical null-homotopy, inducing in turn a null-homotopy of the map $L \to B^2 \ZZ \times B^3(\ZZ/2)$.
Now given grading/orientation data for $X$, grading/orientation data on $L$ is equivalently a homotopy between this null-homotopy and the restriction to $L$ of the chosen null-homotopy of \eqref{Xglobaltwist}.
Note that the space of such null-homotopies has the homotopy type of maps from $L$ to $\Omega(B^2 \ZZ \times B^3(\ZZ/2))=B\ZZ\times B^2(\ZZ/2)$, the
component group of which is $H^1(L, \ZZ)\oplus H^2(L,\ZZ/2)$.  The obstruction to the existence of grading/orientation data for $L$ thus lies in $H^1(L,\ZZ)\oplus H^2(L,\ZZ/2)$, and if this obstruction vanishes, the homotopy classes of choices of grading/orientation data for $L$ form a torsor over $H^0(L, \ZZ)\oplus H^1(L,\ZZ/2)$.

A stable polarization of $X$ gives a global section of the stable Lagrangian Grassmannian, hence a null-homotopy of $X\to B(U/O)$, hence of \eqref{Xglobaltwist}, which by definition agrees with the canonical homotopy of its restriction to any Lagrangian $L\subseteq X$ everywhere tangent to the polarization.
\end{remark}

\subsection{Wrapping exact triangle, stop removal, generation}

The fundamental ingredients underlying our work in this section are the wrapping exact triangle and its consequence stop removal, both proved in \cite{gpsdescent}.
The wrapping exact triangle can be thought of as quantifying the price of wrapping through a stop; it should be compared with Theorem \ref{thm:sheafwex}.

\begin{theorem}[{Wrapping exact triangle \cite[Thm.\ 1.10]{gpsdescent}}]\label{wrapcone}
Let $(X,\Lambda)$ be a stopped Liouville sector, and let $p\in\Lambda$ be a point near which $\Lambda$ is a Legendrian submanifold.
If $L\subseteq X$ is an exact Lagrangian submanifold and $L^w\subseteq X$ is obtained from $L$ by passing $\partial_\infty L$ through $\Lambda$ transversally at $p$ in the positive direction, then there is an exact triangle
\begin{equation}
L^w\to L\to D_p\xrightarrow{[1]}
\end{equation}
in $\W(X,\Lambda)$, where $D_p\subseteq X$ denotes the small Lagrangian disk linking $\Lambda$ at $p$ and the map $L^w\to L$ is the continuation map.
\end{theorem}

The following result about wrapped Fukaya categories is a consequence of the wrapping exact triangle, and can be compared with Theorem \ref{thm:sheafstopremoval}.

\begin{theorem}[{Stop removal \cite[Thm.\ 1.20]{gpsdescent}}]\label{stopremoval}
Let $(X,\Lambda')$ be a stopped Liouville sector, and let $\Lambda\subseteq\Lambda'$ be closed so that its complement $\Lambda'\setminus\Lambda\subseteq(\partial_\infty X)^\circ\setminus\Lambda$ is an isotropic submanifold.
Then pushforward induces an equivalence
\begin{equation}\label{stopremovalqe}
\W(X,\Lambda')/\D\xrightarrow\sim\W(X,\Lambda),
\end{equation}
where $\D$ denotes the collection of small Lagrangian disks linking (Legendrian points of) $\Lambda'\setminus\Lambda$.
\end{theorem}

We will also need to know that:

\begin{theorem}\label{generation}
The cotangent fibers split-generate $\W(T^*M)$.
\end{theorem}

\begin{proof}
    When $M$ is compact (including the case with boundary), this
    is \cite[Thm.\ 1.14 and Ex.\ 1.15]{gpsdescent}. For a
    general possibly non-compact $M$, we observe that any Lagrangian $L \in \W(T^*M)$ 
    is in the essential image of the pushforward functor $\W(T^*M_L) \to
    \W(T^*M)$, for some compact codimension zero submanifold-with-boundary $M_L
    \subseteq M$.  Now push foward the fact that $L$ is split-generated by a fiber in $\W(T^*M_L)$.
\end{proof}

\begin{remark}
In fact, the argument above shows that the fibers generate $\W(T^*M)$, however we only need split-generation.
\end{remark}

Another ingredient which proves useful in our computations is the K\"unneth theorem for Floer cohomology and wrapped Fukaya categories, also proved in \cite{gpsdescent}.

\subsection{Conormals and corners}\label{conormalssec}

Let $U\subseteq M$ be a relatively compact open set.
When $U$ has smooth boundary, we write $L_U \subseteq T^*M$ for (a smoothing of) the union of $U\subseteq M \subseteq T^* M$ with the \emph{outward} conormal along its boundary.
More generally, if $\overline{U}$ is a compact manifold-with-corners with interior $U$, then $L_{U}$ shall mean $L_{\tilde U}$ where $\tilde U$ is obtained from $U$ by smoothing out its boundary.
We also allow the degenerate case that $U$ is a point $p$ (hence, in particular, not open), in which case $L_p$ denotes the cotangent fiber over $p$.
In all of the above cases, we could also equivalently say that $L_U$ is a rounding of $\ss(\ZZ_U)$ (compare Section \ref{subsec:microsupport}).

\begin{remark}\label{cornerconventions}
Various natural constructions, such as the cornering operations of Section \ref{microlocalapprox}--\ref{constructibleconormals} and taking products, introduce corners.
By convention, we conflate such cornered objects with their smoothings, usually without comment.
The choice of this smoothing is always a contractible choice which is ultimately irrelevant.
\end{remark}

Recall that for a Lagrangian $L$, we write $L^+$ for an unspecified small positive Reeb pushoff of $L$, and $L^-$ for a negative pushoff.
Thus if $U$ is an relatively compact open set with smooth boundary and $U^+$ denotes its $\epsilon$ neighborhood in some metric, then $L_{U^+} = L_U^+$.
That is, positive Reeb flow pushes outward conormals out.
In particular, $(T^*_p M)^+ = L_{B_{\epsilon}(p)}$.

Each $L_U$ is exact and possesses canonical grading/orientation data: the codimension zero inclusion $U\subseteq L_U$ is a
homotopy equivalence, and $U$ is a codimension zero submanifold of, and
thereby inherits all of this data from, the zero section.
The grading/orientation data for the zero section arises from the canonical homotopy from its tangent bundle $TM\subseteq T(T^*M)|_M$ to the family of tangent spaces of the cotangent fibers $T^*M\subseteq T(T^*M)|_M$ (which is the chosen polarization of $T^*M$) given by $e^{i\theta}T^*M$ for $\theta\in[0,\pi/2]$ (where $J$ is chosen so that $J(T^*M)=TM$).

\subsection{Floer cohomology between conormals of balls and stable balls}\label{secstableballs}

We study here the Floer cohomology between conormals of open sets with smooth boundary (though recall Remark \ref{cornerconventions} about implicit smoothing of corners).
The assertion that an open set with smooth boundary is a ball shall mean that its \emph{closure} is diffeomorphic to the standard closed unit ball.
Note that a small positive pushoff of the cotangent fiber $L_p$ over a point $p$ is the conormal of a small open ball around $p$, so we may substitute `point' in place of `open ball' in many of the statements below.

\begin{lemma} \label{continuationballs}
Let $U,V\subseteq M$ be balls with $\overline U\subseteq V$.
Then $HF^*(L_V, L_U) = \ZZ$, and is canonically generated by the continuation element, lying in degree zero.
\end{lemma}

\begin{proof}
There is a positive isotopy from $L_U^+$ to $L_V$ in the complement of $\partial_\infty L_U$. Hence $HF^*(L_V, L_U) = HF^*(L_U, L_U)$ (compare \cite[Lem.\ 3.21]{gpssectorsoc}), but this latter group (which is isomorphic to $H^*(L_U) = H^*(U)$) is generated by its identity element.
\end{proof}

\begin{lemma}\label{continuationiso}
Let $U,V\subseteq M$ be balls with $\overline U\subseteq V$, and let $\overline V\subseteq W\subseteq M$.
The continuation map $HF^*(L_W, L_V) \to HF^*(L_W, L_U)$ (i.e.\ multiplication by the continuation element in $HF^*(L_V, L_U)$) is an isomorphism.
\end{lemma}

\begin{proof}
The positive isotopy $L_U \leadsto L_V$ takes place in the complement of $\partial_{\infty} L_W$, hence induces an isomorphism $HF^*(L_W, L_V)=HF^*(L_W, L_U)$ which agrees with multiplication by the continuation element by \cite[Lem.\ 3.26]{gpssectorsoc}.
\end{proof}

\begin{lemma} \label{continuationballinsomething}
Let $U\subseteq M$ be a ball, and let $\overline U\subseteq W\subseteq M$.
There is a canonical isomorphism $HF^*(L_W, L_U) = \ZZ$, with respect to which the continuation maps from Lemma \ref{continuationiso} act as the identity on $\ZZ$.
\end{lemma}

\begin{proof}
The groups $HF^*(L_W,L_p)$ form a local system of $p\in W$ \cite[Lem.\ 3.21]{gpssectorsoc}.
It suffices to show that this local system is canonically isomorphic to the constant local system $\ZZ$.
Indeed, by Lemma \ref{continuationiso} we have a canonical isomorphism $HF^*(L_W,L_U)=HF^*(L_W,L_p)$ for any $p\in U$, which is compatible with the local system structure of $HF^*(L_W,L_p)$ by \cite[Lem.\ 3.26]{gpssectorsoc}.

The assertion that the local system $p\mapsto HF^*(L_W,L_p)$ is canonically trivialized is local on $W$, so let us consider $p$ varying only in a small ball $U\subseteq W$.
Now the isomorphism $HF^*(L_W,L_p)=HF^*(L_W,L_{p'})$ for nearby points $p$ and $p'$ from \cite[Lem.\ 3.21]{gpssectorsoc} is induced by a diffeomorphism supported in $U$ sending $p$ to $p'$.
As such, it is sent under the identifications $HF^*(L_W,L_p)=HF^*(L_U,L_p)$ and $HF^*(L_W,L_{p'})=HF^*(L_U,L_{p'})$ (coming from the fact that, in both cases, they are generated by the `same' Lagrangian intersection and have no Floer differential, and we have chosen the `same' grading/orientation data on $L_W$ and $L_U$ in \S\ref{conormalssec}) to the corresponding isomorphism $HF^*(L_U,L_p)=HF^*(L_U,L_{p'})$.
Now $HF^*(L_U,L_p)$ and $HF^*(L_U,L_{p'})$ are canonically $\ZZ$ by Lemma \ref{continuationiso}, and the isomorphism between them acts as the identity on $\ZZ$ by \cite[Lem.\ 3.26]{gpssectorsoc}.
\end{proof}

\begin{lemma} \label{deformtozero}
Let $V$ be an open set with smooth boundary, and let $U$ be a $\epsilon$-ball centered at a point on $\partial V$. 
Then  $HF^*(L_U, L_V) = 0 = HF^*(L_V, L_U)$.
\end{lemma}

\begin{proof}
During the obvious isotopy of $U$ outward to become disjoint from $V$, their conormals never intersect at infinity.
\end{proof}

By a \emph{stable ball}, we mean a compact manifold-with-boundary which is contractible; the statement that an open set with smooth boundary is a stable ball shall mean its closure is a stable ball.
The reason we study stable balls is that we do not know how to prove that for a subanalytic Whitney triangulation, the `inward cornering' in the sense of Section \ref{constructibleconormals} of an open star is a ball; it is, however, obviously a stable ball.

To compute Floer cohomology between conormals of stable balls, we reduce to the case of conormals to balls by stabilizing (i.e.\ taking their product with conormals to standard balls in $\RR^k$) and appealing to the K\"unneth theorem for Floer cohomology.
We begin by showing that the stabilization of a stable ball is indeed a ball, thus justifying the name.
This uses the following famous corollary of the $h$-cobordism theorem:

\begin{theorem}
A stable ball of dimension $\ge 6$ with simply connected boundary is a ball.\qed
\end{theorem}

\begin{corollary}\label{stableballisstableball}
Let $M$ be a stable ball.  Then $M\times I^k$ is a ball provided $\dim M+k\ge 6$ and $k\geq 1$.
\end{corollary}

\begin{proof}
We just need to check that the boundary of $M\times I^k$ is simply connected.
It suffices to show that for any stable ball $N$ of dimension $\ge 2$, the boundary of $N\times I$ is simply connected.
The boundary of $N\times I$ is, up to homotopy, two copies of $N$ glued along their common boundary.
Since $N$ is contractible, the fundamental group of this gluing vanishes provided $\partial N$ is connected.
If $\partial N$ were disconnected, then by Poincar\'e duality, the cohomology group $H^{\dim N-1}(N)$ would be nonzero, which contradicts contractibility as $\dim N\ge 2$.
\end{proof}

\begin{proposition} \label{fakecontinuation} 
Let $U,V\subseteq M$ be stable balls with $\overline U\subseteq V$.
Then $HF^*(L_V, L_U) = \ZZ$, and it is equipped with a canonical generator $1_{VU}$ which we call the \emph{pseudo-continuation element} (it coincides with the usual continuation map when $U$ and $V$ are balls).
The pseudo-continuation elements are closed under composition: for any triple of stable balls $U,V,W\subseteq M$ with $\overline U\subseteq V$ and $\overline V\subseteq W$, we have $1_{WV} 1_{VU} = 1_{WU}$.
\end{proposition}

\begin{proof}
We multiply by $L_U,L_V$ by $L_{B_{1}(0)},L_{B_{2}(0)} \subseteq T^*\RR^k$ where $k$ is sufficiently large to guarantee that $U\times B_1(0)$ and $V\times B_2(0)$ are balls by Corollary \ref{stableballisstableball}.
By the K\"unneth formula for Floer cohomology (see e.g.\ \cite[Lem.\ 8.3]{gpsdescent}) and freeness of $HF^*(L_{B_{2}(0)}, L_{B_{1}(0)})$ (compare \cite[Rmk.\ 8.4]{gpsdescent}), we have
$$HF^*(L_V \times L_{B_{2}(0)}, L_U \times L_{B_{1}(0)}) = HF^*(L_V, L_U) \otimes HF^*(L_{B_{2}(0)}, L_{B_{1}(0)}) = HF^*(L_V, L_U).$$
On the other hand, by the result for balls Lemma \ref{continuationballs}, we have
$$HF^*(L_V \times L_{B_{2}(0)}, L_U \times L_{B_{1}(0)})=HF^*(L_{V \times B_{2}(0)}, L_{U \times B_{1}(0)})=\ZZ.$$
After arguing that the above identification is compatible with rounding of corners, this defines the canonical generator $1_{VU}\in HF^*(L_V, L_U)$.
The proof that $1_{WV} 1_{VU} = 1_{WU}$ is the same: stabilize to reduce to the corresponding fact for honest continuation maps.
\end{proof}

In order to make sense of the next corollary, recall that Lemma \ref{continuationballs} and Proposition \ref{fakecontinuation} continue to apply in the limiting situation in which $L_U$ is replaced by a cotangent fibre $L_p = T_p^*M$ for some $p \in V$.

\begin{corollary} \label{squish}
Let $U \subseteq M$ be any stable ball.  
Then the pseudo-continuation element $L_U\to T^*_pM$ is an isomorphism in $\W(T^*M)$ for any point $p\in U$. 
\end{corollary}

\begin{proof}
Note that this corollary is tautologically true if $U$ is a ball, as genuine continuation elements by definition are isomorphisms in the wrapped Fukaya category.
By pushing forward, it suffices to treat the case $M=U^+$. 
Appealing to the fully faithful K\"unneth embedding 
(see \cite[Thm.\ 1.5]{gpsdescent}) 
$\W(T^*U^+) \otimes \W(T^*I^k) \hookrightarrow\W(T^*(U^+\times I^k))$, it further suffices to show the result after taking (the image under this embedding of) the product of this pseudo-continuation element with the continuation element $L_{I^k} \to \fiber$ (which is an isomorphism in $\W(T^* I^k)$).
The pseudo-continuation element $L_U\to\fiber$ is, by definition, sent by this stabilization to the continuation element $L_{U\times I^k}\to\fiber$ (which is defined since the stabilized stable ball $U^+\times I^k$ is a ball). This latter map is an isomorphism so we are done.
\end{proof}

Here is an improved version of Lemma \ref{continuationballinsomething}:

\begin{lemma} \label{fakecontinuationballinsomething}
For any stable ball $U\subseteq M$ whose closure is contained in $W\subseteq M$, there is canonical isomorphism $HF^*(L_W, L_U) = \ZZ$, such that for an inclusion $\overline U\subseteq V$ of such stable balls, the pseudo-continuation map $\ZZ=HF^*(L_W,L_V)\to HF^*(L_W,L_U)=\ZZ$ (multiplication by the pseudo-continuation element in $HF^*(L_V,L_U)$) is the identity on $\ZZ$.
\end{lemma}

\begin{proof}
Stabilize to reduce to Lemma \ref{continuationballinsomething}.
\end{proof}

There is similarly an improved version of Lemma \ref{deformtozero}:

\begin{lemma} \label{fakedeformtozero}
Let $V$ be an open set with smooth boundary, and let $U$ be stable ball such that
$U \cap \partial V$ is also a stable ball. 
Then  $HF^*(L_U, L_V) = 0 = HF^*(L_V, L_U)$.
\end{lemma}

\begin{proof}
Stabilization (multiplying both $U$ and $V$ by $I^k$) and appealing to the K\"{u}nneth formula for Floer cohomology (as in the proof of Lemma \ref{fakecontinuation}) reduces this proof to Lemma \ref{deformtozero} (note that $U \cap \partial V$ necessarily divides $U$ into two stable balls).
\end{proof}

A more subtle result about stable balls is the following, which will be important later:

\begin{proposition} \label{relsquish}
Let $X^m \subseteq Y^n$ be an inclusion of stable balls, with $\partial X \subseteq \partial Y$.
Assume there exists another stable ball (with corners) $Z^{m +1} \subseteq Y^n$ such that $\partial Z$ is the union of $X$ with a smooth submanifold
of $\partial Y$.
Then the pseudo-continuation element $L_Y \to L_{B_\epsilon(x)}$ is an isomorphism in $\W(T^*Y, N^*_\infty X)$ for any $x \in X$.
\end{proposition}

\begin{proof}
By stabilization, we reduce to the case that $X$, $Y$, and $Z$ are all balls.
This implies that, up to diffeomorphism, everything is standard: $Y$ is the unit ball, $X$ is the intersection of $Y$ with a linear subspace, and $Z$ is the intersection of $Y$ with a linear halfspace.
Indeed, since $X$ and $Z$ are balls, we can use $Z$ to push $X$ to $Z\cap\partial Y$, thus showing that $X$ is simply a slight inward pushoff of the ball $Z\cap\partial Y\subseteq\partial Y$.

By definition, the pseudo-continuation element becomes the continuation element under stabilization (i.e., after multiplying by the continuation isomorphism from the conormal of a large ball to that of a small ball as in the proof of Corollary \ref{squish}). Once everything is standard, the continuation map $L_Y \to L_{B_\epsilon(x)}$ is an isomorphism, since there is a positive isotopy $L_{B_\epsilon(x)}\leadsto L_Y$ disjoint from $N^*_\infty X$ at infinity.
\end{proof}

We will apply Proposition \ref{relsquish} when (before rounding) $X$ is a simplex in a triangulation, $Y$ is its star, and $Z$ is any simplex containing $X$ of dimension one larger.

\subsection{Fukaya categories of conormals to stars}\label{fukayacalculationssec}

Let $\SSS$ be a Whitney stratification of $M$ by locally closed smooth submanifolds.
For an $\SSS$-constructible open set $U$, we abuse notation and denote by $L_U$ the conormal of the inward cornering of $U$ with respect to $\SSS$ in the sense of Section \ref{constructibleconormals}.
More precisely, we have $L_U:=L_{U^{-\underline\epsilon}}$ for $\underline\epsilon\in\RR^\SSS_{>0}$ satisfying $\epsilon_\alpha\leq f((\epsilon_\beta)_{\beta\subsetneqq\alpha})$.
Recall that $L_{U^{-\underline\epsilon}}$ is disjoint from $N^*\SSS$ at infinity but limits to it as $\underline\epsilon\to 0$.

\begin{lemma} \label{corneringstopped}
Let $L$ be any Lagrangian disjoint at infinity from $N^* \SSS$.  Then for all $\underline \epsilon>0$ sufficiently small, 
$CF^*(L_{U^{- \underline \epsilon}}, L) \xrightarrow{\sim} CW^*(L_{U^{- \underline \epsilon}}, L)_{N^*_\infty\SSS}$.
\end{lemma}

\begin{proof}
Taking $\underline\epsilon\to 0$ is a positive wrapping of $L_{U^{- \underline \epsilon}}$ (compare Section \ref{conormalssec}) which converges to (while remaining disjoint from) $N^*_\infty\SSS$.
It is thus cofinal by Lemma \ref{cofinalitycriterion}.
\end{proof}

Now assume further that $\SSS$ is a triangulation, and let us consider the conormals to (the inward cornernings of) open stars $L_{\star(s)}\in\W(T^*M,N^*_\infty\SSS)$.
Since $\star(s)$ is contractible, $L_{\star(s)}$ is the conormal to a stable ball (Lemma \ref{perturbationdiffeo}), and hence the results about stable balls from Section \ref{secstableballs} above apply, allowing us to deduce the following:

\begin{proposition}\label{homsbetweenstars}
We have
\begin{equation}
HW^*(L_{\star(s)},L_{\star(t)})_{N^*_\infty\SSS}=\begin{cases}\ZZ&t\to s\\0&\textrm{otherwise}\end{cases}
\end{equation}
generated in the former case by the pseudo-continuation element.
\end{proposition}

\begin{proof}
Fix a small $\underline\epsilon>0$ and let $\underline\delta\to 0$.
By Lemma \ref{corneringstopped}, the wrapped Floer cohomology $HW^*(L_{\star(s)},L_{\star(t)})$ is calculated by $HF^*(L_{\star(s)^{-\underline\delta}},L_{\star(t)^{-\underline\epsilon}})$.

Now if $t\to s$, then $\star(t)^{-\underline\epsilon}\subseteq \star(s)^{-\underline\delta}$ is an inclusion of stable balls, so by Proposition \ref{fakecontinuation} $HF^*(L_{\star(s)^{-\underline\delta}},L_{\star(t)^{-\underline\epsilon}})=\ZZ$ is generated by the pseudo-continuation element.

Now suppose that $t\nrightarrow s$.
If $\star(s)\cap\star(t)=\emptyset$, then the desired vanishing is trivial.
Otherwise, we have $\star(s)\cap\star(t)=\star(r)$ where $r$ is the simplex spanned by the union of the vertices of $s$ and $t$.
To show the desired vanishing, it suffices by Proposition \ref{fakedeformtozero} to show that $\star(t)^{-\underline\epsilon}\cap\partial\star(s)^{-\underline\delta}$ is a stable ball, which is the content of Lemma \ref{strangesimplexintersectionisstableball}. 
\end{proof}

It will be convenient to have another perspective on the objects $L_{\star(s)}$.
Let $L_s$ denote the conormal to a small ball centered at any point on the stratum $s$ (this conormal is disjoint from $N^*_\infty\SSS$ at infinity by Whitney's condition (b), compare with the proof of Lemma \ref{tubetransverse}); this is well defined up to Lagrangian isotopy.
One reason the $L_s$ are nice to consider is the following calculation:

\begin{lemma}\label{Lshoms}
For any $\SSS$-constructible open set $U$, we have
\begin{equation}
HW^*(L_U,L_s)_{N^*_\infty\SSS}=\begin{cases}\ZZ&\star(s)\subseteq U\\ 0&\textrm{otherwise.}\end{cases}
\end{equation}
\end{lemma}

\begin{proof}
We calculate using Lemma \ref{corneringstopped}.
If $s$ is a stratum in the interior of $U$, then the ball centered at $s$ is contained in $U$, and in paticular as $L_s$ is a small positive pushoff of a cotangent fiber to a point in $s$, one can arrange for there to be a single intersection point between $L_{U^{- \underline \epsilon}}$ (for small $\underline \epsilon$) and $L_s$. Hence $HW^*(L_U,L_s)_{N^*_\infty\SSS} = HF^*(L_{U^{- \underline \epsilon}}, L_s)  = \ZZ$.
If $s$ is a stratum not contained in the closure of $U$, then the morphism space obviously vanishes since the two Lagrangians ($L_{U^{- \underline \epsilon}}$ for any $\epsilon$ and $L_s$) are disjoint.

%%%%%%%%%%% IMPORTANT NOTE TO COPYEDITORS %%%%%%%%%%%%%
%%%%%%%%%%%%%%%%%%%%%%%%%%%%%%%%%%%%%%%%%%%%%%%%%%%%%%%
%%% If any of the figures needs to be resized, by far the simplest way to do this is
%%% to adjust the "scale" parameter in the \begin{tikzpicture} command.  That is,
%%% \begin{tikzpicture}[scale=\textwidth/10cm]
%%% should be replaced with
%%% \begin{tikzpicture}[scale=\textwidth/???cm]
%%% where ??? is whatever number you want to make the figure(s) the correct size.
%%% A *bigger* number means the figure will be *smaller*.

\usetikzlibrary{intersections}
\begin{figure}[htbp]
\centering
\begin{tikzpicture}[scale=\textwidth/10cm]
\begin{scope}
\clip(1,0)--(0.8,0)arc(0:180:0.8)--(-1,0)--(-1,1)--(1,1)--cycle;
\fill[lightgray!30!white](1,0.2)--(-1,0.2)--(-1,1)--(1,1)--cycle;
\end{scope}
\path[name path = LUhoriz](1,0.2)--(-1,0.2);
\path[name path = LUcircleright](0.8,0)arc(0:90:0.8);
\path[name path = LUcircleleft]((-0.8,0)arc(180:90:0.8);
\path[name intersections={of = LUhoriz and LUcircleright, by=a}];
\path[name intersections={of = LUhoriz and LUcircleleft, by=b}];
\draw[fill=lightgray!30!white](0,0)circle(0.7);
\draw(-1,0)to(1,0);
\draw(a)--(1,0.2);
\draw(b)--(-1,0.2);
\begin{scope}
\clip(-1,0.2)--(1,0.2)--(1,1)--(-1,1)--cycle;
\draw(0,0)circle(0.8);
\end{scope}
\draw[fill](0,0)circle(0.03);
\node at(0.1,-0.15){$q$};
\node at(-0.9,-0.5){$B_\delta(q)$};
\end{tikzpicture}
\quad
\begin{tikzpicture}[scale=\textwidth/10cm]
\begin{scope}
\clip(1,0)--(0.6,0)arc(0:180:0.6)--(-1,0)--(-1,1)--(1,1)--cycle;
\fill[lightgray!30!white](1,0.2)--(-1,0.2)--(-1,1)--(1,1)--cycle;
\end{scope}
\path[name path = LUhoriz](1,0.2)--(-1,0.2);
\path[name path = LUcircleright](0.6,0)arc(0:90:0.6);
\path[name path = LUcircleleft]((-0.6,0)arc(180:90:0.6);
\path[name intersections={of = LUhoriz and LUcircleright, by=a}];
\path[name intersections={of = LUhoriz and LUcircleleft, by=b}];
\draw[fill=lightgray!30!white](0,0)circle(0.7);
\draw(-1,0)to(1,0);
\draw(a)--(1,0.2);
\draw(b)--(-1,0.2);
\begin{scope}
\clip(-1,0.2)--(1,0.2)--(1,1)--(-1,1)--cycle;
\draw(0,0)circle(0.6);
\end{scope}
\end{tikzpicture}
\quad
\begin{tikzpicture}[scale=\textwidth/10cm]
\begin{scope}
\clip(1,0)--(0.35,0)arc(0:180:0.35)--(-1,0)--(-1,1)--(1,1)--cycle;
\fill[lightgray!30!white](1,0.2)--(-1,0.2)--(-1,1)--(1,1)--cycle;
\end{scope}
\path[name path = LUhoriz](1,0.2)--(-1,0.2);
\path[name path = LUcircleright](0.35,0)arc(0:90:0.35);
\path[name path = LUcircleleft]((-0.35,0)arc(180:90:0.35);
\path[name intersections={of = LUhoriz and LUcircleright, by=a}];
\path[name intersections={of = LUhoriz and LUcircleleft, by=b}];
\draw[fill=lightgray!30!white](0,0)circle(0.7);
\draw(-1,0)to(1,0);
\draw(a)--(1,0.2);
\draw(b)--(-1,0.2);
\begin{scope}
\clip(-1,0.2)--(1,0.2)--(1,1)--(-1,1)--cycle;
\draw(0,0)circle(0.35);
\end{scope}
\end{tikzpicture}
\quad
\begin{tikzpicture}[scale=\textwidth/10cm]
\fill[lightgray!30!white](1,0.2)--(-1,0.2)--(-1,1)--(1,1)--cycle;
\draw[fill=lightgray!30!white](0,0)circle(0.7);
\draw(-1,0)to(1,0);
\draw(-1,0.2)to(1,0.2);
\end{tikzpicture}
\caption{The isotopy from $U^{-\underline\epsilon,\delta+\eta}$ to $U^{-\underline\epsilon,\delta-\eta}$ to $U^{-\underline\epsilon}$.}\label{figureLshoms}
\end{figure}

Finally, we claim that if $s$ is a stratum on the boundary of $U$, the morphism space still vanishes.
To prove this, it suffices to construct a cofinal wrapping of $L_U$ which begins disjoint from $L_s$ and remains forever disjoint from $\partial_\infty L_s$ (see \cite[Lem.\ 3.26]{gpssectorsoc}).
Such an isotopy is illustrated in Figure \ref{figureLshoms}, which we now define precisely.
Fix a point $q\in s$, and consider the stratification $\SSS_q$ obtained from $\SSS$ by declaring $\{q\}$ to be its own stratum.
Now the inward cornering with respect to $\SSS_q$ may be denoted $U^{-\underline\epsilon,\delta}$ for $\underline\epsilon\in\RR^\SSS_{>0}$ and $\delta>0$ the parameter associated to the new stratum $\{q\}$.
Fix $\underline\epsilon$ and take $\delta$ to zero, and note that Proposition \ref{wrapnothitballnew} implies this gives an isotopy of $U^{-\underline\epsilon,\delta}$ whose conormals remain disjoint at infinity from $N^*\SSS$.
Once $\delta=0$, we just have $U^{-\underline\epsilon}$, whose conormal has cofinal wrapping by taking $\underline\epsilon\to 0$.
Now take $L_s$ to be the conormal of $B_\delta(q)$ and take the isotopy of $L_U$ to be given the conormal of the isotopy
\begin{equation}
U^{-\underline\epsilon,\delta+\eta}\leadsto U^{-\underline\epsilon,\delta-\eta}\leadsto U^{-\underline\epsilon}
\end{equation}
illustrated in Figure \ref{figureLshoms}, followed by isotoping $U^{-\underline\epsilon}$ by taking $\underline\epsilon\to 0$.
Corollary \ref{neighborhoodcorners} implies that this isotopy remains disjoint at infinity from $L_s=N^*B_\delta(q)$ except possibly at $U^{-\underline\epsilon,\delta}$, but these do not intersect at infinity due to their coorientations being opposite.
\end{proof}

Another reason that the $L_s$ are nice to consider is that we can show using the wrapping exact triangle and stop removal that they (split-)generate:

\begin{proposition}\label{smallstarsgen}
The objects $L_s$ for strata $s$ split-generate $\W(T^*M,N^*_\infty\SSS)$.
\end{proposition}

\begin{proof}
Denote by $\SSS_{\leq k}$ the stratification where we keep all strata of dimension $\leq k$ and combine all other strata into a single top stratum.
We consider the sequence of categories
\begin{multline}
\W(T^*M,N^*_\infty\SSS)=\W(T^*M,N^*_\infty\SSS_{\leq n-1})\to\W(T^*M,N^*_\infty\SSS_{\leq n-2})\to\cdots\\\cdots\to\W(T^*M,N^*_\infty\SSS_{\leq 0})\to\W(T^*M).
\end{multline}
Each of these functors removes a locally closed Legendrian submanifold $N^*_\infty\SSS_{\leq k}\setminus N^*_\infty\SSS_{\leq k-1}$, and thus by stop removal Theorem \ref{stopremoval}, is the quotient by the corresponding linking disks.

The linking disk at a point on $N^*_\infty\SSS_{\leq k}\setminus N^*_\infty\SSS_{\leq k-1}$ can be described as follows.
A point on $N^*_\infty\SSS_{\leq k}\setminus N^*_\infty\SSS_{\leq k-1}$ is simply a point $x$ on a $k$-dimensional stratum together with a covector $\xi$ at $x$ conormal to the stratum.
Consider a small ball $B_a$ centered at $x$, and consider a smaller ball $B_b\subseteq B_a$ disjoint from the stratum containing $x$.
There is a family of balls starting at $B_a$ and shrinking down to $B_b$ whose boundaries are tangent to the stratum containing $x$ only at $(x,\xi)$.
It follows from the wrapping exact triangle Theorem \ref{wrapcone} that the cone on the resulting continuation map $L_{B_a}\to L_{B_b}$ is precisely the linking disk at $(x,\xi)$.

We have thus shown that the linking disks to each locally closed Legendrian $N^*_\infty\SSS_{\leq k}\setminus N^*_\infty\SSS_{\leq k-1}$ are generated by the objects $L_s$.
By Theorem \ref{generation} above, these $L_s$ also split-generate the final category $\W(T^*M)$.
We conclude that the $L_s$ split-generate $\W(T^*M,N^*_\infty\SSS)$, as the quotient by all of them vanishes.
\end{proof}

\begin{remark}
A small variation on the above proof and an appeal to \cite[Thm.\ 1.14]{gpsdescent} shows that the objects $L_s$ in fact generate $\W(T^*M,N^*_\infty\SSS)$.
We give the weaker argument above to minimize the results we need to appeal to.
\end{remark}

\begin{proposition}\label{bigsmallstars}
The pseudo-continuation element $L_{\star(s)}\to L_s$ is an isomorphism in $\W(T^*M,N^*_\infty\SSS)$.
\end{proposition}

\begin{proof}
We proceed by induction on the codimension of $s$.
When $s$ has codimension zero, the desired statement follows from Corollary \ref{squish}.

Now suppose that $s$ has positive codimension.
For any $t$ of strictly smaller codimension than $s$, we have $\Hom(L_{\star(t)},L_{\star(s)})=0$ by Proposition \ref{homsbetweenstars} and $\Hom(L_{\star(t)},L_s)=0$ by Lemma \ref{Lshoms}.

Now by the discussion in the proof of Proposition \ref{smallstarsgen}, the functor
\begin{equation}
\W(T^*M,N^*_\infty\SSS)\to\W(T^*M,N^*_\infty\SSS_{\leq\dim s})
\end{equation}
quotients by cones of $L_t$ for $t$ of strictly smaller codimension than $s$.
By the induction hypothesis and the calculations of the previous paragraph, such cones are left-orthogonal to $L_s$ and $L_{\star(s)}$.
Hence it suffices to check that $L_{\star(s)}\to L_s$ is an isomorphism in $\W(T^*M,N^*_\infty\SSS_{\leq\dim s})$.

Finally, we observe that $L_{\star(s)}\to L_s$ is an isomorphism in $\W(T^*M,N^*_\infty\SSS_{\leq\dim s})$ by Proposition \ref{relsquish}.
Namely, we take $Y=\star(s)^-$, $X=s\cap\star(s)^-$, and $Z=t\cap\star(s)^-$ for any simplex $t$ containing $s$ and of one higher dimension.
\end{proof}

\begin{remark}
For a `smooth triangulation' $\SSS$, there is an obvious positive isotopy from $L_s$ to $L_{\star(s)}$ disjoint from $N^*_\infty\SSS$ (thus proving Proposition \ref{bigsmallstars} in this case), obtained by expanding a small ball centered at a point on $s$ to $\star(s)$, keeping the boundary transverse to the strata of $\SSS$.
We do not know whether this proof can be generalized from smooth triangulations to subanalytic Whitney triangulations.
\end{remark}

\subsection{Functors from poset categories to Fukaya categories}

\begin{definition}\label{posettocohofukaya}
Let $M$ be a manifold with Whitney stratification $\SSS$, and let
$U: \Pi \to \Op_\SSS(M)$ be a map from a poset $\Pi$ to the poset of $\SSS$-constructible open subsets of $M$.
Suppose further that each $U(\pi)^-$ (from Section \ref{constructibleconormals}) is a stable ball.
Define a functor on cohomology categories
\begin{equation}
H^*F_U:\ZZ[\Pi] \to H^*\W(T^*M, N^*_\infty\SSS)^\op
\end{equation}
by $H^*F_U(\pi):=L_{U(\pi)}$ and $H^*F_U(1_{\pi,\pi'})=1_{U(\pi'),U(\pi)}\in HW^*(L_{U(\pi')}, L_{U(\pi)})$ is the pseudo-continuation element.
\end{definition}

\begin{remark}\label{whytwist}
Note that the definition of $H^*F_U$ depends on having defined pseudo-continuation elements with the compatibility properties from Proposition \ref{fakecontinuation}, which in turn depends on having equipped the cotangent fibers of $T^*M$ with continuously varying grading/orientation data (compare Section \ref{gradorsec}).
For general grading/orientation data on $T^*M$, it may not be possible to define continuously varying grading/oriention data on the cotangent fibers, in which case we could only define $H^*F_U$ to respect composition up to sign.
The resulting $2$-cocycle, or rather its class in $H^2(N\Pi,\ZZ/2)$, would represent (the pullback of) the obstruction in $H^2(M,\ZZ/2)$ to choosing continuously varying relative $\Pin$-structures on the cotangent fibers.
\end{remark}

\begin{proposition}\label{ffhenceainf}
For any functor $f:\ZZ[\Pi]\to H^*\C$ such that $H^*\C(f(x),f(y))$ is concentrated in degree zero for every pair $x\leq y\in\Pi$, there exists an $\ainf$ functor $F:\ZZ[\Pi]\to\C$ with $H^*F=f$.
Moreover, given any two $\ainf$ functors $F,G:\ZZ[\Pi]\to\C$ such that $H^*\C(F(x),G(y))$ is concentrated in degree zero for every pair $x\leq y\in\Pi$ and a natural transformation $t:f\to g$, the space of $\ainf$ natural transformations $T:F\to G$ with $H^*T=t$ is contractible.
\end{proposition}

\begin{proof}
We show existence of a lift $F$ by induction.
Lift the action on objects arbitrarily.
Take $F^1$ to be any map in the correct cohomology class $H^*F^1=f$.
Having chosen $F^1,\ldots,F^{k-1}$, the existence of an $F^k$ satisfying the $\ainf$ functor equations is equivalent to a certain element of
\begin{equation}\label{functorobstructiongroup}
\prod_{\pi_0,\ldots,\pi_k\in\Pi}\Hom(\ZZ[\Pi](\pi_0,\pi_1)\otimes\cdots\otimes\ZZ[\Pi](\pi_{k-1},\pi_k),\C(F(\pi_0),F(\pi_k))),
\end{equation}
(namely the sum of all the terms of the $\ainf$ functor equations with $k$ inputs except for those involving $F^k$) being a coboundary.
This element is always a cocycle due to $F^1,\ldots,F^{k-1}$ satisfying the $\ainf$ functor equations, so it suffices to show that its class in cohomology vanishes.
The cohomology of \eqref{functorobstructiongroup} is of course simply
\begin{equation}
\prod_{\pi_0\leq\cdots\leq\pi_k\in\Pi}H^*\C(F(\pi_0),F(\pi_k)),
\end{equation}
which is concentrated in degree zero by hypothesis.
The obstruction class thus vanishes for degree reasons for $k\geq 3$.
For $k=2$, the obstruction class measures the failure of $H^*F$ to respect composition, so by hypothesis the obstruction vanishes in this case as well.
We conclude that there always exists an $F^k$ compatible with the previously chosen $F^1,\ldots,F^{k-1}$.
(Compare \cite[Lem.\ 1.9]{seidelbook}, where a variant on this obstruction theory argument is explained in more detail.)

To construct a natural transformation \cite[(1d)]{seidelbook} $T:F\to G$ with $H^*T=t$, first pick some $T^0$ lifting $t$.
Given $T^0,\ldots,T^{k-1}$, the obstruction to the existence of $T^k$ is a degree $1-k$ cohomology class in
\begin{equation}\label{nattransobstructiongroup}
\prod_{\pi_0\leq\cdots\leq\pi_k\in\Pi}\C(F(\pi_0),G(\pi_k)).
\end{equation}
It hence vanishes for degree reasons for $k\geq 2$, and for $k=1$ it measures the failure of $H^*T$ to respect morphisms, hence vanishes in this case as well.

Contractibility of the space of natural transformations $T$ with $H^*T=t$ is, concretely, the assertion that the complex of pre-natural transformations $\Hom(F,G)$ is acyclic in negative (cohomological) degree and that any two natural transformations (i.e.\ degree zero cocycles) $T$ and $T'$ with $H^*T=t=H^*T'$ are cohomologous.
In both cases, we should produce a pre-natural transformation $Q$ of degree $-g<0$ with prescribed value of $dQ$.
Such $Q=(Q^0,Q^1,\ldots)$ is again constructed by induction.
The existence of $Q^k$ then comes down to the vanishing in degree $1-k-g$ of the cohomology of \eqref{nattransobstructiongroup} for all $g\geq 1$ and $k\geq 0$ (except for $g=1$ and $k=0$, for which the relevant obstruction measures the failure of the desired value $T-T'$ of $dQ$ to vanish on $H^*\C$, hence vanishes by the assumption $H^*T=t=H^*T'$).
\end{proof}

\begin{remark}
The assertion of Proposition \ref{ffhenceainf} over a field $k$ (instead of over $\ZZ$) is a straightforward consequence of the fact that any $\ainf$ category $\C$ over $k$ is quasi-isomorphic to an $\ainf$ category $\tilde{\C}$ with vanishing differential, sometimes called a \emph{minimal model} of $\C$.
In this case, the essential image of $f$ inside $\tilde{\C}$ would be necessarily concentrated in degree zero on the chain level, hence have vanishing higher order $\ainf$ structure maps for degree reasons (and hence this essential image inside $\tilde{\C}$, equivalently $\C$, is \emph{formal}).
It follows that any functor on cohomology categories $k[\Pi]\to H^*\tilde{\C}$ would lift tautologically to an $\ainf$ functor by taking all higher operations to vanish.
The proof above bypasses the question of the existence of minimal models over $\ZZ$.
\end{remark}

\begin{corollary}\label{posettofukaya}
There is a unique up to contractible choice $\ainf$ functor
\begin{equation}
F_U:\ZZ[\Pi] \to \W(T^*M,N^*_\infty\SSS)^\op
\end{equation}
lifting the functor on cohomology categories from Defintion \ref{posettocohofukaya}.
\end{corollary}

\begin{proof}
By Corollary \ref{corneringstopped}, the wrapped Floer cohomology group $HW^*(L_{U(\pi)}, L_{U(\pi')})$ is simply the Floer cohomology of two nested stable balls, which is $\ZZ$ by Proposition \ref{fakecontinuation}.
Thus Proposition \ref{ffhenceainf} is applicable.
\end{proof}

\begin{remark}
To extend Corollary \ref{posettofukaya} to the Fukaya category with a $\ZZ/N$-grading, we would need to add to the requirement that $F$ (and the natural transformations $F_1\to F_2$) must lift to $\ZZ$-graded categories locally (the $\ZZ$-grading is only defined locally, over any contractible open subset of $M$).
\end{remark}

For the next corollary, let us denote by $\HHH^*$ the functor from $\Mod\ZZ$ to the category of graded abelian groups given by taking the cohomology of \emph{objects} of $\Mod\ZZ$.
The functor $\HHH^*$ factors through, but does not coincide with, the functor $H^*:\Mod\ZZ\to H^*\Mod\ZZ$ which exists for any $\ainf$ category in place of $\Mod\ZZ$ (and which takes cohomology of \emph{morphisms}).

\begin{corollary}\label{minimalmodelsformodules}
Consider functors $F:\ZZ[\Pi]\to\Mod\ZZ$ such that $\HHH^*F(x)$ is free and concentrated in degree zero for all $x\in\Pi$.
Given any two such functors $F$ and $G$ and a natural transformation $t:\HHH^*F\to\HHH^*G$, the space of natural transformations $T:F\to G$ with $\HHH^*T=t$ is contractible.
In particular, any such functor $F$ is quasi-isomorphic to $i\HHH^*F:\ZZ[\Pi]\to\Mod\ZZ$, namely the composition of $\HHH^*F$ with the inclusion $i$ of free abelian groups into $\Mod\ZZ$ (as complexes concentrated in degree zero).
\end{corollary}

\begin{proof}
We first argue that if $P,Q\in\Mod\ZZ$ are such that $\HHH^*P$ and $\HHH^*Q$ are free and concentrated in degree zero, then the natural map
\begin{equation}\label{homologyhom}
H^*\Hom(P,Q)\to\Hom(\HHH^*P,\HHH^*Q)
\end{equation}
is an isomorphism.
Since $\HHH^*P$ is projective, there is a quasi-isomorphism $\HHH^*P\to P$ (and the same for $Q$).
It follows that there is a quasi-isomorphism of chain complexes $\Hom(P,Q)=\Hom(\HHH^*P,\HHH^*Q)$ (homomorphisms in $\Mod\ZZ$).
Since $\HHH^*P$ and $\HHH^*Q$ are projective and concentrated in degree zero, the complex $\Hom(\HHH^*P,\HHH^*Q)$ is (quasi-isomorphic to) homomorphisms of abelian $\ZZ$ modules $\HHH^*P\to\HHH^*Q$ concentrated in degree zero; thus \eqref{homologyhom} is an isomorphism as desired.

Now suppose $F$ and $G$ are as in the statement and a natural transformation $t:\HHH^*F\to\HHH^*G$ is given.
Since $\HHH^*F$ and $\HHH^*G$ are free and concentrated in degree zero, we see from \eqref{homologyhom} that the data of $t$ is equivalent to the data of a natural transformation $\bar t:H^*F\to H^*G$.
Now \eqref{homologyhom} also implies that the hypotheses of Proposition \ref{ffhenceainf} are satisfied, so the space of natural transformations $T$ with $H^*T=\bar t$ (which is, as just noted, equivalent to $\HHH^*T=t$) is contractible.

It is immediate from the definition that $\HHH^*F=\HHH^*i\HHH^*F$, so the final statement follows from the first.
\end{proof}

\begin{definition}\label{Fdefn}
For a Whitney triangulation $\SSS$, let
\begin{equation}
F_\SSS:\ZZ[\SSS]\to\W(T^*M,N^*_\infty \SSS)^\op
\end{equation}
denote the functor induced from Definition \ref{posettocohofukaya} and Corollary \ref{posettofukaya} by the map associating to each simplex of $\SSS$ its open star.
\end{definition}

\begin{theorem}\label{Fequivalence}
The functor $F_\SSS$ is a Morita equivalence.
\end{theorem}

\begin{proof}
Proposition \ref{homsbetweenstars} shows is full faithfulness of $F_\SSS$, and Propositions \ref{smallstarsgen} and \ref{bigsmallstars} together show essential surjectivity of $F_\SSS$ (after passing to $\Perf$).
\end{proof}

We now show that $F_\SSS$ is compatible with refinement (compare Lemma \ref{posetmeaningpushforward}):

\begin{theorem}\label{Frefinement}
For $\SSS'$ a refinement of $\SSS$, the following diagram commutes:
\begin{equation}
\begin{tikzcd}
\ZZ[\SSS']\ar{r}{F_{\SSS'}} \ar{d}[swap]{r} & \W(T^*M,N^*_\infty \SSS')^\op\ar{d}{\rho}\\
\ZZ[\SSS]\ar{r}{F_\SSS} &  \W(T^*M,N^*_\infty \SSS)^\op
\end{tikzcd}
\end{equation}
up to contractible choice.
\end{theorem}

\begin{proof}
There are two functors $\rho\circ F_{\SSS'}$ and $F_{\SSS}\circ r$ from $\ZZ[\SSS']$ to $\W(T^*M,N^*_\infty \SSS)$.
By Proposition \ref{ffhenceainf}, it suffices to define a canonical natural isomorphism between the induced functors on cohomology categories.
It is most natural to define this canonical natural isomorphism in the direction $F_{\SSS}\circ r\Longrightarrow\rho\circ F_{\SSS'}$.

To a stratum $s$ of $\SSS'$, the composition $F_{\SSS}\circ r$ associates the conormal of $\star_\SSS(r(s))$, and the composition $\rho\circ F_{\SSS'}$ associates the conormal of $\star_{\SSS'}(s)$.
Since $\star_\SSS(r(s))\supseteq\star_{\SSS'}(s)$ is an inclusion of stable balls, we may consider by Proposition \ref{fakecontinuation} the pseudo-continuation element from one to the other.
Since pseudo-continuation elements are closed under composition by Proposition \ref{fakecontinuation}, it is easy to check that this defines a natural transformation $H^*(F_{\SSS}\circ r)\Longrightarrow H^*(\rho\circ F_{\SSS'})$.

This natural transformation is in fact a natural isomorphism since the natural maps from both $L_{\star_\SSS(r(s))}$ and $L_{\star_{\SSS'}(s)}$ to $L_s=L_{r(s)}$ are isomorphisms by Proposition \ref{bigsmallstars}.
\end{proof}

\subsection{In conclusion}

\begin{theorem}\label{fukcasting}
The functor $\Lambda\mapsto\Perf\W(T^*M,\Lambda)^\op$ is a microlocal Morse theatre in the sense of Definition \ref{microlocalmorsetheatre}, which casts the linking disks at smooth points of $\Lambda$ as the Morse characters.
\end{theorem}

\begin{proof}
Definition \ref{Fdefn} and Theorems \ref{Fequivalence} and \ref{Frefinement} give the identification between $\SSS\mapsto\Perf\SSS$ and $\SSS\mapsto\Perf\W(T^*M,N^*_\infty\SSS)^\op$ via the functors $F_\SSS$.

Stop removal Theorem \ref{stopremoval} says that $\W(T^*M,\Lambda')\to\W(T^*M,\Lambda)$ is the quotient by the linking disks at the smooth points of $\Lambda'\setminus\Lambda$.
It therefore suffices to show that the Morse characters are precisely (isomorphic to) these linking disks.

Recall from Definition \ref{morsecharacter} that a Morse character at a smooth point $p\in\Lambda$ is defined as follows.
We choose a function $f:M\to\RR$ and an $\epsilon>0$ such that $f^{-1}(-\infty, \epsilon)$ is relatively compact, $f$ has no critical values in $[-\epsilon,\epsilon]$ and $df$ is transverse to $\RR_{>0} \times \Lambda$ over $f^{-1}[-\epsilon,\epsilon]$, intersecting it only at $p$ (where $f$ vanishes).
We also choose a subanalytic Whitney triangulation $\SSS$ such that $\Lambda\subseteq N^*_\infty\SSS$ and $f^{-1}(-\infty,-\epsilon)$ and $f^{-1}(-\infty,\epsilon)$ are constructible.
The Morse character associated to these choices is then defined as the image in $\W(T^*M,\Lambda)$ of
\begin{equation}
\cone(1_{f^{-1}(-\infty, -\epsilon)} \to 1_{f^{-1}(-\infty, \epsilon)})\in\Perf\SSS=\Perf\W(T^*M,N^*_\infty\SSS)^\op
\end{equation}
where the morphism $1_{f^{-1}(-\infty, -\epsilon)} \to 1_{f^{-1}(-\infty, \epsilon)}$ is (the linearization of) the canonical one from \eqref{uniquenattrans}.
To show that this cone is indeed
sent to the linking disk at $p$ in $\W(T^*M,\Lambda)$, we will make use of the wrapping exact triangle Theorem \ref{wrapcone}, which says that the linking disk at $p$ is the cone of the continuation map associated to any positive isotopy of Lagrangians in $T^*M$ which crosses $\Lambda$ exactly once transversely at $p$.
Specifically, there is an obvious positive isotopy from the conormal of $f^{-1}(-\infty, -\epsilon)$ to the conormal of $f^{-1}(-\infty, \epsilon)$, namely $f^{-1}(-\infty,t)$ for $t\in[-\epsilon,\epsilon]$, since $f$ has no critical values in the interval $[-\epsilon,\epsilon]$; the cone of the associated continuation element in $\W(T^*M, \Lambda)$ is thus the desired linking disk.

The conormals of $f^{-1}(-\infty, \pm \epsilon)$ are not themselves objects of the wrapped Fukaya category associated to the triangulation $\W(T^*M,N^*_\infty\SSS)$, seeing as they by definition touch the stop. However, in Section \ref{fukayacalculationssec} we studied the conormals of the inward cornerings of $\SSS$-constructible open sets $U$, which we denoted $L_U:= L_{U^{-\underline\epsilon}}$. These Lagrangians $L_{f^{-1}(-\infty, \pm \epsilon)}$ are thus, in particular, objects of $\W(T^*M, N^*_\infty \SSS)$, and Lemma \ref{smoothcornering} provides an isotopy between them and the (usual) conormals of $f^{-1}(-\infty, \pm \epsilon)$ which takes place in the complement of $\Lambda$, thus inducing an isomorphism in $\W(T^*M,\Lambda)$.
Therefore to complete the argument, it suffices to show that:
\begin{enumerate}
\item\label{oldA}$F_{\SSS}(1_{f^{-1}(-\infty, \pm\epsilon)}) \in\Perf\W(T^*M,N^*_\infty\SSS)$ is isomorphic to $L_{f^{-1}(-\infty, \pm\epsilon)}$, and 
\item\label{oldB}the canonical morphism $1_{f^{-1}(-\infty, -\epsilon)} \to 1_{f^{-1}(-\infty, \epsilon)}$ is sent by $F_{\SSS}$ and the isomorphisms \ref{oldA} to an element in $HW^*(L_{f^{-1}(-\infty, \epsilon)}, L_{f^{-1}(-\infty, -\epsilon)} )_{N^* \SSS}$ whose image in $\W(T^*M, \Lambda)$ is the continuation element associated to the natural positive isotopy between the conormals of $f^{-1}(-\infty, \pm \epsilon)$.
\end{enumerate}

Regarding \ref{oldA}, let us establish the more general assertion that $F_{\SSS}(1_W)$ is isomorphic to $L_W$ for any relatively compact open $\SSS$-constructible set $W$.
Since $F_\SSS$ is a Morita equivalence, it suffices by Yoneda to show that the pullback $F_\SSS^*L_W=CW^*(L_W,F_\SSS(-))$ is isomorphic in $\Mod\SSS$ to $1_W$.
Using Lemma \ref{corneringstopped} and Lemma \ref{fakecontinuationballinsomething} (for the case $\star(s)\subseteq W$) and Proposition \ref{bigsmallstars} and Lemma \ref{Lshoms} (for the case $\star(s)\nsubseteq W$), we have canonical isomorphisms
\begin{equation}\label{indicatorobjectwiseiso}
HW^*(L_W,L_{\star(s)})_{N^*_\infty\SSS}=\begin{cases}\ZZ& \star(s)\subseteq W\\
0&\textrm{otherwise}\end{cases}=1_W(s).
\end{equation}
This identifies $F_\SSS^*L_W=1_W$ objectwise (i.e.\ identifies their evaluations at every $s\in\SSS$).
To identify $F_\SSS^*L_W=1_W$ as modules, it suffices by Corollary \ref{minimalmodelsformodules} to compare the action of morphisms $t\to s$ for $\star(t)\subseteq\star(s)\subseteq W$ at the level of cohomology (since the category $\ZZ[\SSS]$ and both modules are cohomologically concentrated in degree zero, there is no room for any higher homotopies).
In other words, we should show that the map
\begin{equation}
HW^*(L_W,L_{\star(s)})_{N^*_\infty\SSS}\to HW^*(L_W,L_{\star(t)})_{N^*_\infty\SSS}
\end{equation}
given by multiplication with the pseudo-continuation element in $HW^*(L_{\star(s)},L_{\star(t)})_{N^*_\infty\SSS}$ acts as the identity map on $\ZZ$ under the isomorphism \eqref{indicatorobjectwiseiso}, and this is precisely what is stated in Lemma \ref{fakecontinuationballinsomething}.
This completes the proof of \ref{oldA}.

Turning to \ref{oldB}, first we note that, as $\partial_\infty L_{f^{-1}(-\infty, \epsilon)}$ falls immediately into the stop $N^*_\infty\SSS$ (Lemma \ref{corneringstopped}), we have
\begin{equation}\label{placewhereneedtoagree}
HW^*(L_{f^{-1}(-\infty, \epsilon)},L_{f^{-1}(-\infty,-\epsilon)})_{N^*_\infty\SSS} = HF^*(L_{f^{-1}(-\infty, \epsilon)},L_{f^{-1}(-\infty,-\epsilon)}).
\end{equation}
In turn the isotopies of Lemma \ref{smoothcornering} between the (usual) conormals of $f^{-1}(-\infty, \pm \epsilon)$ and the inward cornerings $L_{f^{-1}(-\infty, \pm \epsilon)}$ do not cross each other at infinity, hence induce an isomorphism between $HF^*$ of the conormals with the above $HF^*$ group of their inward cornerings. These
transfer the continuation element associated to the isotopy of conormals from $f^{-1}(-\infty,-\epsilon)$ to $f^{-1}(-\infty,\epsilon)$ to an element of \eqref{placewhereneedtoagree} which we will call the \emph{cornered continuation element} (multiplication by which is called the cornered continuation map).
The image of the cornered continuation element in $\W(T^*M, \Lambda)$ is, by definition, the continuation element between the conormals of $f^{-1}(-\infty, \pm \epsilon)$ (as the isomorphisms between $HF^*$ are compatible with the map to $HW^*$).
It thus suffices to show that the cornered continuation element in \eqref{placewhereneedtoagree} is the image of the canonical map $1_{f^{-1}(-\infty, -\epsilon)} \to 1_{f^{-1}(-\infty, \epsilon)}$ under $F_{\SSS}$ and the isomorphisms of \ref{oldA}.
Equivalently, we are to show that the canonical map $1_{f^{-1}(-\infty, -\epsilon)} \to 1_{f^{-1}(-\infty, \epsilon)}$ agrees under the isomorphisms $HW^*(L_{f^{-1}(-\infty,\pm \epsilon)}, F_{\SSS}(-)) = 1_{f^{-1}(-\infty,\pm \epsilon)}(-)$ from \ref{oldA} with the pulled back cornered continuation map.
Again by Corollary \ref{minimalmodelsformodules}, it suffices to make this comparison at the level of cohomology.
That is, we are to show that multiplication by the cornered continuation map
\begin{equation}
HW^*(L_{f^{-1}(-\infty,-\epsilon)},L_{\star(s)})_{N^*_\infty\SSS}\to HW^*(L_{f^{-1}(-\infty,\epsilon)},L_{\star(s)})_{N^*_\infty\SSS}
\end{equation}
acts as the identity on $\ZZ$ under the isomorphisms \eqref{indicatorobjectwiseiso} (for $\star(s)\subseteq f^{-1}((-\infty,-\epsilon))$).
Now, the isomorphisms of both sides with $\ZZ$ coming from Lemma \ref{fakecontinuationballinsomething} are compatible with pseudo-continuation elements, so by Proposition \ref{bigsmallstars} it is equivalent to show that multiplication by the cornered continuation map
\begin{equation}\label{twomapshere}
HF^*(L_{f^{-1}(-\infty,-\epsilon)},L_s)\to HF^*(L_{f^{-1}(-\infty,\epsilon)},L_s)
\end{equation}
acts as the identity on $\ZZ$ under the isomorphisms of Lemma \ref{continuationballinsomething} (for $s\subseteq f^{-1}((-\infty,-\epsilon))$) (we replaced $HW^*$ with $HF^*$ using Lemma \ref{corneringstopped}).
By the definition given in Lemma \ref{continuationballinsomething}, the `identity on $\ZZ$' map \eqref{twomapshere} is simply the identity map on the single Floer generator we get when unperturbing $L_s$ back to be the cotangent fiber of a point on $s$.
Now the cornered continuation map \eqref{twomapshere} agrees by the last part of \cite[Lem.\ 3.26]{gpssectorsoc} with the isomorphism \eqref{twomapshere} from \cite[Lem.\ 3.21]{gpssectorsoc} associated to the isotopy from $L_{f^{-1}(-\infty,-\epsilon)}$ to $L_{f^{-1}(-\infty,\epsilon)}$.
Since this isotopy takes place far away from $L_s$, it also by definition acts as the identity on the single Floer generator of both sides.
\end{proof}

\begin{proposition}\label{wfunctorial}
For any analytic open inclusion of analytic manifolds $M'\hookrightarrow M$, the pushforward functor $\W(T^*M',\Lambda')\to\W(T^*M,\Lambda)$ for subanalytic singular isotropics with $\Lambda'\supseteq\Lambda\cap S^*M'$ defines a morphism of microlocal Morse theatres.
\end{proposition}

\begin{proof}
The reasoning of Theorem \ref{Frefinement} applies without change.
\end{proof}

\section{Examples} \label{sec:corollaries}

\subsection{Cotangent bundles} 
\label{cotangent}

Let $M$ be a smooth manifold (assumed connected for sake of notation).
The cotangent fiber $F_q\in\W(T^*M)$ generates by Abouzaid \cite{abouzaidcriterion,abouzaidcotangent} when $M$ is closed and by \cite[Thm.\ 1.14]{gpsdescent} in general.

When $M$ is closed, Abbondandolo--Schwarz \cite{abbondandoloschwarz} and Abouzaid \cite{abouzaidtwisted} calculated the endomorphism algebra of the fiber as $CW^*(F_q,F_q)=C_{-*}(\Omega_q M)$ (using relative $\Pin$ structures as in Section \ref{gradorsec}).
The present Theorem \ref{sheaffukayaequivalence} (which does not depend on any of \cite{abbondandoloschwarz, abouzaidtwisted, abouzaidcotangent,abouzaidcriterion}) gives a proof of this fact for all (not necessarily closed) $M$:

\begin{corollary}
There is a quasi-isomorphism $CW^*(F_q, F_q) = C_{-*}(\Omega_q M)$.
Moreover if $M \subseteq N$ is a codimension zero inclusion, there is a commutative diagram
\begin{equation}
\begin{tikzcd}
CW^*(F_q, F_q)_{T^*M} \ar{r} \ar{d} & C_{-*}(\Omega_q M) \ar{d} \\
CW^*(F_q, F_q)_{T^*N} \ar{r} & C_{-*}(\Omega_q N)
\end{tikzcd}
\end{equation}
where the left hand vertical arrow is covariant inclusion and the right hand vertical arrow is induced by pushforward of loops.
\end{corollary}

\begin{proof}
Note that there exists a real analytic structure on $M$ whose induced smooth structure agrees with the given one.
Taking $\Lambda = \emptyset$ in Theorem \ref{sheaffukayaequivalence} gives $\Perf \W(T^*M) = \Sh_\emptyset(M)^c$.
It is well known that $\Sh_\emptyset(M)^c=\Perf C_{-*}(\Omega_qM)$, for example because both are the global sections of the constant cosheaf of linear categories with costalk $\Perf\ZZ$.
Indeed, $U\mapsto U$ is a cosheaf of spaces, equivalently of $\infty$-groupoids, which upon linearizing yields $U\mapsto\Perf C_{-*}(\Omega U)$, and $U\mapsto\Sh_\emptyset(U)^c$ is a cosheaf since $U\mapsto\Mod\Sh_\emptyset(U)^c=\Sh_\emptyset(U)$ is a sheaf.

We may derive the more precise assertion that $C_{-*}(\Omega_qM)$ is endomorphisms of the cotangent fiber by following a fiber through the equivalence, e.g.\ by considering the inclusion of the cotangent bundle of a disk, or equivalently by introducing a stop along the conormal of the boundary of a disk and then removing it.
\end{proof}

\subsection{Plumbings}

Many authors have studied Fukaya categories of plumbings \cite{abouzaidplumbing,abouzaidsmith, etgu-lekili} and their sheaf counterparts \cite{bezrukavnikov-kapranov}.
Here we compute the wrapped category of a plumbing.

Let $\Pi_{2n}$ be the Liouville pair $(\CC^n, \partial_\infty (\RR^n \cup i\RR^n))$; we term it the plumbing sector.
Plumbings are formed by taking a manifold $M$ (usually disconnected) with spherical boundary $\partial M = \coprod S^{n-1}$, 
and gluing the Liouville pair $(T^*M, \partial M)$ to some number of plumbing sectors along the spheres.

One can model the wrapped Fukaya category of the plumbing sector directly in sheaf theory: we can view it as the pair 
$(T^*\RR^n, N^*_\infty\{0\})$, and the category $\Sh_{N^*_\infty\{0\}}(\RR^n)$ has a well known description in terms of the Fourier
transform as described in \cite{bezrukavnikov-kapranov}.
This category is equivalent to $\W(\Pi_{2n})$ by Theorem \ref{sheaffukayaequivalence}.
To apply the gluing results of \cite{gpsdescent}, however, we need to know how the wrapped Fukaya categories of the two boundary sectors include, which is slightly more than what Theorem \ref{sheaffukayaequivalence} tells us.
Hence we give a direct computation of the wrapped Fukaya category of the plumbing sector.
Take a positive Reeb pushoff of the boundary of a cotangent fiber in $T^*\RR^n$, so it is now the outward conormal of a small ball.
Deleting the original cotangent fiber, we obtain the Liouville sector $T^* S^{n-1}  \times \A_2$ where $\A_2$ denotes the Liouville sector $(\CC,\{e^{2\pi ik/3}\}_{k=0,1,2}\infty)$.
We can get back to the plumbing sector $\Pi_{2n}$ by adding back the missing fiber, which amounts to attaching a Weinstein handle along one of the boundary sectors $T^*(S^{n-1} \times I)$.
We may thus deduce from \cite[Thm.\ 1.28, Thm.\ 1.5, and Cor.\ 1.18]{gpsdescent} that:

\begin{lemma}
$\W(\Pi_{2n})$ is Morita equivalent to
\begin{equation}
\colim(\Perf(\bullet) \leftarrow \Perf C_*(\Omega S^{n-1}) \to \Perf(\bullet \to \bullet) \otimes \Perf C_*(\Omega S^{n-1})).
\end{equation}
\end{lemma}

Gluing in the remaining manifolds, we conclude: 

\newlength\Awidth
\settowidth\Awidth{A}
\begin{corollary}
The wrapped Fukaya category of a plumbing is Morita equivalent to the colimit of the diagram 
\begin{equation}
\hspace{-2\Awidth}
\begin{tikzcd}[column sep = tiny]
\coprod \Perf(\bullet) & & \\
\coprod \Perf C_*(\Omega S^{n-1}) \ar{u} \ar{r} &  \coprod \Perf(\bullet \to \bullet)  \otimes \Perf C_*(\Omega S^{n-1})  \\
& \coprod \Perf C_*(\Omega S^{n-1}) \ar{u} \ar{r} & \coprod \Perf C_*(\Omega M_i)
\end{tikzcd}
\hspace{-2\Awidth}
\end{equation}
where $M_i$ are the components of $M$.
\end{corollary}

\subsection{Proper modules and infinitesimal Fukaya categories} \label{NZ}

Recall that for a dg or $A_\infty$ category $\C$, we write $\Prop \C := \Fun(\C, \Perf\ZZ)$ for the category 
of proper (aka pseudo-perfect) modules.  It is immediate from our main result that 
$\Prop \Sh_\Lambda(M)^c = \Prop \W(T^*M, \Lambda)^\op$.

Recall from Corollary \ref{propersheaves} that any proper $\Sh_\Lambda(M)^c$-module is representable by an object of $\Sh_\Lambda(M)$ with perfect stalks, i.e.\ a constructible sheaf in the classical sense.
Let us describe some objects in the Fukaya category $\W(T^*M, \Lambda)$ which necessarily give rise to proper modules (and thus to sheaves on $M$ with perfect stalks, microsupported inside $\Lambda$).

\begin{definition}\label{forwardstoppedsubcategory}
For any stopped Liouville manifold $(X,\Lambda)$, we define the \emph{forward stopped subcategory} $\W^{\epsilon}(X,\Lambda)$ to be 
the full subcategory of $\W(X,\Lambda)$ generated by Lagrangians which admit a positive wrapping into $\Lambda$, meaning $\partial_\infty L$ becomes eventually contained in arbitrarily small neighborhoods of $\Lambda$.
By Lemma \ref{cofinalitycriterion}, such a wrapping is necessarily cofinal.
\end{definition}

\begin{example}
If $\Lambda$ admits a ribbon $F$ (or, alternatively, is itself equal to a Liouville hypersurface $F$) then $\W^{\epsilon}(X,\Lambda)$ contains all Lagrangians whose boundary at infinity is contained in a neighborhood of a small negative Reeb pushoff of $\Lambda$ (or $F$).
\end{example}

\begin{example}
All compact (exact) Lagrangians are contained in $\W^{\epsilon}(X,\Lambda)$, as their boundary at infinity $\emptyset$ is wrapped into $\Lambda$ by the trivial wrapping.
\end{example}

\begin{proposition}
All objects of $\W^{\epsilon}(X,\Lambda)$ co-represent proper modules over $\W(X,\Lambda)$; that is, the restriction of the Yoneda embedding $\W(X,\Lambda)\hookrightarrow\Mod\W(X,\Lambda)^\op$ to $\W^\epsilon(X,\Lambda)$ has image contained in $\Prop\W(X,\Lambda)^\op$.
\end{proposition}

\begin{proof}
Morphisms in the wrapped category can be computed by cofinally positively wrapping the first factor.
Any $L \in \W^{\epsilon}(X,\Lambda)$ admits such a wrapping $\{L_t\}_{t\geq 0}$ which converges at infinity to $\Lambda$.
It follows that after some time $t$, its boundary at infinity stays disjoint at infinity from $K$, and hence $CW^*(L,K)=CF^*(L_t,K)$ for sufficiently large $t$.
\end{proof}

\begin{corollary}\label{epsilonproper}
The equivalence $\Perf\W(T^*M, \Lambda)^\op=\Sh_\Lambda(M)^c$ sends $\W^{\epsilon}(T^*M, \Lambda)$ into $\Prop\Sh_\Lambda(M)^c$.
\end{corollary}

Recall that for a Whitney triangulation $\SSS$, the category $\W(T^*M, N^*_\infty \SSS)^\op$ is Morita equivalent to $\ZZ[\SSS]$, hence smooth and proper.
The generators $L_{\star(s)}$ of $\W(T^*M, N^*_\infty \SSS)$ used to prove this equivalence were shown in that proof to lie in $\W^\epsilon(T^*M, N^*_\infty \SSS)$, so we have:

\begin{proposition}\label{wrappedisinf}
For a Whitney triangulation $\SSS$, the inclusion $\W^{\epsilon}(T^*M, N^*_\infty \SSS) \subseteq \W(T^*M, N^*_\infty \SSS)$ is a Morita equivalence.\qed
\end{proposition}

\begin{remark}
Corollary \ref{epsilonproper} is very similar to the original Nadler--Zaslow correspondence \cite{nadlerzaslow}, restricted to Lagrangians with fixed asymptotics.
To be more precise, recall that Nadler--Zaslow wish to consider an infinitesmially wrapped Fukaya category $\W^{\inf}(T^*M,\Lambda)$ of Lagrangians `asymptotic at infinity to $\Lambda$' and then show it is equivalent to a category of sheaves on $M$ with microsupport inside $\Lambda$.

If $\Lambda$ is a smooth Legendrian and $\W^{\inf}(T^*M,\Lambda)$ is defined to consist of Lagrangians which are conical at infinity, ending inside $\Lambda$, then there is a fully faithful embedding $\W^{\inf}(T^*M,\Lambda)\hookrightarrow\W^\epsilon(T^*M,\Lambda)$, sending a Lagrangian ending inside $\Lambda$ to its small negative pushoff, as this pushoff  tautologically wraps positively back into $\Lambda$.
Hence Corollary \ref{epsilonproper} recovers a version of \cite{nadlerzaslow} when $\Lambda$ is a smooth Legendrian.
One can certainly imagine constructing such an embedding $\W^{\inf}(T^*M,\Lambda)\hookrightarrow\W^\epsilon(T^*M,\Lambda)$ for more general (e.g.\ subanalytic isotropic) $\Lambda$, provided one is given a definition of $\W^{\inf}(T^*M,\Lambda)$ for such $\Lambda$.
\end{remark}

\begin{remark}\label{infinitesimalproper}
We \emph{do not} know when $\W^{\epsilon}(T^*M, \Lambda)^\op \hookrightarrow \Prop \Sh_\Lambda(M)^c$ is a Morita equivalence. 
Note that the assertion of such an equivalence (for $\W^{\inf}(T^*M,\Lambda)^\op$) \emph{is not} made in \cite{nadlermicrolocal}, although that work is occasionally misquoted to suggest that it is.
What is actually said is that one can get all objects of $\Prop \Sh_\Lambda(M)^c$ from twisted complexes of objects of $\W^{\inf}(T^*M,\Lambda')^\op$ for a \emph{possibly larger} $\Lambda'$ which, as twisted complexes, pair trivially 
with all Lagrangians contained in a neighborhood of $\Lambda'\setminus\Lambda$.  Such Lagrangians
might be said to be ``Floer theoretically supported away from $\Lambda'\setminus\Lambda$''.
\end{remark}

To make a precise statement along the lines of Remark \ref{infinitesimalproper},
realizing a version of the Nadler--Zaslow equivalence, we have:

\begin{proposition}\label{nzrecovery}
If $\SSS$ is any subanalytic Whitney triangulation of compact $M$ with $\Lambda\subseteq N^*_\infty\SSS$, and $\D$ denotes the collection of linking disks
to smooth points of $N^*_\infty \SSS \backslash \Lambda$, then
\begin{equation}
\Prop\Sh_\Lambda(M)^c = \Prop \W(T^*M, \Lambda)^\op = (\Tw\W^{\epsilon}(T^*M, N^*_\infty \SSS)^\op)_{\Ann(\D)}
\end{equation}
where $\Tw$ denotes twisted complexes (i.e.\ any model for the the
pre-triangulated, non idempotent-completed, hull), and the subscript
$\Ann(\D)$ indicates taking the full subcategory of objects annihilated by $CW^*(-,D)=0$ for all $D\in\D$.
\end{proposition}

\begin{proof}
For such an $\SSS$, the functor $j: \W(T^*M, N^*_\infty \SSS) \to \W(T^*M,
\Lambda)$ is the quotient by $\D$ by Theorem \ref{stopremoval}.
Pullback of modules under any localization is a fully faithful embedding, identifying the category of modules over the localized category with the full subcategory of modules over the original category which annihilate the objects quotiented by (see Section \ref{quotients} and \cite[Lem.\ 3.12 and 3.13]{gpssectorsoc}).
Properness of a module is also clearly equivalent to properness of its pullback.
We thus conclude that
\begin{equation}
j^*: \Prop \W(T^*M, \Lambda)^\op \hookrightarrow \Prop \W(T^*M, N^*_\infty \SSS)^\op
\end{equation}
embeds the former as the full subcategory of the latter annihilating $\D$.

Now $\W(T^*M, N^*_\infty \SSS)$ (Morita equivalent to $\Perf\SSS^\op$ by Proposition \ref{Fequivalence}) is
smooth and proper by Lemma \ref{exceptionalpropersmooth} (since $M$ is compact and thus there are finitely many simplices).
Hence $\Prop \W(T^*M, N^*_\infty \SSS)^\op = \Perf  \W(T^*M,
N^*_\infty \SSS)^\op = \Perf \W^{\epsilon}(T^*M, N^*_\infty \SSS)^\op$ (by Proposition \ref{wrappedisinf}).
Finally, we observe that idempotent completion is
unecessary by Lemma \ref{exceptionalidempotentcomplete}, as $\Perf\SSS$ has a generating exceptional collection.
\end{proof}

\begin{remark}\label{noncompactnzrecovery}
For non-compact $M$, the same proof implies that
\begin{equation}
\Prop\Sh_\Lambda(M)^c = (\Prop \W^{\epsilon}(T^*M, N^*_\infty\SSS)^\op)_{\Ann(\D)} \supseteq (\Perf\W^{\epsilon}(T^*M, N^*_\infty\SSS)^\op)_{\Ann(\D)}
\end{equation}
but the inclusion is not generally an equality.
Though, if at least $\Lambda$ is compact, a similar argument should relate proper modules annihilating co-representatives of the stalks at infinity with the annihilator in $\Tw\W^{\epsilon}(T^*M, N^*_\infty \SSS)^\op$ of $\D$ and the fibers at infinity.
\end{remark}

\begin{example}
Let us explain how our `stopped' setup can be used to make ordinary (not wrapped) Floer cohomology calculations using sheaves.
Suppose given two Lagrangians $L, K \subseteq T^*M$ for which $\Lambda:=\partial_\infty L\cup\partial_\infty K$ is subanalytic.
We are interested in computing $HF^*(L^+, K)$.  Thus consider the wrapped category $\W(T^*M, \Lambda)$ and small negative pushoffs $L^-,K^-\in\W(T^*M, \partial_\infty L \cup \partial_\infty K)$, and observe that
$$HF^*(L^+,K)=HW^*(L^-,K^-)_\Lambda.$$
By our main result, the right hand side can be computed as $\Hom(\F_K,\F_L)$ in the sheaf category $\Sh_\Lambda(M)$, provided we can determine the
sheaves $\F_L$ and $\F_K$ to which $L^-$ and $K^-$ are sent by our Theorem \ref{sheaffukayaequivalence}.

Here we make only a few observations regarding how to determine these sheaves.  
Because linking disks go to microstalks and $L^-,K^-$ are forward stopped, we can see immediately
that $\F_L, \F_K$ have microstalk $\ZZ$ along the respective loci $\partial_\infty L,\partial_\infty K\subseteq\Lambda$.
For the same reason, for $p$ away from the front projection of $\Lambda=\partial_\infty L\cup\partial_\infty K$, we have
$$\F_L|_p \cong CF^*(L, T_p^* M) \qquad \qquad \F_K|_p \cong CF^*(K, T_p^* M).$$
In some cases, e.g.\ in case that $L$ intersects every cotangent fiber either once or not at all, this data already suffices to determine 
$\F_L$.   In particular, this situation occurs in \cite{STWZ}, where 
sheaf calculations are made exhibiting cluster transformations arising from comparing different fillings
of Legendrian knots.  The present discussion suffices to translate those calculations into calculations in Lagrangian Floer theory.
\end{example}

\subsection{Legendrians and constructible sheaves}
\label{aas}

\begin{corollary}
Let $\Lambda \subseteq J^1 \RR^n \subseteq S^*\RR^{n+1}$ be a smooth compact Legendrian.
Let $D=D_1\sqcup\cdots\sqcup D_n$ be a disjoint union of linking disks at distinct points of $\Lambda$, at least one on each connected component.
Consider the algebra
\begin{equation}
\A_\Lambda := CW^*(D, D)_{T^* \RR^{n+1}, \Lambda}=\bigoplus_{i,j=1}^nCW^*(D_i,D_j)_{T^*\RR^{n+1},\Lambda}.
\end{equation}
Then $\Mod\A_\Lambda^\op$ is equivalent to the category $\Sh_\Lambda(T^* \RR^{n+1})_0$ of sheaves microsupported inside $\Lambda$ and with vanishing stalk at infinity.

This equivalence identifies the microstalk along $\Lambda$ near $D_i$ with the direct summand of the forgetful functor $\Mod\A_\Lambda^\op\to\Mod\ZZ$ corresponding to the idempotent $e_i:=\1_{D_i}\in CW^*(D_i,D_i)\subseteq\A_\Lambda$ ($e_1,\ldots,e_n$ are orthogonal idempotents summing to the identity).
Hence $\Prop\A_\Lambda^\op$ is equivalent to the subcategory of $\Sh_\Lambda(T^*\RR^{n+1})_0$ of objects with perfect microstalks along $\Lambda$ (or, equivalently, with perfect stalks).
\end{corollary}

\begin{proof}
Our generation results \cite[Thm.\ 1.14]{gpsdescent} imply that $\W(T^*\RR^{n+1}, \Lambda)$ is generated by $D_1,\ldots,D_n$ and a cotangent fiber $F$ near infinity.
Because we assume that $\Lambda \subseteq J^1 \RR^n$, the cotangent fiber at negative (in the last coordinate) infinity 
can be cofinally positively wrapped without intersecting $\Lambda$, and likewise the (isomorphic) cotangent
fiber at positive infinity can be cofinally negatively wrapped without intersecting $\Lambda$.
These large wrappings are conormals to large disks in $\RR^{n+1}$ containing the projection of $\Lambda$;
they thus have vanishing wrapped Floer cohomology (in both directions) with the linking disks $D_i$ to $\Lambda$.
Thus $D=D_1\sqcup\cdots\sqcup D_n$ and $F$ are orthogonal objects of $\W(T^* \RR^{n+1}, \Lambda)$.

Denote by $\mu=\mu_1\oplus\cdots\oplus\mu_n$ and $\sigma\in\Sh_\Lambda(T^*\RR^{n+1})^c$ the objects corresponding to $D=D_1\oplus\cdots\oplus D_n$ and $F$.
They are orthogonal, and have endomorphism algebras $\A_\Lambda^\op$ and $\ZZ$, respectively.

We have $\Sh_\Lambda(T^*\RR^{n+1})=\Mod\Sh_\Lambda(T^*\RR^{n+1})^c=\Mod\W(T^* \RR^{n+1}, \Lambda)^\op=\Mod\A_\Lambda^\op\oplus\Mod\ZZ$, and this equivalence is given concretely by $\F\mapsto\Hom(\mu,\F)\oplus\Hom(\sigma,\F)$.
By Theorem \ref{sheaffukayaequivalence}, $\Hom(\mu,\F)=\Hom(\mu_1,\F)\oplus\cdots\oplus\Hom(\mu_n,\F)$ is the direct sum of microstalks along $\Lambda$ near $D_1,\ldots,D_n$, and $\Hom(\sigma,\F)$ is the stalk at infinity.

To see that perfect stalks is equivalent to perfect microstalks along $\Lambda$ for objects of $\Sh_\Lambda(T^*\RR^{n+1})_0$, argue as follows.
Suppose microstalks are perfect.
Stalks are computed by $\Hom(\ZZ_{B_\epsilon(x)},\F)$ for some sufficiently small $\epsilon>0$ (in terms of $\Lambda$), since changing $\epsilon$ is non-characteristic by Whitney's condition (b) for a subanalytic Whitney stratification $\SSS$ whose conormal contains $\Lambda$.
Now moving $B_\epsilon(x)$ generically to infinity picks up some number of microstalks when its conormal passes through $\Lambda$ (transversally), and eventually gives zero since the stalk of $\F$ near infinity vanishes.
Thus perfect microstalks implies perfect stalks.
Perfect stalks implies perfect microstalks was proven in Corollary \ref{propersheaves}.
\end{proof}

Let us comment on the relation of the above result to the `augmentations are sheaves' statement in \cite{shende-treumann-zaslow, nrssz} (and later
developments such as \cite{rutherford-sullivan, an-bae-su}).  
There is an evident similarity: both relate augmentations of an algebra associated to a Legendrian to categories of sheaves microsupported
in that Legendrian.  But they are not exactly the same: the algebra $\A_\Lambda$ is not by definition the Chekanov--Eliashberg dga, and 
moreover in \cite{nrssz} the category of augmentations is defined by a somewhat complicated procedure, not just as proper modules over a dga.
Also in \cite{nrssz}, the authors restrict attention to augmentations, i.e.\ $1$-dimensional representations of the dga, whereas the above 
result concerns the entire representation category (the underlying $\ZZ$-module of the representation being the microstalk), specializing to a comparison of rank $k$ representations with rank $k$ microstalk sheaves for every $k$.

In fact, $\A_\Lambda$ was conjectured by Sylvan to be a version of the Chekanov--Eliashberg dga with enhanced $C_*(\Omega \Lambda)$ coefficients.
A precise statement comparing $\A_{\Lambda}$ to such a generalized ``loop space
dga'' can be found in \cite[Conj.\ 3]{ekholm-lekili}, where it is explained that
the comparison should follow from a slight variant of the surgery techniques of
\cite{bourgeoisekholmeliashberg,ekholmsurgery}.
The relation between the multiple copy
construction of \cite{nrssz} and the loop space dga can also be extracted from \cite{ekholm-lekili}.

Finally we note that a version of the above discussion serves to translate between the arguments of \cite{shende-conormal} and \cite{ekholmngshende}.

\subsection{Fukaya--Seidel categories of cotangent bundles}

Let $W: T^*M \to \CC$ be an exact symplectic fibration with singularities.
The associated
Fukaya--Seidel category is by (our) definition $\W(T^*M, W^{-1}(-\infty))$.  
According to \cite[Cor.\ 3.9]{gpsdescent}, retracting the stop to its core does not affect the category: 
$\W(T^*M, W^{-1}(-\infty)) = \W(T^*M, \cc_{W^{-1}(-\infty)})$.  
Thus if the fiber is Weinstein, then we may calculate the corresponding Fukaya--Seidel category using Theorem \ref{sheaffukayaequivalence} (provided the core is subanalytic).

In particular, the sheaf theoretic work on mirror symmetry for toric varieties may now be translated into assertions
regarding the wrapped Fukaya category.   Recall that 
\cite{FLTZ-Morelli} introduced for any $n$-dimensional toric variety 
$\TTT$ a certain
Lagrangian $\Lambda_\TTT \subseteq T^*(S^1)^n$.  They conjectured,\footnote{Strictly speaking, they conjectured the proper module version of 
this statement.} and \cite{kuwagaki} proved, that 
$\Sh_{\partial_\infty\Lambda_\TTT}((S^1)^n)^c = \Coh(\TTT)$, where we use $\Coh$ to denote the dg category of coherent complexes.  By Theorem \ref{sheaffukayaequivalence}, we may conclude:

\begin{corollary}\label{wrappedcoh}
$\Perf\W(T^*(S^1)^n, \partial_\infty\Lambda_\TTT)^\op = \Coh(\TTT)$.
\end{corollary}

When $\TTT$ is smooth and Fano, it was expected that the $\Coh(\TTT)$ should be 
equivalent to the Fukaya--Seidel category of the mirror
Hori--Vafa superpotential \cite{hori-vafa}.  
To compare this expectation with Corollary \ref{wrappedcoh}, it 
suffices to show that $\partial_\infty \Lambda_\TTT$ is in fact the core of the fiber of said superpotential in the Fano case. This is 
shown under certain hypotheses in \cite{gammage-shende} and in general in \cite{pengzhou}. 
We summarize the above discussion in the right column of \eqref{eq:gammage-shende}. 

These results may be compared with \cite{abouzaid-toric}, which for smooth projective $\TTT$ gives
a fully faithful embedding of $\Perf(\TTT)$ into an infinitesimal Fukaya category for the superpotential. 
In the Fano case, this is recovered and upgraded
to an equivalence by taking proper modules of the formulation in \cite{gammage-shende}.  

Note that Corollary \ref{wrappedcoh} gives an equivalence in the general 
(non-Fano, non-compact, singular, and stacky) case, although this equivalence is
not yet formulated in terms of a superpotential.  Such a formulation is known to be 
somewhat subtle, requiring the exclusion of some critical values; see
\cite[Sec.\ 5]{aurouxkatzarkovorlovprojective} or \cite{abouzaid-toric}.  
It may be interesting to explore this using the present methods.

\section{Partially wrapped Fukaya categories and microlocal sheaves}\label{microsection}

The purpose of this section is to prove Theorem \ref{microsheaffukayaequivalence}.
The reasoning in this section depends only on the statement of Theorem \ref{sheaffukayaequivalence}, together with various results from \cite{gpsdescent} and \cite{nadler-shende}; as such, it can be read independently of previous sections of this article.

The main point in the derivation of Theorem \ref{microsheaffukayaequivalence} from Theorem \ref{sheaffukayaequivalence} is to properly exploit, on both the Fukaya side and the sheaf side, the `doubling trick', which allows one to embed the category associated (on either side) to a Liouville manifold (possibly relative a singular isotropic stop) into the category associated to a cotangent bundle relative an appropriate `doubled' stop obtained from an embedding of (a stabilization of) the given Liouville manifold into the co-sphere bundle.
The use of such an embedding to reduce to cotangent bundles was advocated for on the sheaf side in \cite{shende-microlocal}.
The doubling trick has appeared in various forms on the Fukaya side \cite{gpsdescent, sylvanswappable}, and we develop it systematically below.
It has been used on the sheaf side in \cite{guillermou, nadler-shende} to embed categories of \emph{microlocal} sheaves into categories of (usual) sheaves, and we will use it for the same purpose here. 

While our eventual appeal to Theorem \ref{sheaffukayaequivalence} will require the relative core in question to be subanalytic isotropic, most of the intermediate results of this section require much weaker assumptions.
A subset $\Lambda$ of a contact (resp.\ symplectic) manifold will be called \emph{mostly Legendrian} (resp.\ \emph{Lagrangian}) \cite[Def.\ 1.7]{gpsdescent} iff the complement $\Lambda^\subcrit:=\Lambda\setminus\Lambda^\crit$ of the open locus $\Lambda^\crit\subseteq\Lambda$ where $\Lambda$ is a smooth Legendrian (resp.\ Lagrangian) can be covered by the smooth image of a second countable smooth manifold of dimension strictly less than Legendrian (resp.\ Lagrangian).
A conical mostly Lagrangian subset is (locally) the cone over a mostly Legendrian subset (proof: intersect with a generic contact type hypersurface).
A subanalytic isotropic subset is mostly Legendrian/Lagrangian.

\subsection{Homological cocores}\label{hccsec}

We begin by introducing homological cocores, which are a simultaneous generalization of linking disks and cocores.
They are analogous to co-representatives of microstalks in the sheaf theoretic context.
At various points in our arguments below, it will be relevant to assume a given Liouville manifold $X$ `admits homological cocores' in the sense defined below (in fact, admitting homological cocores is most naturally a condition on a pair $(X,\Lambda)$, which turns out to be independent of $\Lambda$ and invariant under deformations of $X$).
Every Weinstein manifold admits homological cocores.

We begin by recalling a special case of the K\"unneth embedding from \cite{gpsdescent}.
Let $(X,\Lambda)$ be a stopped Liouville manifold.
We have a K\"unneth functor \cite[Thm.\ 1.5]{gpsdescent}
\begin{equation}\label{firstkunneth}
\W(X,\Lambda)\hookrightarrow\W\bigl((X,\Lambda)\times(\CC,\infty\cup e^{i[\frac\pi 2,\frac{3\pi}2]}\infty)\bigr)
\end{equation}
given by multiplication by the linking disk $D_\infty\in\W(\CC,\infty\cup e^{i[\frac\pi 2,\frac{3\pi}2]}\infty)$ at $\infty\in\partial_\infty\CC$; it is fully faithful since the endomorphism algebra of $D_\infty$ is $\ZZ$.
We also have a fully faithful embedding
\begin{equation}\label{fftear}
\W((X,\Lambda)\times(\CC_{\Re\geq 0},\infty))\subseteq\W\bigl((X,\Lambda)\times(\CC,\infty\cup e^{i[\frac\pi 2,\frac{3\pi}2]}\infty)\bigr)
\end{equation}
by \cite[Lem.\ 3.7]{gpsdescent} (to be explicit, $(X,\Lambda)\times(\CC_{\Re\geq 0},\infty)=(X\times\CC_{\Re\geq 0},(\cc_X\times\infty)\cup(\Lambda\times\RR_{>0}))$).
The image of the K\"unneth functor \eqref{firstkunneth} is evidently contained in this full subcategory \eqref{fftear}, so we obtain a functor
\begin{equation}\label{secondkunneth}
\W(X,\Lambda)\hookrightarrow\W((X,\Lambda)\times(\CC_{\Re\geq 0},\infty))
\end{equation}
which will be used throughout this section.

\begin{definition}[Homological cocore]\label{hccdef}
Let $(X,\Lambda)$ be a stopped Liouville manifold whose relative core $\cc_{X,\Lambda}$ is mostly Lagrangian.
A \emph{homological cocore} at a smooth Lagrangian point $p\in\cc_{X,\Lambda}$ is an object of $\Perf\W(X,\Lambda)$ whose image under ($\Perf$ of) the K\"unneth embedding \eqref{secondkunneth} is the linking disk at $p\times\infty\in\cc_{X,\Lambda}\times\infty$.
\end{definition}

Recall from \cite[Proof of Thm.\ 1.14]{gpsdescent} that if $L\subseteq X$ is exact, cylindrical at infinity, and intersects $\cc_{X,\Lambda}$ precisely once, transversely, at a smooth Lagrangian point $p\in\cc_{X,\Lambda}$, then $L$ is a homological cocore at $p$.
In particular, (properly embedded) cocores of critical Weinstein handles are homological cocores.
Also recall from \cite[Sec.\ 9.1]{gpsdescent} that the linking disk at a point of $\Lambda$ is a homological cocore at the corresponding point of $\cc_{X,\Lambda}$.

\begin{definition}\label{ahccdef}
We say that $(X,\Lambda)$ (with $\cc_{X,\Lambda}$ mostly Lagrangian) \emph{admits homological cocores} iff every smooth Lagrangian point of $\cc_{X,\Lambda}$ has a homological cocore.
\end{definition}

It follows from stop removal and the vanishing of $\W(X\times\CC_{\Re\geq 0})$ \cite[Lem.\ 9.1]{gpsdescent} that the linking disks to $\cc_{X,\Lambda}\times\infty$ split-generate $\W((X,\Lambda)\times(\CC_{\Re\geq 0},\infty))$.
Hence $(X,\Lambda)$ admits homological cocores iff the K\"unneth embedding \eqref{secondkunneth} is a Morita equivalence (and in this case $\Perf\W(X,\Lambda)$ is split-generated by the homological cocores).
In fact, we have the following equivalent characterizations of admitting homological cocores.

\begin{proposition}\label{enoughprop}
For a stopped Liouville manifold $(X,\Lambda)$ whose relative core $\cc_{X,\Lambda}$ is mostly Lagrangian, the following are equivalent:
\begin{itemize}
\item$(X,\Lambda)$ admits homological cocores.
\item The K\"unneth embedding \eqref{secondkunneth} is a Morita equivalence.
\item The K\"unneth embedding is a Morita equivalence:
\begin{equation}
\label{kunc}\W(X,\Lambda)\hookrightarrow\W((X,\Lambda)\times(\CC,\pm\infty)).
\end{equation}
\item The K\"unneth embedding is a Morita equivalence:
\begin{equation}
\label{kunhalfplanenolambda}\W(X)\hookrightarrow\W(X\times(\CC_{\Re\geq 0},\infty)).
\end{equation}
\item The K\"unneth embedding is a Morita equivalence:
\begin{equation}
\label{kuncwithoutlambda}\W(X)\hookrightarrow\W(X\times(\CC,\pm\infty)).
\end{equation}
\end{itemize}
\end{proposition}

\begin{proof}
The equivalence of admitting homological cocores and the K\"unneth embedding \eqref{secondkunneth} being a Morita equivalence was already argued for above.

We argue that \eqref{secondkunneth} and \eqref{kunhalfplanenolambda} are equivalent.
These are the statements that the categories $\W((X,\Lambda)\times(\CC_{\Re\geq 0},\infty))$ and $\W(X\times(\CC_{\Re\geq 0},\infty))$ (respectively) are split-generated by Lagrangians of the form $L\times[i\RR]$.
Taking $L$ to be a linking disk of $\Lambda$, we see that the linking disks to $\Lambda\times\RR_{>0}$ inside $\W((X,\Lambda)\times(\CC_{\Re\geq 0},\infty))$ are of the desired form.
Therefore it is split-generated by Lagrangians of the form $L\times[i\RR]$ iff its quotient by the linking disks to $\Lambda\times\RR_{>0}$ is split-generated by these objects, and this quotient is precisely $\W(X\times(\CC_{\Re\geq 0},\infty))$ by stop removal.
The same argument shows that \eqref{kunc} and \eqref{kuncwithoutlambda} are equivalent.

We argue that \eqref{kunhalfplanenolambda} and \eqref{kuncwithoutlambda} are equivalent.
They are the statements that the categories $\W(X\times(\CC_{\Re\geq 0},\infty))$ and $\W(X\times(\CC,\pm\infty))$ (respectively) are split-generated by Lagrangians of the form $L\times[i\RR]$.
These statements are equivalent since the natural functor $\W(X\times(\CC_{\Re\geq 0},\infty))\to\W(X\times(\CC,\pm\infty))$ is an equivalence by \cite[Cor.\ 3.9]{gpsdescent}.
\end{proof}

Note that condition \eqref{kuncwithoutlambda} does not involve $\Lambda$ and is invariant under deforming $X$; it holds whenever $X$ is Weinstein by \cite[Cor.\ 1.18]{gpsdescent}.

If $\cc_{X,\Lambda}$ is mostly Lagrangian, the stabilization $(X,\Lambda)\times(\CC,\pm\infty)$ admits homological cocores since every component of the smooth Lagrangian locus of $\cc_{(X,\Lambda)\times(\CC,\pm\infty)}=\cc_{X,\Lambda}\times\RR$ is unbounded.
It follows that the K\"unneth embedding
\begin{equation}
\W((X,\Lambda)\times(\CC,\pm\infty)^k)\hookrightarrow\W((X,\Lambda)\times(\CC,\pm\infty)^{k+1})
\end{equation}
is a Morita equivalence for every $k>0$.

\subsection{Liouville hypersurfaces} \label{polarizedweinstein}

Recall that a \emph{Liouville hypersurface embedding} $X\hookrightarrow Y$ is a codimension one embedding of a Liouville domain $(X,\lambda)$ into a contact manifold $(Y,\xi)$ such that there exists a contact form $\alpha$ on $(Y,\xi)$ whose restriction to $X$ coincides with $\lambda$.  A \emph{Liouville pair} $(Z,X)$ is a Liouville manifold $Z$ together with a Liouville hypersurface embedding $X \hookrightarrow \partial_\infty Z$. 

We will often abuse terminology and speak of a Liouville hypersurface embedding of a Liouville \emph{manifold} into a contact manifold, to mean a Liouville hypersurface embedding of a Liouville domain whose completion is the given Liouville manifold.  

We record here two real analytic approximation results for later use.

\begin{lemma}\label{analyticcontactomorphism}
Any codimension zero smooth embedding of real analytic contact manifolds $(U,\xi_U)\hookrightarrow(Y,\xi_Y)$ can be smoothly approximated over compact subsets of the domain by real analytic embeddings.
\end{lemma}

\begin{proof}
First, we approximate the given embedding by a real analytic map $f$ which does not necessarily respect contact structures.
We now have two real analytic contact structures $f^*\xi_Y$ and $\xi_U$ on $U$ which are $C^\infty$-close.
Interpolating linearly yields a real analytic family of real analytic contact structures $\xi_t$ interpolating between $f^*\xi_Y$ and $\xi_U$.
By Gray's theorem, we obtain a real analytic family of real analytic vector fields $V_t$ defined uniquely by the properties $V_t\in\xi_t$ and $\sL_{V_t}\xi_t=\frac d{dt}\xi_t$.
The total flow of this family $V_t$ thus defines a real analytic diffeomorphism of $U$ (possibly defined only on a large compact subset due to lack of completeness) carrying $\xi_U$ to $f^*\xi_Y$.
Pre-composing $f$ by this diffeomorphism gives the desired real analytic map.
\end{proof}

\begin{corollary}\label{analytichypersurface}
Let $(X,\lambda)$ be a real analytic Liouville domain, and let $(Y,\xi)$ be a real analytic contact manifold.
Any Liouville hypersurface embedding $X\hookrightarrow Y$ can be smoothly approximated by real analytic Liouville hypersurface embeddings.
\end{corollary}

\begin{proof}
Given a Liouville hypersurface embedding $X\hookrightarrow Y$, there is an induced codimension zero inclusion of contact manifolds $X\times[0,1]\hookrightarrow Y$ to which we may apply Lemma \ref{analyticcontactomorphism}.
\end{proof}

We now study the question of when a Liouville manifold $X$ admits a Liouville hypersurface embedding $X\hookrightarrow S^*M$.
Such an embedding determines three pieces of `formal' data\footnote{The term `formal' has a precise meaning in the context of the $h$-principle, see \cite[Sec.\ 5.3]{eliashbergmishachev}.}
\begin{enumerate}
\item A smooth map $f:X\to M$. 
\item A splitting $f^*TM=B\oplus\underline\RR$.
\item An isomorphism of complex vector bundles $TX=B\otimes_\RR\CC$.
\end{enumerate}
The first two pieces of data are equivalent to a homotopy class of smooth maps $X\to S^*M$.
Indeed, up to contractible choices, a lift of $f:X\to M$ to $S^*M$ is the same as a non-vanishing section of $f^*TM$, which is the same as a trivialized subbundle $\underline\RR\subseteq f^*TM$, which is the same as a splitting $f^*TM=\underline\RR\oplus B$.
The isomorphism $TX=B\otimes_\RR\CC$ comes from the derivative of the embedding, which identifies $TX$ with the pullback of the contact distribution on $S^*M$.
There is an existence $h$-principle for Liouville hypersurface embeddings (under a certain `half-dimensional' hypothesis on the core), namely:

\begin{lemma}\label{hprinciple}
Let $X$ be a Liouville manifold whose core $\cc_X$ is contained in a finite union of locally closed submanifolds of dimension at most half the dimension of $X$.
Every triple of formal data as above comes from a Liouville hypersurface embedding $X\hookrightarrow S^*M$.
\end{lemma}

\begin{proof}
The formal data is (homotopy) equivalent to a smooth map $p:X\to S^*M$ together with an isomorphism $q:TX=p^*\xi$ where $\xi$ is the contact distribution of $S^*M$.
Equivalently, it is the data of a smooth map $p:X\times[0,1]\to S^*M$ together with an \emph{isocontact isomorphism} $q:T(X\times[0,1])=p^*TS^*M$ (i.e.\ an isomorphism respecting contact distributions and their conformal symplectic structure).
The $h$-principle \cite[16.1.1]{eliashbergmishachev} now guarantees that the pair $(p,q)$ is homotopic to an \emph{isocontact immersion}, i.e.\ one for which $p$ is an immersion which pulls back the contact structure on $S^*M$ to the contact structure $\lambda+dt$ on $X\times[0,1]$.
The assumption on the core $\cc_X$ ensures that a generic perturbation of $p$ is an embedding in a small neighborhood of $\cc_X\times\{\frac 12\}$.
\end{proof}

\begin{corollary}\label{stableembedding}
Let $X$ be a Liouville manifold.
For any stable polarization $TX=B\otimes_\RR\CC$, there is a Liouville hypersurface embedding of $X\times\CC^k$ (some $k<\infty$) into some $S^*M$, compatible with stable polarizations.
\end{corollary}

\begin{proof}
In view of Lemma \ref{hprinciple} (whose hypothesis is trivially satisfied for $X\times\CC^k$ once $k\geq\frac 12\dim X$), it suffices to show that there exists a manifold $N$ and a map $f:X\to N$ such that $f^*TN$ and $B$ are stably isomorphic (i.e.\ isomorphic after direct summing with some $\underline\RR^m$).
To see that this is true, note that the tangent bundle to the Grassmannian of $n$-planes in $\RR^N$ is (stably) inverse to the tautological vector bundle.
\end{proof}

\begin{remark}\label{stableembeddingwithtwist}
Corollary \ref{stableembedding} concerns the product polarization of $X\times\CC^k$.
By contrast, there always exists a twisted stabilization, i.e.\ the total space of an arbitrary polarized symplectic vector bundle over $X$ (equivalently, 
the complexification of a rank $k$ real vector bundle) which embeds into $S^*\RR^N$ as in \cite{shende-microlocal, nadler-shende}.  In the present article, we need
to restrict to untwisted stabilization because in \cite{gpsdescent} we have only proven an untwisted K\"unneth theorem.  Meanwhile 
in the sheaf-theoretic settings of \cite{shende-microlocal, nadler-shende}, the corresponding twisted
K\"unneth result is a formality, and it is convenient to embed into $S^* \RR^N$ rather than some $S^*M$ in order to have (homotopical) uniqueness of embeddings.
\end{remark}

\subsection{Doubling I: fully faithful embeddings of Fukaya categories}

We first recall the doubling trick in the `absolute' (i.e.\ no stop) setting.
Consider a Liouville pair $(Z,X)$.
Grading/orientation data on $Z$ determines 
such data on $X$ by restriction; when $Z=T^*M$, our primary case of interest, this means that $X$ is equipped with the polarization induced from the Legendrian foliation of $S^*M$ by co-spheres.
With respect to such compatible data, there is a functor 
$\W(X) \to \W(Z, \cc_X)$
obtained by composing the K\"unneth map $\W(X)\to\W(X\times\CC_{\Re\geq 0},\cc_X\times\{\infty\})$ with the canonical neighborhood $X\times\CC_{\Re\geq 0}\hookrightarrow Z$ 
of the Liouville hypersurface $X$ (an embedding of Liouville sectors).

We now consider the double $D(\cc_X):=\cc_X\sqcup\cc_X^\epsilon$, where $\cc_X^\epsilon$ denotes a small positive pushoff of $\cc_X$.
There is a functor
\begin{equation}\label{doubleff}
\W(X)\to\W(Z,\cc_X\sqcup\cc_X^\epsilon)
\end{equation}
defined by including $(X\times\CC_{\Re\geq 0},\cc_X\times\{\infty\})$ into $(T^*M,\cc_X\sqcup\cc_X^\epsilon)$ as the canonical neighborhood of the first copy of $\cc_X$ inside $D(\cc_X)$.
When $\cc_X$ is mostly Lagrangian, this functor evidently sends a homological cocore at a point of $\cc_X$ to the linking disk at the corresponding point of the first copy of $\cc_X$ inside the double.

\begin{proposition}\label{doubleffprop}
The functor \eqref{doubleff} is fully faithful.
\end{proposition}

\begin{proof}
\cite[Ex.\ 10.7]{gpsdescent} asserts that the covariant pushforward $\W(X) \to \W(Z, X \sqcup X^\epsilon)$ is fully faithful, where $X^\epsilon$ denotes a small positive pushoff of $X$.
Combining this with the fact that the functor $\W(Z, X \sqcup X^\epsilon)\to\W(Z,\cc_X\sqcup\cc_X^\epsilon)$ is an equivalence \cite[Cor.\ 3.9]{gpsdescent}, we conclude that \eqref{doubleff} is fully faithful.
\end{proof}

We now explain the doubling trick in the presence of a stop (`relative doubling').

\begin{construction}[Doubling the relative core of a Liouville hypersurface]\label{doubling-gps}
Let $(Z,X)$ be a Liouville pair, and let $\Lambda\subseteq\partial_\infty X$ be a stop.
We will define the double $D(\cc_{X,\Lambda})\subseteq\partial_\infty Z$.
The double is contained in a small neighborhood of $X$, so it suffices to define it as a subset of $\partial_\infty(X\times\CC_{\Re\geq 0})$ (and then push forward under the standard neighborhood $X\times\CC_{\Re\geq 0}\hookrightarrow Z$ of $X$ inside $Z$).

The double $D(\cc_{X,\Lambda})\subseteq\partial_\infty(X\times\CC_{\Re\geq 0})$ is the stop of the product of stopped Liouville manifolds
\begin{equation}\label{firstproduct}
(X,\Lambda)\times(\CC,\{\pm i\infty\})=(X\times\CC,(\cc_X\times\{\pm i\infty\})\cup(\Lambda\times i\RR)),
\end{equation}
which indeed lies inside $\partial_\infty(X\times\CC_{\Re\geq 0})\subseteq\partial_\infty(X\times\CC)$.
The double $D(\cc_{X,\Lambda})$ is evidently comprised of a `first copy' of $\cc_{X,\Lambda}$ namely $(\cc_X\times\{-\infty\})\cup(\Lambda\times i\RR_{<0})$ and a `second copy' of $\cc_{X,\Lambda}$ namely $(\cc_X\times\{+\infty\})\cup(\Lambda\times i\RR_{>0})$, joined along their common boundary $\Lambda=\Lambda\times 0$.
\end{construction}

The fact that $D(\cc_{X,\Lambda})\subseteq\partial_\infty(X\times\CC_{\Re\geq 0})$ lies on the boundary poses no issue for defining $D(\cc_{X,\Lambda})\subseteq\partial_\infty Z$ as its image under (the action on boundaries at infinity of) $X\times\CC_{\Re\geq 0}\hookrightarrow Z$.
We will, however, want to consider $\W(X\times\CC_{\Re\geq 0},D(\cc_{X,\Lambda}))$, and for the purpose of defining this category, we implicitly push $D(\cc_{X,\Lambda})$ inward using a choice of contact vector field transverse to the boundary.
Alternatively, we could use a different Liouville structure on $\CC_{\Re\geq 0}$ which is strictly isomorphic to $T^*[0,\epsilon)$ near the boundary, which makes pushing easy (simple translation).

Let us now generalize the functor \eqref{doubleff} and Proposition \ref{doubleffprop} to the relative setting.
We consider the composition
\begin{multline}\label{doublecompose}
\W(X,\Lambda)\overset{\eqref{secondkunneth}}\hookrightarrow\W((X,\Lambda)\times(\CC_{\Im\leq 0},-i\infty))\\\to\W(X\times\CC,(\cc_X\times\{-\infty,\pm i\infty\})\cup(\Lambda\times i\RR)).
\end{multline}
The target category is identified with $\W(X\times\CC_{\Re\geq 0},D(\cc_{X,\Lambda}))$ by \cite[Cor.\ 3.9]{gpsdescent}, so we obtain a canonical functor
\begin{equation}\label{doubleffrelativespecial}
\W(X,\Lambda)\to\W(X\times\CC_{\Re\geq 0},D(\cc_{X,\Lambda})),
\end{equation}
and hence composing with any inclusion $X\times\CC_{\Re\geq 0}\hookrightarrow Z$, a functor
\begin{equation}\label{doubleffrelative}
\W(X,\Lambda)\to\W(Z,D(\cc_{X,\Lambda})).
\end{equation}
When $\cc_{X,\Lambda}$ is mostly Lagrangian, this functor evidently sends a homological cocore at a point of $\cc_{X,\Lambda}$ to the linking disk at the corresponding point of the first copy of $\cc_{X,\Lambda}$ inside the double.

\begin{proposition}\label{relativedoublingff}
The functor \eqref{doubleffrelative} is fully faithful.
\end{proposition}

\begin{proof}
Appealing to the definition of the functor \eqref{secondkunneth}, the functor \eqref{doubleffrelative} is the composition
\begin{multline}\label{doublecomposesecond}
\W(X,\Lambda)\overset{\eqref{firstkunneth}}\hookrightarrow\W((X,\Lambda)\times(\CC,-i\infty\cup e^{i[0,\pi]}\infty))\\
\to\W(X\times\CC,(\cc_X\times\{-\infty,\pm i\infty\})\cup(\Lambda\times i\RR))\to\W(Z,D(\cc_{X,\Lambda}))
\end{multline}
of K\"unneth, stop removal, and pushforward.
Given that K\"unneth is fully faithful, it suffices to show that the composition of the latter two functors is fully faithful when restricted to product objects $L\times D_{-i\infty}\subseteq(X,\Lambda)\times(\CC,-i\infty\cup e^{i[0,\pi]}\infty)$.
In fact, we will show they are both full faithful on such objects.

To show full faithfulness comes down to understanding cofinal wrappings.
It was shown in \cite[Sec.\ 7.4]{gpsdescent} that products of cofinal wrappings are cofinal (this was the basis for full faithfulness of the K\"unneth functor).
But the results of \cite[Sec.\ 7.4]{gpsdescent} are better: they in fact show that if wrappings of $L\subseteq(X,\Lambda)$ and $D_{-i\infty}\subseteq(\CC,-i\infty\cup e^{i[0,\pi]}\infty)$ satisfy the cofinality criterion \cite[Lem.\ 2.2]{gpsdescent}, then so does their product inside $(X,\Lambda)\times(\CC,-i\infty\cup e^{i[0,\pi]}\infty)$.

Now the cofinality criterion is robust in an important way.
Choose wrappings of $L\subseteq(X,\Lambda)$ and $D_{-i\infty}\subseteq(\CC,\{\pm i\infty\})$ satisfying the cofinality criterion.
Their product satisfies the cofinality criterion in $(X,\Lambda)\times(\CC,\{\pm i\infty\})$, \emph{hence also in} $(X\times\CC,(\cc_X\times\{-\infty,\pm i\infty\})\cup(\Lambda\times i\RR))$, as it stays away from the additional stop $\cc_X\times\{-\infty\}$.
Satisfaction of the cofinality criterion is also preserved under cutting out a neighborhood of this additional stop at $\cc_X\times\{-\infty\}$ and embedding into $\W(Z,D(\cc_{X,\Lambda}))$.

We have thus described cofinal wrappings of product objects in the three categories in \eqref{doublecomposesecond} other than $\W(X,\Lambda)$.
The desired full faithfulness results follow using \cite[Lem.\ 3.20]{gpssectorsoc}.
\end{proof}

\subsection{A first comparison}

We now combine the doubling trick embeddings with Theorem \ref{sheaffukayaequivalence} to arrive at a first sheaf theoretic description of some partially wrapped Fukaya categories.
Note that the doubling construction works real analytically by appealing to Corollary \ref{analyticcontactomorphism} to make the contactomorphisms involved in Construction \ref{doubling-gps} real analytic (and in the below we tacitly assume that doubling takes place real analytically in this sense).

Combining Proposition \ref{relativedoublingff} and Theorem \ref{sheaffukayaequivalence}, we obtain the following.

\begin{corollary} \label{cor:asheafdescriptionrel}
Let $(X,\Lambda)$ be a stopped real analytic Liouville manifold whose relative core $\cc_{X,\Lambda}$ is subanalytic isotropic.
Let $M$ be a real analytic manifold and $X\hookrightarrow S^*M$ an analytic Liouville hypersurface embedding.
There is a fully faithful embedding
\begin{equation}\label{fftosheavesrelative}
\W(X,\Lambda)^\op\hookrightarrow\Sh_{D(\cc_{X,\Lambda})}(M)^c
\end{equation}
which sends homological cocores of $(X,\Lambda)$ to co-representatives of microstalks at the corresponding points of the first copy of $\cc_{X,\Lambda}$ inside $D(\cc_{X,\Lambda})$.
\qed
\end{corollary}

In particular, if $(X,\Lambda)$ admits homological cocores, then \eqref{fftosheavesrelative} is a Morita equivalence onto the full subcategory split-generated by co-representatives of the microstalks at smooth points of the first copy of $\cc_{X,\Lambda}$ inside $D(\cc_{X,\Lambda})$.

In order to bridge the gap between Corollary \ref{cor:asheafdescriptionrel} and Theorem \ref{microsheaffukayaequivalence}, note first that Corollary \ref{stableembedding} implies that there always exists some hypersurface embedding $X\hookrightarrow S^*M$ (which can be assumed real analytic by Corollary \ref{analytichypersurface}).
The remaining work thus concerns only the sheaf side: we must relate the sheaf category in \eqref{fftosheavesrelative} (which is, in particular, not \emph{a priori} independent of the choice 
of Liouville hypersurface embedding) to the microsheaf category defined in \cite{nadler-wrapped, shende-microlocal, nadler-shende}.  This is accomplished in 
\cite{nadler-shende}, which we adapt to our purposes in the next subsection. 

\subsection{Doubling II: microlocal sheaves and antimicrolocalization}

Let us now recall the definition of the microlocal sheaf categories appearing in Theorem \ref{microsheaffukayaequivalence}.
The category which sheaf theorists typically associate to a closed subset of $S^*M$ is 
defined as follows.  One forms the ``Kashiwara--Schapira stack'' by sheafifying the presheaf \emph{of categories}
on $T^*M$ given by the formula 
$\muSh^\pre(\Omega) := \Sh(M) / \Sh_{T^*M \setminus \Omega}(M)$.
The presheaf $\muSh^\pre$ is already discussed in \cite{kashiwara-schapira}; working with its sheafification is a more modern
phenomenon, see e.g.\ \cite{guillermou, nadler-wrapped, nadler-shende}.
The notion of microsupport makes sense for a section of this sheaf, and we write $\muSh_\Lambda$ for the subsheaf of full
subcategories of objects with microsupport inside $\Lambda$.
The subsheaf $\muSh_\Lambda\subseteq\muSh$ is evidently supported on $\Lambda$.

The sheaf $\muSh$ is conic; in particular given $(T^*M \setminus M) \xrightarrow{\pi} S^*M \xrightarrow{\iota} T^*M$
we have canonically $\pi^* \iota^* \muSh = \muSh|_{T^*M \setminus M}$.  We denote also the sheaf $\iota^* \muSh$ on $S^*M$ by $\muSh$. 
Likewise for $\Lambda \subseteq S^*M$ we have $\muSh_\Lambda$.  We will consider this sheaf for $\Lambda$ locally closed, and be
interested in the category $\muSh_\Lambda(\Lambda)$.

By construction there are evident maps $\Sh(M) \to \muSh(\Omega)$ for any open $\Omega \subseteq T^*M$, and similarly $\Sh_\Lambda(M) \to \muSh_{\Lambda}(\Lambda \cap \Omega)$; in particular $\Sh_{\Lambda}(M) \to \muSh_{\Lambda \setminus M}(\Lambda \setminus M) = \muSh_{\Lambda_{\infty}}(\Lambda_\infty)$ (the last of which being
in the cosphere bundle).  We term all such maps `microlocalization functors'. 

In fact, the category $\muSh_\Lambda(\Lambda)$ is defined for any space $\Lambda$ equipped with a germ of closed embedding into a contact manifold carrying
a stable polarization \cite{shende-microlocal}.
Indeed, such a contact manifold admits a homotopically unique isocontact embedding into $S^*\RR^N$ as $N\to\infty$ by the $h$-principle \cite[16.1.2]{eliashbergmishachev}.
The key insight of \cite{shende-microlocal} is that, while the image of $\Lambda$ under such an embedding would have vanishing microsheaf category, one can obtain the correct category by thickening $\Lambda$ along the relevant Lagrangian polarization of the normal bundle.
The role of these polarizations on the sheaf side is entirely parallel to the role of polarizations on the Fukaya side to determine grading/orientation data as discussed in Section \ref{gradorsec}; also compare with Corollary \ref{stableembedding} and Remark \ref{stableembeddingwithtwist}.
Note that we may also define $\muSh_\Lambda(\Lambda)$ by embedding into $S^*M$ for any manifold $M$, since such an $M$ admits an embedding into $\RR^N$.

\begin{remark}
For Liouville manifolds $X$ and $X'$ satisfying the hypotheses of Theorem \ref{microsheaffukayaequivalence}, it follows from Theorem \ref{microsheaffukayaequivalence} that if $X$ and $X'$ are within the same Liouville deformation class, then $\muSh_{\cc_X}(\cc_X)=\muSh_{\cc_{X'}}(\cc_{X'})$.
This equivalence is highly non-obvious from the sheaf theoretic standpoint, but is proven directly in \cite{nadler-shende} under certain assumptions of isotropicity.
\end{remark}

The doubling trick in the sheaf context is developed in \cite{nadler-shende}, resulting in embeddings between categories of (microlocal) sheaves parallel to the embeddings between Fukaya categories discussed above.
Let us now recall the precise definition of the doubling operation which is relevant in the sheaf context \cite{nadler-shende}, so as to compare it with Construction \ref{doubling-gps} from the Fukaya context.

We begin with a discussion of the contact manifold
\begin{equation}\label{prodCcoords}
(\CC\times V,\lambda_\CC+\alpha_V)
\end{equation}
where $\lambda_\CC$ is a Liouville form on $\CC$ (for the standard symplectic structure) and $(V,\alpha_V)$ is a contact manifold with choice of contact form.
First, note that the specific choice of Liouville form on $\CC$ is of no importance, as for $f:\CC\to\RR$ there is a strict contactomorphism
\begin{align}\label{liouvilledoesnotmatter}
(\CC\times V,\lambda_\CC+\alpha_V)&\to(\CC\times V,\lambda_\CC+df+\alpha_V),\\
(x,y)&\mapsto(x,e^{-f(x)R_{\alpha_V}}y),
\end{align}
(at least, provided the Reeb flow on $V$ is complete).
Next, for a subset $\Lambda_0\subseteq V$ and a smooth arc $\gamma$ in $\CC$, define
\begin{equation}\label{tracealonggamma}
\gamma\tildetimes\Lambda_0 := \bigcup_{t\in\gamma}\left(\{t\}\times e^{-g(t)R_{\alpha_V}}\Lambda_0\right),
\end{equation}
where $g:\gamma\to\RR$ is a primitive for $\lambda_\CC|_\gamma$, namely $dg=\lambda_\CC|_\gamma$ (so $g$ is well-defined up to adding a locally constant function).
Note that the meaning of $\gamma\tildetimes\Lambda_0$ does not depend on the choice of Liouville form on $\CC$, as the definition is compatible with the contactomorphisms \eqref{liouvilledoesnotmatter}.
If $\Lambda_0$ is isotropic then so is $\gamma\tildetimes\Lambda_0$.

\begin{construction}[Doubling a subset with boundary cooordinates]\label{doubling-ns}
Begin with a 
contact manifold $Y$ and a locally closed relatively compact
$\Lambda\subseteq Y$.
Also fix, in a neighborhood of $\overline{\Lambda} \setminus \Lambda$, coordinates on $Y$ of the form \eqref{prodCcoords} (regarded as a germ near $\{0\}\times\Lambda_0$ for compact $\Lambda_0\subseteq V$) in which $\Lambda=\RR_{>0}\tildetimes\Lambda_0$ (so $\Lambda_0$ is identified with $\overline\Lambda\setminus\Lambda$); we call these \emph{boundary coordinates} for $\Lambda$.
Now the double $D(\Lambda)$ is, near $\Lambda_0$, defined to be $\gamma\tildetimes\Lambda_0$ where $\gamma$ is the immersed arc obtained from two copies of $\RR_{>0}$
by adding a small loop enclosing a sufficiently small positive area near the origin (a contractible choice).
Away from $\{0\}\times V$, the double is thus $\Lambda\sqcup e^{\epsilon R_{\lambda_\CC+\alpha_V}}\Lambda$ (note that $R_{\lambda_\CC+\alpha_V}=R_{\alpha_V}$), which is extended globally by extending the contact form $\lambda_\CC+\alpha_V$ globally (a contractible choice).
The double $D(\Lambda)$ thus consists of $\Lambda$ (the `first copy') and a positive Reeb pushoff of $\Lambda$ (the `second copy') joined appropriately near their boundary.

Note that when $\Lambda$ is subanalytic with $C^r$ subanalytic
boundary coordinates, then we may ensure that the double $D(\Lambda)$ is subanalytic by choosing both $\gamma$ and the global extension of $\alpha$ to be $C^r$ subanalytic.\footnote{A $C^r$ subanalytic function $\RR^n\to\RR$ is one which is $C^r$ and has subanalytic graph.  This class of functions is closed under composition, hence gives rise to a notion of $C^r$ subanalytic manifolds, etc.}\footnote{The integer $r$ is tacitly assumed to be sufficiently large, and we make no attempt to determine the minimum value of $r$ needed for our constructions to go through (though it will not be particularly large).  Note that the tangent bundle of a $C^r$ subanalytic manifold is
a $C^{r-1}$ subanalytic vector bundle, hence the highest regularity one can impose on a contact form $\alpha$ is $C^{r-1}$ subanalytic.
The exterior derivative $d\alpha$, hence also the Reeb vector field $R_\alpha$, will then be $C^{r-2}$ subanalytic.
This would suggest that at a very minimum we must take $r\geq 3$ to ensure we can integrate $R_\alpha$.}
We will tacitly assume that $D(\Lambda)$ is defined in this way whenever $\Lambda$ is assumed to be subanalytic with $C^r$ subanalytic boundary coordinates.
\end{construction}

The relative core of a Liouville hypersurface always has boundary coordinates (of the same regularity as the hypersurface) in the sense of Construction \ref{doubling-ns}.
Indeed, a neighborhood of the boundary of a Liouville hypersurface is given by $(Y\times\RR_s\times\RR_t,e^s\lambda+dt)$ (a germ near $(s,t)=(0,0)$, with the hypersurface itself being the locus $Y\times\{s\leq 0\}\times\{t=0\}$), and we can scale the contact form to be $\lambda+e^{-s}dt$, which has the desired form \eqref{prodCcoords} near $(s,t)=(0,0)$.
In these coordinates, any relative core will have the desired form $\Lambda_0\times\{s\leq 0\}\times\{t=0\}$.

\begin{proposition}
For the relative core of a Liouville hypersurface, the doubles defined in Constructions \ref{doubling-gps} and \ref{doubling-ns} are canonically isotopic.
\end{proposition}

\begin{proof}
Fix a Liouville manifold $X$ with a stop $\Lambda\subseteq\partial_\infty X$, and consider the stopped Liouville manifold
\begin{equation}
(X,\Lambda)\times(\CC,\pm i\infty)=(X\times\CC,(\Lambda\times i\RR)\cup(\cc_X\times\{\pm i\infty\}).
\end{equation}
The stop, as a subset of $\partial_\infty(X\times\CC_{\Re\geq 0})$, is the double of $\cc_{X,\Lambda}$ defined in Construction \ref{doubling-gps}.
We will exhibit an isotopy from it to the double defined in Construction \ref{doubling-ns}.

We begin with the family of stops
\begin{equation}
(\Lambda\times\{e^{i\theta},e^{-i\theta}\}\RR_{\geq 0})\cup(\cc_X\times\{e^{i\theta},e^{-i\theta}\}\infty)
\end{equation}
for $\theta$ from $\pi/2$ to $0$.
This family may be written as
\begin{equation}\label{stopviagamma}
(\gamma_\theta\tildetimes\Lambda)\cup(\partial_\infty\gamma_\theta\times\cc_X)
\end{equation}
where $\gamma_\theta=\{e^{i\theta},e^{-i\theta}\}\RR_{\geq 0}$ (illustrated in the top row of Figure \ref{figurearms}) and we have fixed a contact form on $V=\partial_\infty X$ to obtain coordinates $(\CC,\lambda_\CC)\times(V,\alpha_V)\subseteq\partial_\infty(X\times\CC)$.

%%%%%%%%%%% IMPORTANT NOTE TO COPYEDITORS %%%%%%%%%%%%%
%%%%%%%%%%%%%%%%%%%%%%%%%%%%%%%%%%%%%%%%%%%%%%%%%%%%%%%
%%% If any of the figures needs to be resized, by far the simplest way to do this is
%%% to adjust the "scale" parameter in the \begin{tikzpicture} command.  That is,
%%% \begin{tikzpicture}[scale=\textwidth/10cm]
%%% should be replaced with
%%% \begin{tikzpicture}[scale=\textwidth/???cm]
%%% where ??? is whatever number you want to make the figure(s) the correct size.
%%% A *bigger* number means the figure will be *smaller*.

\begin{figure}[htbp]
\centering
\begin{tikzpicture}[scale=\textwidth/10cm]
\begin{scope}[shift={(0,2.2)}]\draw(0,0)circle(1);\draw(0,-1)to(0,0)to(0,1);\end{scope}
\begin{scope}[shift={(2.2,2.2)}]\draw(0,0)circle(1);\draw(5/13,-12/13)to(0,0)to(5/13,12/13);\end{scope}
\begin{scope}[shift={(4.4,2.2)}]\draw(0,0)circle(1);\draw(12/13,-5/13)to(0,0)to(12/13,5/13);\end{scope}
\begin{scope}[shift={(6.6,2.2)}]\draw(0,0)circle(1);\draw(1,0)to(0,0)(1,0);\end{scope}
\draw(0,0)circle(1);\draw(0,-1)to(0,0)to(0,1);
\begin{scope}[shift={(2.2,0)}]\draw(0,0)circle(1);\draw(5/13,-12/13)to(1/13,-2.4/13)..controls(0,0)and(0,0)..(1/13,2.4/13)to(5/13,12/13);\end{scope}
\begin{scope}[shift={(4.4,0)}]\draw(0,0)circle(1);\draw(12/13,-5/13)to(2.4/13,-1/13)..controls(0,0)and(-.2,-.2)..(-.2,0)..controls(-.2,.2)and(0,0)..(2.4/13,1/13)to(12/13,5/13);\end{scope}
\begin{scope}[shift={(6.6,0)}]\draw(0,0)circle(1);\draw(1,0)to(.1,0)..controls(0,0)and(0,-.1)..(-.1,-.1)..controls(-.2,-.1)and(-.2,.1)..(-.1,.1)..controls(0,.1)and(0,0)..(.1,0)to(1,0);\end{scope}
\end{tikzpicture}
\caption{The family of arcs $\gamma_\theta$ (top) and their smoothings $\tilde\gamma_\theta$ near the origin (bottom).}\label{figurearms}
\end{figure}

Now we smooth $\gamma_\theta$ near the origin to obtain a family of immersed arcs $\tilde\gamma_\theta$ (embedded except at $\theta=0$) as in the bottom row of Figure \ref{figurearms}.
We would like to consider \eqref{stopviagamma} with $\tilde\gamma_\theta$ in place of $\gamma_\theta$.
Note, however, that while $\gamma_\theta$ and $\tilde\gamma_\theta$ agree outside a compact subset of $\CC$, the same is not true of $\gamma_\theta\tildetimes\Lambda$ and $\tilde\gamma_\theta\tildetimes\Lambda$, due to the fact that the actions of $\gamma_\theta$ and $\tilde\gamma_\theta$ necessarily differ at $\theta=0$ (this being the difference of areas enclosed).
Thus while $\gamma_\theta\tildetimes\Lambda$ has two `arms' which coincide at $\theta=0$, the twisted product $\tilde\gamma_\theta\tildetimes\Lambda$ has two `arms' which near infinity differ by a small positive Reeb pushoff at $\theta=0$.
This small positive isotopy extends to the ambient contact manifold $\partial_\infty(X\times\CC)$, hence we can, in particular, apply it to the part of \eqref{stopviagamma} lying near infinity in the $\CC$-coordinate.
This defines the desired isotopy from the double in the sense of Construction \ref{doubling-gps} (at $\theta=\pi/2$) to the double in the sense of Construction \ref{doubling-ns} (at $\theta=0$).
\end{proof}

The following `stabilize and then double' construction will be crucial in what follows.
Let $\Lambda\subseteq S^*M$ be equipped with boundary coordinates.
We consider its `stabilization' $\Lambda\times(0,1)$ inside $S^*(M\times\RR)$, where $(0,1)\subseteq\RR\subseteq T^*\RR$ is contained in the zero section.
This stabilization is naturally equipped with two boundary charts, one near $\partial\Lambda\times(0,1)$ (obtained from the boundary coordinates for $\Lambda$ by multiplying by $T^*(0,1)$) and one near $\Lambda\times\partial(0,1)$ (obtained from the trivial chart $(0,1)\subseteq T^*(0,1)$ by multiplying by $\Lambda\subseteq S^*M$).
When $M$ is real analytic and the boundary coordinates for $\Lambda$ are $C^r$ subanalytic, the first chart is also $C^r$ subanalytic; the second chart is always analytic.
These two boundary charts overlap near $\partial\Lambda\times\partial(0,1)$ in a chart of the form $(\CC^2,\RR_{\geq 0}^2)\times W$.
Viewing the first factor as $(T^*\RR^2,\RR_{\geq 0}^2)$, we may simply smooth the corner of $\RR_{\geq 0}^2$ as in \cite[Sec.\ 7.1]{gpsdescent} to obtain the smoothed product $(\Lambda\times(0,1))^\sm$.
This splices together the charts near $\partial\Lambda\times(0,1)$ and $\Lambda\times\partial(0,1)$ to define boundary coordinates for $(\Lambda\times(0,1))^\sm$.
This smoothing and splicing can be done in the $C^r$ subanalytic category.
The double $D((\Lambda\times(0,1))^\sm)$ is thus defined, and we abuse notation by writing it as $D(\Lambda\times(0,1))$.

The doubling trick for sheaf categories from \cite{nadler-shende} concerns $D(\Lambda\times(0,1))$, and its proof relies on just a short list of its properties, which are easier to see from a different
description of it, as a `movie of creation and destruction' of $\Lambda$,
denoted $(\Lambda,\partial\Lambda)^{\prec\succ}$ in 
\cite[Sec.\ 7.4]{nadler-shende}.  Thus to apply 
the results of \cite{nadler-shende} we must give an isotopy
$D(\Lambda\times(0,1)) \sim (\Lambda,\partial\Lambda)^{\prec\succ}$.  
Let us do this now:

\begin{lemma}\label{doublesameasns}
For $\Lambda\subseteq S^*M$ equipped with boundary coordinates, the double of the stabilization $D(\Lambda\times(0,1))\subseteq S^*(M\times\RR)$ is isotopic to the `movie of creation and destruction' $(\Lambda,\partial\Lambda)^{\prec\succ}\subseteq S^*(M\times\RR)$ from \cite[Sec.\ 7.4]{nadler-shende}.
\end{lemma}

\begin{proof}
The main point is to take the picture from Construction \ref{doubling-ns} based on the Lagrangian projection and translate it into the front projection.

First we translate Construction \ref{doubling-ns} itself into the front projection.
We add an imaginary third coordinate $\RR_z$ to $\CC$ to form $(\CC\times\RR_z,\lambda_\CC+dz)$.
A Legendrian curve $\gamma$ in $(\CC\times\RR_z,\lambda_\CC+dz)$ determines $\gamma\tildetimes\Lambda_0\subseteq(\CC\times V,\lambda_\CC+\alpha_V)$ for $\Lambda_0\subseteq V$ via \eqref{tracealonggamma}.
The `front projection' is the projection $(\CC_{x+iy}\times\RR_z,dz-y\,dx)\to\RR_x\times\RR_z$.
The Legendrian curve relevant for Construction \ref{doubling-ns} has front projection given by the two rays $\RR_{x\geq 0}\times\{z=0,\epsilon\}$ joined by a single cusp in the standard way illustrated in Figure \ref{cuspfigure}.

%%%%%%%%%%% IMPORTANT NOTE TO COPYEDITORS %%%%%%%%%%%%%
%%%%%%%%%%%%%%%%%%%%%%%%%%%%%%%%%%%%%%%%%%%%%%%%%%%%%%%
%%% If any of the figures needs to be resized, by far the simplest way to do this is
%%% to adjust the "scale" parameter in the \begin{tikzpicture} command.  That is,
%%% \begin{tikzpicture}[scale=\textwidth/60cm]
%%% should be replaced with
%%% \begin{tikzpicture}[scale=\textwidth/???cm]
%%% where ??? is whatever number you want to make the figure(s) the correct size.
%%% A *bigger* number means the figure will be *smaller*.

\begin{figure}[htbp]
\centering
\begin{tikzpicture}[scale=\textwidth/60cm]
\draw(0,0)..controls(1,0)and(2,0)..(3,1)..controls(4,2)and(5,2)..(6,2)to(30,2);
\draw(0,0)..controls(1,0)and(2,0)..(3,-1)..controls(4,-2)and(5,-2)..(6,-2)to(30,-2);
\end{tikzpicture}
\caption{A standard cusp.}\label{cuspfigure}
\end{figure}

To understand the double of the stabilization $D(\Lambda\times(0,1))$, we may look in the corner boundary coordinates $T^*\RR_{\geq 0}^2\times W$.
Passing to the front projection and smoothing the corner, we see that in these coordinates, the double is given by two parallel copies of $\RR_{\geq 0}^2$ (with its corner smoothed) joined by cusps (i.e.\ the standard cusp in Figure \ref{cuspfigure} times a smoothing of the boundary of $\RR_{\geq 0}^2$).
Up to fixing a standard model of this cusped object (which some readers might call a `square-ish quarter of a flying saucer'), this is exactly the definition of $(\Lambda,\partial\Lambda)^{\prec\succ}\subseteq S^*(M\times\RR)$ from \cite[Sec.\ 7.4]{nadler-shende}.
\end{proof}

We now state the doubling trick for sheaf categories from \cite{nadler-shende} (substituting $D(\Lambda\times(0,1))$ in place of $(\Lambda,\partial\Lambda)^{\prec\succ}$ in accordance with Lemma \ref{doublesameasns} and the preceding discussion).
The crucial point for us is that it realizes a given category of \emph{microlocal} sheaves as (a full subcategory of) a certain category of \emph{sheaves} with a given singular support condition; this process is termed `antimicrolocalization' in \cite{nadler-shende}.

\begin{theorem}[{\cite[Thm.\ 7.30]{nadler-shende}}]  \label{microsheavesassheaves}
Let $\Lambda$ be a Whitney stratifiable isotropic inside $S^*M$ with $C^r$ boundary coordinates.
The category $\Sh_{D(\Lambda\times(0,1))}(M\times\RR)$ is the orthogonal direct sum of its full subcategories $\Sh_\emptyset(M\times\RR)$ (local systems) and $\Sh_{D(\Lambda\times(0,1))}(M\times\RR)_0$ (objects with vanishing stalk at infinity), and the microlocalization functor
\begin{equation}
\Sh_{D(\Lambda\times(0,1))}(M\times\RR)_0\xrightarrow\sim\muSh_{\Lambda\times(0,1)}(\Lambda\times(0,1))=\muSh_\Lambda(\Lambda)
\end{equation}
is an equivalence.
\qed
\end{theorem}

(The K\"unneth equivalence $\muSh_{\Lambda\times(0,1)}(\Lambda\times(0,1))=\muSh_\Lambda(\Lambda)$ is standard.)

\begin{remark}
The actual hypothesis of Theorem \ref{microsheavesassheaves} in \cite{nadler-shende} (``sufficiently isotropic'') is somewhat weaker than being Whitney stratifiable isotropic.
In our applications, we will in fact always have subanalyticity of $\Lambda$, hence in particular Whitney stratifiability.
\end{remark}

\begin{corollary}\label{mucompactlygenerated}
Let $\Lambda$ be a locally closed relatively compact subanalytic isotropic inside $S^*M$ with $C^r$ subanalytic boundary coordinates.
The category $\muSh_\Lambda(\Lambda)$ is compactly generated by co-representatives of microstalks at the smooth points of $\Lambda$.
\end{corollary}

\begin{proof}
By Theorem \ref{microsheavesassheaves}, we have $\muSh_\Lambda(\Lambda)=\Sh_{D(\Lambda\times(0,1))}(M\times\RR)_0$.
It thus suffices to show that the microstalk co-representatives $\mu_\xi\in\Sh_{D(\Lambda\times(0,1))}(M\times\RR)$ for smooth points $\xi$ of the first copy of $\Lambda\times(0,1)$ inside $D(\Lambda\times(0,1))$ lie in the full subcategory $\Sh_{D(\Lambda\times(0,1))}(M\times\RR)_0$ and compactly generate it.

The category $\Sh_{D(\Lambda\times(0,1))}(M\times\RR)$ is compactly generated by Corollary \ref{microstalksgenerate} (which immediately implies its orthogonal full subcategories from Theorem \ref{microsheavesassheaves} are also compactly generated).
The $\mu_\xi$ (which are compact) are by definition left orthogonal to local systems, so they lie in $\Sh_{D(\Lambda\times(0,1))}(M\times\RR)_0$ by Theorem \ref{microsheavesassheaves}.
An object of $\Sh_{D(\Lambda\times(0,1))}(M\times\RR)$ right orthogonal to all these $\mu_\xi$ must have microsupport contained in $\overline\Lambda\times[0,1]$ (the closure of the second copy), but this implies empty microsupport since $\overline\Lambda\times[0,1]$ is isotropic and every smooth Legendrian component has boundary (so the relevant microstalk always vanishes, which is enough by Proposition \ref{removemicrosupport}).
It follows that the $\mu_\xi$ compactly generate $\Sh_{D(\Lambda\times(0,1))}(M\times\RR)_0$ as claimed.
\end{proof}

\begin{remark}
One may eliminate the subanalyticity hypotheses from Corollary \ref{mucompactlygenerated} at the cost of appealing to more general representability theorems as in Lemma \ref{adjoints}.
For example, for $\Lambda$ stratifiable by isotropics,
$\muSh_\Lambda(\Lambda)$ is compactly generated by co-representatives of microstalks at the smooth points of $\Lambda$.
Indeed, the arguments of Lemma \ref{adjoints} imply that for any open $U\subseteq\Lambda$, the restriction $\muSh_\Lambda(\Lambda)\to\muSh_U(U)$ has a left adjoint which preserves compact objects.
Taking $U$ to be a contractible open subset near a given smooth Legendrian point of $\Lambda$ produces a compact co-representative of the microstalk.
These then compactly generate by Proposition \ref{removemicrosupport}.
\end{remark}

\subsection{Proof of Theorem \ref{microsheaffukayaequivalence}}

We begin by deriving from Theorem \ref{microsheavesassheaves} a sheaf theoretic analogue of Proposition \ref{relativedoublingff}.
Although it is a purely sheaf theoretic statement, the proof we give passes through the Fukaya category and the results of \cite{gpsdescent}.

\begin{corollary}\label{mustarff} 
Let $(X,\Lambda)$ be a stopped Liouville manifold.
Let $M$ be a real analytic manifold and $X\hookrightarrow S^*M$ a
Liouville hypersurface embedding such that the image of $\cc_{X, \Lambda}$ is subanalytic isotropic with $C^r$ subanalytic boundary coordinates.
The left adjoint $\mu^*$ of the microlocalization functor $\mu:\Sh_{D(\cc_{X,\Lambda})}(M)\to\muSh_{\cc_{X,\Lambda}}(\cc_{X,\Lambda})$ is fully faithful.
\end{corollary}

We will see in the proof that $\mu^*$ exists and preserves compact objects for formal reasons.

\begin{proof}
We consider the commuting diagram
\begin{equation}\label{sheaffunctors}
\begin{tikzcd}[column sep = scriptsize]
\Sh_{D(\cc_{X,\Lambda})}(M)\ar{d}{\mu}&\ar{l}{\sim}[swap]{r}\Sh_{D(\cc_{X,\Lambda})\times(0,1)}(M\times(0,1))\ar{d}{\mu}&\Sh_{D(\cc_{X,\Lambda}\times(0,1))}(M\times\RR)\ar{l}[swap]{r}\ar{dl}{\mu}\\
\muSh_{\cc_{X,\Lambda}}(\cc_{X,\Lambda})&\ar{l}{\sim}[swap]{r}\muSh_{\cc_{X,\Lambda}\times(0,1)}(\cc_{X,\Lambda}\times(0,1))
\end{tikzcd}
\end{equation}
where the functors are restriction $r$ and microlocalization $\mu$.
It is a standard result of microlocal sheaf theory that the two leftmost restriction functors $r$ are equivalences.
We note that we may indeed choose the double of the stabilization $D(\cc_{X,\Lambda}\times(0,1))$ so that over $M\times(0,1)$ it coincides with $D(\cc_{X,\Lambda}) \times (0,1)$ (compare the picture from Lemma \ref{doublesameasns}), so the upper right restriction functor $r$ is defined.

Because microsupport respects limits and colimits, so too do the microlocalization functors $\mu$ (see e.g.\ \cite[Rem.\ 6.1]{nadler-shende} for more details) and restriction functors $r$.
The domain sheaf categories are compactly generated by Corollary \ref{microstalksgenerate}, and Brown representability holds for the opposites of compactly generated categories by \cite{neeman-book,krause}, so all functors $r$ and $\mu$ in \eqref{sheaffunctors} admit left adjoints $r^*$ and $\mu^*$.\footnote{In fact, $r^*$ and $\mu^*$ exist in general by arguing as in Lemma \ref{adjoints}.}
Being left adjoint to co-continuous functors, each $r^*$ and $\mu^*$ preserves compact objects.

The microlocal sheaf categories appearing in \eqref{sheaffunctors} are compactly generated by co-representatives of microstalk functors by Corollary \ref{mucompactlygenerated}.
The images of these compact generators under $\mu^*$ are again co-representatives of the same microstalk functors.
To show that a given $\mu^*$ is fully faithful, it suffices to check on compact objects.

By Theorem \ref{microsheavesassheaves}, the diagonal $\mu$ in \eqref{sheaffunctors} is the projection onto an orthogonal direct summand of the domain.
Its left adjoint $\mu^*$ is thus the inclusion of this orthogonal direct summand, hence, in particular, is fully faithful.
Thus to prove full faithfulness of the other vertical $\mu^*$ functors, it suffices to show full faithfulness of
\begin{equation}
r^*:\Sh_{D(\cc_{X,\Lambda})\times(0,1)}(M\times(0,1))\to\Sh_{D(\cc_{X,\Lambda}\times(0,1))}(M\times\RR)
\end{equation}
restricted to co-representatives of the microstalk functors at the first copy of $\cc_{X,\Lambda}\times(0,1)$.
By Proposition \ref{sfecompatibility}, this functor (restricted to compact objects) corresponds under Theorem \ref{sheaffukayaequivalence} to the pushforward functor
\begin{equation}\label{Wside}
\W(T^*(M\times(0,1)),D(\cc_{X,\Lambda})\times(0,1))\to\W(T^*(M\times\RR),D(\cc_{X,\Lambda}\times(0,1))).
\end{equation}
It thus suffices to show that the restriction of this functor to the linking disks of the first copy of $\cc_{X,\Lambda}\times(0,1)$ is fully faithful.

The inclusion
\begin{equation}
(X,\Lambda)\times(\CC,\pm\infty)\times(\CC_{\Re\geq 0},\infty)\hookrightarrow(T^*(M\times\RR),D(\cc_{X,\Lambda}\times(0,1)))
\end{equation}
around the first copy of $\cc_{X,\Lambda}\times(0,1)=\cc_{(X,\Lambda)\times(\CC,\pm\infty)}$ induces a fully faithful functor on $\W$ by Proposition \ref{relativedoublingff} and the fact that $(X,\Lambda)\times(\CC,\pm\infty)$ admits homological cocores.
Now, after a deformation, the above inclusion factors through $(T^*(M\times(0,1)),D(\cc_{X,\Lambda})\times(0,1))$ as the identity map on $T^*(0,1)$ times the canonical inclusion
\begin{equation}
(X,\Lambda)\times(\CC_{\Re\geq 0},\infty)\hookrightarrow(T^*M,D(\cc_{X,\Lambda}))
\end{equation}
around the first copy of $\cc_{X,\Lambda}$.
It thus suffices to show that the induced map on wrapped Fukaya categories is also fully faithful.
To do this, we multiply the proof of Proposition \ref{relativedoublingff} by $T^*(0,1)$.
Namely, we consider Lagrangians inside $T^*(0,1)\times(X,\Lambda)$ times the linking disk of $(\CC_{\Re\geq 0},\infty)$ and consider product wrappings inside $T^*(0,1)\times(X,\Lambda)\times(\CC,\{\pm i\infty\})$ which we conclude satisfy the cofinality criterion, hence remain cofinal after removing a neighborhood of the additional stop at $\cc_X\times\{-\infty\}$ and gluing onto $T^*M$.
\end{proof}

\begin{corollary}\label{sheafdescmicrolocal} 
Let $(X,\Lambda)$ be a stopped Liouville manifold.
Let $M$ be a real analytic manifold and $X\hookrightarrow S^*M$ a
Liouville hypersurface embedding such that the image of $\cc_{X, \Lambda}$ is subanalytic isotropic with $C^r$ subanalytic boundary coordinates.
There is a fully faithful functor $\W(X,\Lambda)^\op\hookrightarrow\muSh_{\cc_{X,\Lambda}}(\cc_{X,\Lambda})^c$ characterized uniquely by commutativity of the diagram
\begin{equation}\label{sheafdescmicrolocaleqn}
\begin{tikzcd}
\W(X,\Lambda)^\op\ar[hook]{r}{\eqref{doubleffrelative}}\ar[hook]{d}&\Perf\W(T^*M,D(\cc_{X,\Lambda}))^\op\ar[equal]{d}{\text{Thm \ref{sheaffukayaequivalence}}}\\
\muSh_{\cc_{X,\Lambda}}(\cc_{X,\Lambda})^c\ar[hook]{r}{\mu^*}&\Sh_{D(\cc_{X,\Lambda})}(M)^c
\end{tikzcd}
\end{equation}
where $\mu^*$ denotes the restriction to compact objects of the left adjoint of the microlocalization functor.
\end{corollary}

\begin{proof}
Proposition \ref{relativedoublingff} and Corollary \ref{mustarff} ensure that the horizontal arrows in \eqref{sheafdescmicrolocaleqn} are fully faithful.
The essential image of \eqref{doubleffrelative} is contained in the subcategory generated by linking disks of the first copy of $\cc_{X,\Lambda}$ since it factors through $\W((X,\Lambda)\times(\CC_{\Re\geq 0},\infty))$ which is generated by linking disks.
The functor $\mu^*$ obviously sends co-representatives of microstalks (which exist by Corollary \ref{mucompactlygenerated}) to co-representatives of microstalks, which are identified with linking disks under Theorem \ref{sheaffukayaequivalence}.
\end{proof}

\begin{remark}
The hypotheses of Corollaries \ref{mustarff} and \ref{sheafdescmicrolocal} may be ensured by assuming $(X,\lambda)$ is analytic, $\cc_{X,\Lambda}$ is subanalytic isotropic, and the embedding $X\hookrightarrow S^*M$ is analytic.
This is how we proceed to prove Theorem \ref{microsheaffukayaequivalence}, relying on the abstract analytic approximation result of Corollary \ref{analytichypersurface}.
However when applying Corollary \ref{analytichypersurface} in practice (and in particular in \cite{gammage-shende}), it can be more convenient to simply check subanalyticity of the core inside $S^*M$ and existence of analytic boundary coordinates (which holds vacuously if $\Lambda=\emptyset$).
\end{remark}

Now the embedding of Theorem \ref{microsheaffukayaequivalence} is simply defined to be that of Corollary \ref{sheafdescmicrolocal} for a choice of Liouville hypersurface embedding, which is guaranteed to exist by Corollary \ref{stableembedding}.

\begin{proof}[Proof of Theorem \ref{microsheaffukayaequivalence}]
By Corollary \ref{stableembedding}, there is a Liouville hypersurface embedding $X\times\CC^k\hookrightarrow S^*M$ compatible with polarizations for some manifold $M$.
Equip $M$ with a real analytic structure, and use Corollary \ref{analytichypersurface} to perturb the embedding to be analytic.
Now apply Corollary \ref{sheafdescmicrolocal} to $(X,\Lambda)\times(\CC,\pm\infty)^k$ and the embedding $X\times\CC^k\hookrightarrow S^*M$ to obtain an embedding
\begin{equation}
\W((X,\Lambda)\times(\CC,\pm\infty)^k)\hookrightarrow\muSh_{\cc_{X,\Lambda}\times\RR^k}(\cc_{X,\Lambda}\times\RR^k),
\end{equation}
which sends homological cocores to co-representatives of microstalks since Theorem \ref{sheaffukayaequivalence} sends linking disks to co-representatives of microstalks.
Finally, combine this with the K\"unneth embedding $\W(X,\Lambda)\hookrightarrow\W((X,\Lambda)\times(\CC,\pm\infty)^k)$ and the equivalence $\muSh_{\cc_{X,\Lambda}}(\cc_{X,\Lambda})=\muSh_{\cc_{X,\Lambda}\times\RR^k}(\cc_{X,\Lambda}\times\RR^k)$.
\end{proof}

While the equivalence of Theorem \ref{sheaffukayaequivalence} is canonical, the embedding of Theorem \ref{microsheaffukayaequivalence} depends \emph{a priori} on a choice of analytic hypersurface embedding $X\times\CC^k\hookrightarrow S^*M$ compatible with polarizations.
We do strongly expect that it is independent of these choices, and moreover that pursuing the present methods a bit further would show this.

In some instances, there is a particularly natural choice of Liouville hypersurface embedding for which the category $\Sh_{\overline{\cc_{X,\Lambda}}}(M)$ is of interest.
It is then of interest to know that the embedding of Theorem \ref{microsheaffukayaequivalence} (associated to this particular hypersurface embedding) and the equivalence of Theorem \ref{sheaffukayaequivalence} intertwine pushforward on Fukaya categories and (the left adjoint of) microlocalization.
We stated this compatibility in the introduction as \eqref{theoremscommute}.
Here we make a stronger statement, relevant in applications, with Corollary \ref{sheafdescmicrolocal} in place of Theorem \ref{microsheaffukayaequivalence}.
Corollary \ref{sheafdescmicrolocal} requires only that the image of the core be subanalytic and have subanalytic boundary coordinates, rather than requiring the hypersurface itself to be analytic.

\begin{proposition}\label{sheafdescmicrolocalwithremoval}
In the notation and assuming the hypotheses of Corollary \ref{sheafdescmicrolocal}, the following diagram commutes:
\begin{equation}\label{sheafdescmicrolocalwithremovaleqn}
\begin{tikzcd}
\W(X,\Lambda)^\op\ar{r}\ar[hook]{d}[swap]{\text{Cor \ref{sheafdescmicrolocal}}}&\Perf\W(T^*M,\overline{\cc_{X,\Lambda}})^\op\ar[equal]{d}{\text{Thm \ref{sheaffukayaequivalence}}}\\
\muSh_{\cc_{X,\Lambda}}(\cc_{X,\Lambda})^c\ar{r}{\mu^*}&\Sh_{\overline{\cc_{X,\Lambda}}}(M)^c
\end{tikzcd}
\end{equation}
\end{proposition}

\begin{proof}
Append to the right side of \eqref{sheafdescmicrolocaleqn} a square diagram forgetting down from $D(\cc_{X,\Lambda})$ to the first copy $\overline{\cc_{X,\Lambda}}\subseteq D(\cc_{X,\Lambda})$.
\end{proof}

\begin{example}[Mirror symmetry for very affine hypersurfaces]\label{veryaffine}
Let $W_\TTT: (\CC^*)^n \to \CC^*$ be the Hori--Vafa 
mirror superpotential to a smooth toric stack $\TTT$.  In Corollary \ref{wrappedcoh}, 
we discussed how the results of the present article allows to translate the sheaf theoretic work of 
\cite{FLTZ-Morelli, kuwagaki} into a mirror symmetry statement equating the Fukaya--Seidel category 
of $W_\TTT$ with $\Coh(\TTT)$.  This also depended on certain calculations of skeleta in \cite{gammage-shende, pengzhou}. 

The main purpose of \cite{gammage-shende} was to provide the relevant skeletal calculations and microlocal
sheaf theoretic results to prove the expected mirror symmetry 
between the wrapped Fukaya category of a generic fiber 
(which we denote $W_\TTT^{-1}(-\infty)$) and the category of coherent sheaves on the toric boundary 
$\Coh(\partial \TTT)$.  Theorem \ref{sheafdescmicrolocal} provides the translation
between microlocal sheaf theory and wrapped Fukaya categories.
To summarize, we have the following commutative diagram 
\begin{equation}\label{eq:gammage-shende}
\begin{tikzcd}
\Coh(\partial \TTT) \ar{r} \ar[equals]{d}[swap]{\text{\cite{gammage-shende}}} & \Coh{\TTT}  \ar[equals]{d}{\text{\cite{FLTZ-Morelli, kuwagaki}}} \\ 
\muSh_{\Lambda_\TTT}(\Lambda_\TTT)^c \ar{r}{\mu^*} \ar[equal]{d}[swap]{\text{Thm \ref{sheafdescmicrolocal}}} & \Sh_{\Lambda_\TTT}((S^1)^n)^c \ar[equals]{d}{\text{Thm \ref{sheaffukayaequivalence}}} \\
\Perf \W(W_\TTT^{-1}(-\infty)) \ar{r} & \Perf \W((\CC^*)^n, W_\TTT^{-1}(-\infty)) 
\end{tikzcd}
\end{equation}
in which the bottom square is \eqref{sheafdescmicrolocalwithremovaleqn}, using the fact that $-\Lambda_\TTT$ is the core of $W_\TTT^{-1}(-\infty)$ from \cite{gammage-shende,pengzhou}.
(The absence of an `op' is due to the appearance of the minus sign in $-\Lambda_\TTT$, and a corresponding use of an antipodal map.) 
\end{example}

\subsection{Making the core subanalytic}\label{secsubanskel}

The goal of this subsection is to show that every Weinstein sector may be perturbed to be real analytic and have subanalytic relative core (and hence satisfy the hypotheses of Theorem \ref{microsheaffukayaequivalence}).

\begin{proposition}\label{analyticstablemflds}
Let $M$ be a real analytic manifold, and let $V$ be a real analytic vector field on $M$ which is convex and complete at infinity and which is gradient-like with respect to a proper Morse function with finitely many critical points.
Suppose that a neighborhood of every zero of $V$ has local analytic coordinates in which $V=\sum_ia_ix_i\frac\partial{\partial x_i}$ for some $a_i\in\QQ\setminus\{0\}$.
Then the union of all stable manifolds $C\subseteq M$ is a subanalytic subset of $M$.
In fact, for any subanalytic subset $\Lambda\subseteq\partial_\infty M$, the union $C\cup(\Lambda\times\RR)\subseteq M$ is subanalytic.
\end{proposition}

\begin{proof}
Fix a proper smooth Morse function $\phi:M\to\RR$ with respect to which $V$ is gradient-like.
There is no real need to make $\phi$ real analytic, though the usual real analytic approximation results allow us to do so if we like.

The core $C$ is compact, so it is vaccuously true that $C_\Lambda:=C\cup(\Lambda\times\RR)$ is subanalytic over $\{\phi>T\}$ for some large $T<\infty$.
By $V$-invariance of $C_\Lambda$, if an interval $[T',T]$ contains no critical values of $\phi$, then $C_\Lambda$ subanalytic over $\{\phi>T\}$ implies $C_\Lambda$ is subanalytic over $\{\phi>T'\}$.
Thus the point is to understand what happens when we cross a critical value of $\phi$.
We may assume the critical values of $\phi$ are distinct.

Fix a critical value of $\phi$, which by translating $\phi$ we may assume is zero. Supposing that $C_\Lambda$ is subanalytic over $\{\phi>\epsilon\}$, let us show that $C_\Lambda$ is subanalytic over $\{\phi>-\epsilon\}$.
It is trivial that $C_\Lambda$ is subanalytic away from the stable manifold of the critical point of $\phi$ in question.
Thus let us work in local analytic coordinates $[-1,1]^{n+m}$ near this critical point in which $V=\sum_ia_ix_i\frac\partial{\partial x_i}-\sum_jb_jy_j\frac\partial{\partial y_j}$ for $a_i,b_i\in\QQ_{>0}$.
We now consider the proper map
\begin{equation}\label{flowmap}
\left\{\begin{matrix}\hfill x_1^2+\cdots+x_n^2=1\\\hfill y_1^2+\cdots+y_m^2=1\\\hfill s,t\geq 0\end{matrix}\right\}\xrightarrow{\textstyle(s^{a_1}x_1,\ldots,s^{a_n}x_n,t^{b_1}y_1,\ldots,t^{b_m}y_m)}\RR^{n+m},
\end{equation}
which is analytically defineable since $a_i,b_i\in\QQ$.
Note that for fixed values of $(x_1,\ldots,x_n,y_1,\ldots,y_m)$ and of the product $st$, the image is a flow line of $V$; in fact, this identifies the space of broken flow lines of $\sum_ia_ix_i\frac\partial{\partial x_i}-\sum_jb_jy_j\frac\partial{\partial y_j}$ on $\RR^{n+m}$ with
\begin{equation}\label{flowspace}
\{x_1^2+\cdots+x_n^2=1\}\times\{y_1^2+\cdots+y_m^2=1\}\times\RR_{\geq 0}.
\end{equation}
We may now show that $C_\Lambda$ is subanalytic in a neighborhood of the stable manifold $\{x_1=\cdots=x_n=0\}\times\RR^m_y$ as follows.
Choose a small real analytic hypersurface $H$ transverse to $V$ near $\{x_1^2+\cdots+x_n^2=1\}\times\{y_1=\cdots=y_m=0\}$.
Since $H$ lies in the locus where $\phi$ is positive, the intersection $C_\Lambda\cap H$ is subanalytic.
Now the image of $C_\Lambda\cap H$ under the backward flow of $V$ may be described by projecting it to \eqref{flowspace}, taking its inverse image in the domain of \eqref{flowmap}, and taking its image under \eqref{flowmap}; the result is subanalytic since \eqref{flowmap} is proper.
Near the stable manifold, $C_\Lambda$ is the union of this subanalytic set (the image of $C_\Lambda\cap H$ under the backward flow of $V$) with the stable manifold, hence is subanalytic.
\end{proof}

\begin{corollary}\label{analyticskeleton}
Every Weinstein manifold can be perturbed to admit a real analytic structure such that for every subanalytic subset at infinity, the associated relative core is subanalytic.
\end{corollary}

\begin{proof}
The standard Weinstein handle
\begin{equation}\label{weinsteinhandle}
\left(\RR^{2k}\times\RR^{2(n-k)},\sum_{i=1}^ndx_i\wedge dy_i,\sum_{i=1}^k\frac 12(-x_i\partial_{x_i}+3y_i\partial_{y_i})+\sum_{i=k+1}^n\frac 12(x_i\partial_{x_i}+y_i\partial_{y_i})\right)
\end{equation}
is real analytic.
Any critical point of a Weinstein manifold may be perturbed so as to coincide locally with \eqref{weinsteinhandle} (see \cite{cieliebakeliashberg} and \cite[Lemma 6.6]{girouxpardon}).
We may thus construct (after perturbation) any Weinstein manifold by iteratively attaching such standard handles.
Now the attaching maps may be perturbed to be real analytic by Lemma \ref{analyticcontactomorphism}.
We therefore obtain a real analytic Weinstein manifold $(X,\omega,Z)$ to which Proposition \ref{analyticstablemflds} applies.
\end{proof}

\begin{corollary}\label{analyticrelativeskeleton}
Every Weinstein sector is equivalent to a stopped Weinstein manifold with subanalytic isotropic relative core.
\end{corollary}

\begin{proof}
A Weinstein sector is (equivalent to) a Liouville pair $(X,F)$ where $X$ and $F$ are both Weinstein.
Apply Corollary \ref{analyticskeleton} to $X$ and $F$ individually, and apply Corollary \ref{analytichypersurface} to the embedding $F\hookrightarrow\partial_\infty X$.
\end{proof}

\appendix

\section{Review of categorical notions}\label{categoricalsection}

We will assume the reader is familiar with the basic definitions of differential graded (dg) and/or
$\ainf$ categories, functors between them, modules, and bimodules, for which there are many references.
In this section we review notation, assumptions, and relevant notions/results.

All of our dg or $\ainf$ categories $\C$ 
have morphism cochain complexes linear over a fixed commutative ring
(which we take for simplicity of notation to be $\ZZ$), which are $\ZZ$-graded
and cofibrant in the sense of \cite[Sec.\ 3.1]{gpssectorsoc} (an assumption which
is vacuous if working over a field). We further assume that all such $\C$ are
at least cohomologically unital, meaning that the underlying cohomology-level
category $H^*(\C)$ has identity morphisms (this follows if $\C$ itself is strictly unital, as is the case in the dg setting).
We say objects in $\C$ are isomorphic if they are isomorphic in $H^*(\C)$.

\subsection{Functors, modules, and bimodules}\label{functormodulebimodulesec}

For two ($\ainf$ or dg) categories $\C$ and $\D$, we use the notation
\begin{equation}\label{functors}
    \Fun(\C, \D)
\end{equation}
to refer to the ($\ainf$) category of $\ainf$ functors from $\C$ to $\D$ (compare
\cite[Sec.\ (1d)]{seidelbook}, noting that we consider here homologically unital functors).
Note that $\Fun(\C,\D)$ is in fact a dg
category whenever $\D$ is.
The morphism space between $f,g \in \Fun(\C,\D)$ is the \emph{derived} space of natural transformations (as opposed to the space of strict natural transformations, which can be defined in the dg setting but not in the more general $\ainf$ setting).

An $\ainf$ functor $f: \C \to \D$ is called fully faithful (essentially
surjective, an equivalence) if the induced functor on cohomology categories
$H^*(f): H^*(\C) \to H^*(\D)$ is. We use freely the similar notion of a
bilinear $\ainf$ functor $\C \times \D \to \E$ (see
\cite{lyubashenkomultilinear}), which are themselves objects of an $\ainf$
category which is dg if $\E$ is.

Denote by $\Mod\ZZ$ the dg category of dg $\ZZ$-modules, i.e.\ the category of (implicitly $\ZZ$-graded) unbounded complexes of $\ZZ$-modules localized at acyclic complexes.
When relevant, we take as our model of this category cofibrant complexes of $\ZZ$-modules.

A left (respectively right) module over a category $\C$ is, by definition a
functor from $\C^\op$ (respectively $\C$) to $\Mod\ZZ$.  More generally, a $(\C,
\D)$ bimodule is a bilinear functor $\C^\op \times \D \to \Mod\ZZ$; this
notion specializes to the previous two notions by taking $\C$ or $\D = \ZZ$
(meaning the category with one object $*$ and endomorphism algebra $\ZZ$),
see \cite[Sec.\ 3.1]{gpssectorsoc}. By the above discussion, left modules, right
modules, and bimodules are each objects of dg categories, denoted
\begin{align}
\Mod\C&=\Fun(\C^\op,\Mod\ZZ)\\
\Mod\C^\op&=\Fun(\C,\Mod\ZZ)\\
[\C,\D]&=\Fun(\C^\op\times\D,\Mod\ZZ)
\end{align}
respectively. We will most frequently discuss left modules, which we simply
call modules. There are canonical fully faithful Yoneda embeddings (see e.g., \cite[Sec.\ (1l)]{seidelbook} for a more detailed description on morphism spaces):
\begin{align}
    \C \hookrightarrow \Mod \C\quad & X \mapsto \hom_{\C}(-, X)\\
    \C^\op \hookrightarrow \Mod\C^\op\quad & Y \mapsto \hom_{\C}(Y, -)\\
    \C \times \D^\op \hookrightarrow [\C, \D]\quad & (X,Y) \mapsto \hom_{\C}(-,X) \otimes_\ZZ \hom_{\D}(Y,-)
\end{align}
and we call any (bi)module in the essential image of these embeddings
representable.
Recall that any $\C$ possesses a canonical (not necessarily representable)
$(\C, \C)$ bimodule, the diagonal bimodule $\C_{\Delta}$ (defined on the level of objects by
$\C_{\Delta}(-,-) = \hom_{\C}(-,-)$).  

A $(\D, \C)$ bimodule $\B$ induces, via convolution (aka tensor product), a functor
\begin{align}
    \B \otimes_{\C} - : \Mod \C &\to \Mod \D\\
    \M&\mapsto\B(-,-)\otimes_\C\M(-)
\end{align}
(note that this is a version of the derived tensor product), and more generally
a functor $[\C,\E] \to [\D,\E]$ for any category $\E$.  This functor
always has a right adjoint, given by $\N\mapsto\hom_{\Mod\D}(\B,\N)$.\footnote{We say $f: \C \to \D$ has right adjoint (or is the left adjoint of) $g: \D \to \C$ if there is in isomorphism in $[\C, \D]$ between $\hom_{\D}(f(-),-)$ and $\hom_{\C}(-, g(-))$.}
As one might
expect, convolving with the diagonal bimodule is (isomorphic to) the identity.
Not every functor $\Mod \C \to \Mod \D$ comes from a bimodule, however there is
a characterization of those that do:

\begin{theorem}[{compare \cite[Thm.\ 1.4]{toenmorita}}]\label{morita}
    The convolution map $[\D, \C] \to\Fun(\Mod \C, \Mod \D)$ is fully faithful,
    and its essential image is precisely the co-continuous functors, i.e.\ those that preserve small direct sums.
\end{theorem}

(By `$F$ preserves small direct sums' we mean `the natural map $\bigoplus_\alpha F(X_\alpha)\to F(\bigoplus_\alpha X_\alpha)$ is an isomorphism'.)

\mathchardef\mhyphen="2D
\begin{proof}[Proof Sketch]
    If $\Fun_{\mathrm{co\mhyphen cont}}(\Mod \C, \Mod \D)$ denotes the co-continuous functors,
    observe that restriction to (the Yoneda image of) $\C$ induces
    tautologically a map (which is an equivalence) $\Fun_{\mathrm{co\mhyphen cont}}(\Mod \C, \Mod
    \D) \to \Fun(\C, \Mod \D) = \Fun(\C, \Fun(\D^\op, \Mod\ZZ)) = [\D, \C]$; in other words co-continuous
    functors from $\Mod \C$ are determined by what they do on $\C$. One checks
    that this is a two-sided inverse to the convolution map, up to homotopy.
\end{proof}

Given an $\ainf$ functor $f: \C \to \D$, there is a pair of (adjoint) induced
functors on module categories: first, there is an induced restriction map 
\begin{equation}\label{restriction}
    f^*: \Mod \D \to \Mod \C
\end{equation}
given by pre-composing with $f^\op$; one can show this is isomorphic
to tensoring with the graph $(\C, \D)$ bimodule $(f^\op,\id)^*\D_{\Delta}=\D_\Delta(f(-),-)$ (see \cite[Lem.\ 3.7]{gpssectorsoc}).
In particular, there is a natural functor $\D \to \Mod \C$ given by composing
\eqref{restriction} with the Yoneda embedding for $\D$.  There is also (left
adjoint to $f^*$) an induction map
\begin{equation}\label{induction}
    f_!: \Mod \C \to \Mod \D
\end{equation}
given by tensoring with the graph $(\D, \C)$ bimodule $(\id,f)^*\D_{\Delta}=\D_\Delta(-,f(-))$.
One can directly compute that $f_!$ sends a representable
over $X \in \C$ to an object isomorphic to the representable over $f(X)$.
Conversely, we have:

\begin{lemma}
If a $(\D,\C)$ bimodule $\B$ has the property that $\B(-,c)$ is representable
by an object $f(c)\in\D$ each $c \in \C$, then convolving with $\B$ is isomorphic in $\Fun(\Mod\C,\Mod\D)$
to the induction of a (unique up to isomorphism)
$\ainf$ functor $f: \C \to \D$ sending $c$ to $f(c)$. In particular,
$f_! = \B\otimes_\C-$ admits a right adjoint, namely $f^*$. \qed
\end{lemma}

Note that $f^*$ also admits a right adjoint $f_*$, called co-induction, induced
by taking hom from $(f^\op,\id)^*\D_\Delta$, by the earlier discussion.

\subsection{Pre-triangulated, idempotent complete, and co-complete categories}\label{twperfcocompletesec}

A category $\C$ is called pre-triangulated iff it closed under taking mapping cones; in this case $H^0\C$ is triangulated in the usual sense.
Every category has a well defined pre-triangulated closure $\Tw\C$,\footnote{The notation $\Tw\C$ is usually taken to mean the specific model of the pre-triangulated closure of $\C$ given by the category of so-called ``twisted complexes'' of objects of $\C$.}
which can be defined as the closure of the image of $\C$ in $\Mod\C$ under taking mapping cones (see \cite[Sec.\ 3]{seidelbook}).
An object of $\Mod\C$ which is in the closure of $\A\subseteq\C$ under taking mapping cones is said to be generated by $\A$.

A category $\C$ is called idempotent complete iff $H^0\C$ is closed under retracts (usually this property is only considered when $\C$ is already pre-triangulated).
Every category has a well defined idempotent completion $\C^\pi$, which can be defined similarly as a full subcategory of $\Mod\C$ (see \cite[Sec.\ 4]{seidelbook}).
An object of $\Mod\C$ which is in the closure of $\A\subseteq\C$ under taking mapping cones and retracts is said to be split-generated by $\A$.
The category of perfect modules $\Perf\C\subseteq\Mod\C$ is by definition the full subcategory spanned by objects split-generated by $\C$; in other words, $\Perf\C=(\Tw\C)^\pi$.

A category $\CCC$ is called co-complete iff it is pre-triangulated and has all (small) direct sums.
Equivalently, $\CCC$ is co-complete iff it has all small colimits.
In particular, a co-complete category is idempotent complete (compare \cite[Prop.\ 1.6.8]{neeman-book}).
We will also call co-complete categories \emph{large categories}.

\subsection{Compactly generated categories}

Any category of modules $\Mod\C$ (or more generally bimodules, etc.)\ over a small category $\C$ inherits from $\Mod\ZZ$ the property of being co-complete.
Large (i.e.\ co-complete) categories of the form $\Mod\C$ may be characterized intrinsically as follows.

Let $\CCC$ be a large category.
We say an object
$X \in \CCC$ is compact if $\hom_\CCC(X,-)$ commutes with arbitrary
direct sums (i.e.\ is co-continuous). Denoting by $\CCC^c\subseteq\CCC$ the full subcategory of compact objects,
we say that a co-complete category $\CCC$ is compactly generated if there is a small
collection (i.e., a set) of compact objects $\C\subseteq \CCC^c$ satisfying the following equivalent conditions:
\begin{itemize}
\item An object $X\in\CCC$ is zero if and only if it is right-orthogonal to $\C$ (meaning $\hom_\CCC(-,X)$ annihilates $\C$).
\item The natural map $\CCC\to\Mod\C$ sending $Y \mapsto \hom_\CCC(-, Y)$ is an equivalence.
\end{itemize}
Thus compactly generated categories $\CCC$ are precisely those of the form $\Mod\C$ for some small category $\C$.

It is natural to ask to what extent $\Mod\C$ determines $\C$.
This is answered by the following well known fact:

\begin{lemma}\label{compactperf}
The compact objects of $\Mod\C$ are precisely $\Perf\C$.
\end{lemma}

\begin{proof}
Let $\M\in\Mod\C$ be compact.
There is a natural quasi-isomorphism $\C_\Delta\otimes_\C\M\xrightarrow\sim\M$ \cite[Lem.\ 3.7]{gpssectorsoc}\cite[Lem.\ A.1]{gpsdescent} which expresses $\M$ as an infinite twisted complex of Yoneda modules.
Since $\M$ is compact, the inverse quasi-isomorphism factors through some finite subcomplex, so $\M$ is a retract of a finite twisted complex of Yoneda modules.
\end{proof}

In particular, the inclusion $\C \subseteq \Perf \C$ induces an equivalence $\Mod \Perf \C = \Mod \C$.

\subsection{Morita equivalence}\label{moritasec}

We say categories $\C$ and $\D$ are Morita equivalent iff there exists a $(\C,
\D)$ bimodule $\PPP$ and a $(\D, \C)$ bimodule $\Q$ inducing, via
convolution, an inverse pair of equivalences
\begin{equation}\label{moritaeq}
    \Mod\C \stackrel{\simeq}{\longleftrightarrow} \Mod \D.
\end{equation}
Actually, every equivalence $\Mod\C\xrightarrow\sim\Mod\D$ is isomorphic to convolution by a bimodule by Theorem \ref{morita} (since an equivalence is necessarily co-continuous), so $\C$ and $\D$ are Morita equivalent iff there is an equivalence $\Mod\C=\Mod\D$.

\begin{lemma}
$\C$ and $\D$ are Morita equivalent iff there is an equivalence
$\Perf\C = \Perf\D$.
\end{lemma}

\begin{proof}
An equivalence preserves compact objects, so an equivalence between $\Mod\C$ and $\Mod\D$ restricts to an equivalence $\Perf\C=\Perf\D$.
Conversely, any equivalence $\Perf \C=\Perf \D$ induces an equivalence $\Mod\C=\Mod \Perf \C = \Mod \Perf \D=\Mod\D$.
\end{proof}

In particular, the canonical inclusion $\C \hookrightarrow \Perf\C$ is a Morita
equivalence.  In light of the above Lemma, we will also refer to an equivalence $\Perf\C=\Perf\D$ as a Morita equivalence between $\C$ and $\D$.
We say a property of $\C$ is ``a Morita-invariant notion'' if its
validity only depends on $\Perf\C$ up to equivalence.

\subsection{Adjoints and compact objects}

The following is a useful criterion for when a functor preserves compact objects.

\begin{lemma}\label{preservecompactobjects}
If a functor $f: \CCC \to \DDD$ has a co-continuous right adjoint $g$, then $f$ sends compact objects to compact objects.
\end{lemma}

\begin{proof}
For $c \in \CCC$ a compact object, we have 
\begin{multline}
\hom_\DDD\Bigl(f(c), \bigoplus_\alpha d_\alpha\Bigr) = \hom_\CCC\Bigl(c, g\bigl(\bigoplus_\alpha d_\alpha\bigr)\Bigr) = \hom_\CCC\Bigl(c, \bigoplus_\alpha g(d_\alpha)\Bigr)\\
 =  \bigoplus_\alpha \hom_\CCC(c, g(d_\alpha)) = \bigoplus_\alpha \hom_\DDD(f(c), d_\alpha)
\end{multline}
as desired.
\end{proof}

For example, if $f:\C\to\D$ is a functor of small categories, the pullback on module categories $f^*:\Mod\D\to\Mod\C$ is co-continuous and has a left adjoint $f_!:\Mod\C\to\Mod\D$ extending $f$ (see Section \ref{functormodulebimodulesec}) which thus preserves compact objects (a fact which can also be seen from Lemma \ref{compactperf}).

\subsection{Brown representability}

On the level of large categories, a version of Brown
representability gives effective criteria for deducing the existence of
adjoints to functors.

\begin{theorem}[{Compare \cite[Thm.\ 8.4.4]{neeman-book} or \cite[Cor.\ 5.5.2.9]{luriehttpub}}]\label{hasrightadjoint}
    Let $\CCC$ and $\DDD$ be large categories with $\CCC$ compactly generated.
    If an $\ainf$ functor $f: \CCC \to \DDD$ is co-continuous, then $f$ admits a right adjoint.
\end{theorem}

\begin{proof}[Proof Sketch]
We suppose that $\DDD$ is also compactly generated, so one can write $\CCC = \Mod \C$, $\DDD = \Mod \D$ with $\C = \CCC^c$ and $\D = \DDD^c$.
Then we observe that if $f$ is co-continuous, it comes (by Theorem \ref{morita}) from convolving with a bimodule, which always has a right adjoint as described above.
\end{proof}

Theorem \ref{hasrightadjoint} also holds under the
weaker hypothesis that $\CCC$ is \emph{well generated} rather than compactly
generated, by work of Neeman adapted to the dg/$\ainf$ case 
(for a definition of this notion see \cite[Sec.\ 8] {neeman-book}, and
for a proof of Theorem \ref{hasrightadjoint} in that setting, see \cite[Prop.\ 8.4.2 and Thm.\ 8.4.4]{neeman-book}).

\subsection{Quotients and localization}\label{quotients}

Given a (small) $\ainf$ (or dg) category $\C$ and a full subcategory $\D
\subseteq \C$, there is a well-defined notion of the quotient (dg or $\ainf$)
category $\C / \D$
which comes equipped with a functor
\begin{equation}
q: \C \to \C / \D
\end{equation}
(see
\cite{drinfelddgquotient, lyubashenkoovsienko} for an explicit model in the dg
and $\ainf$ cases respectively, also discussed in \cite[Sec.\ 3.1.3]{gpssectorsoc}). The pair $\C/\D$ and $q$ satisfy the following
universal property: any functor $\C \to \E$ which sends $\D$ to $0$ factors
essentially uniquely through $\C/\D$ via $q$; more precisely, the pre-composition
$q^*: \Fun(\C/\D, \E) \hookrightarrow \Fun(\C, \E)$ fully
faithfully embeds the former category as the full subcategory $\Fun_{\Ann(\D)}(\C, \E)$ of the latter consisting of functors from
$\C$ to $\E$ which annihilate $\D$.
Taking $\E$ to be $(\Mod\ZZ)^\op$, we note in
particular that the pullback map
\begin{equation}\label{fullyfaithfulrightadjoint}
    q^*: \Mod(\C/\D) \to \Mod\C
\end{equation}
is a fully faithful embedding whose essential image is the $\C$ modules which annihilate $\D$ (see \cite[Lem.\ 3.12 and 3.13]{gpssectorsoc}).
It follows from these universal properties that the quotient $\C/\D$ depends only on the full subcategory of $\C$ split-generated by $\D$.
If $\C$ is pre-triangulated then so is $\C/\D$, however be warned that $\C/\D$ need not be idempotent complete even if $\C$ is.

In light of \eqref{fullyfaithfulrightadjoint}, we have the following equivalent perspective on localization in terms of large categories.
Let $\CCC$ be a compactly generated large category, and let $\DDD \subseteq \CCC$ be a full subcategory closed under cones and arbitrary direct sums (hence itself a large category) which is compactly generated by a
subset of $\CCC$'s compact objects $\D \subseteq \C:=\CCC^c$ (conversely, any full subcategory $\D\subseteq\C$ gives rise to such a $\DDD\subseteq\CCC$, namely the image of the induced functor $\Mod\D\to\Mod\C$, which is co-continuous since it is left adjoint to pullback of modules, and is fully faithful because the unit of this adjunction is an isomorphism \cite[Lem.\ 3.7]{gpssectorsoc}).
The quotient of $\CCC$ by $\DDD$, denoted $\CCC/\DDD$, is by definition the full subcategory of $\CCC$ which is right-orthogonal to $\DDD$ (we may also write this as $\CCC/\D$; note that the universal property of direct sum implies that the right-orthogonal of $\DDD$ is the same as the right-orthogonal of $\D$).
According to \eqref{fullyfaithfulrightadjoint}, this quotient $\CCC/\DDD$ is precisely $\Mod(\C/\D)$.
Thus the large quotient $\CCC/\DDD$ is compactly generated, and passing to compact objects recovers the quotient of categories of compact objects, up to Morita equivalence.
By the discussion in Section \ref{functormodulebimodulesec}, the embedding $q^*:\CCC/\DDD\hookrightarrow\CCC$ is right adjoint to a functor $q_!:\CCC\to\CCC/\DDD$ extending $q:\C\to\C/\D$.
We thus conclude:

\begin{lemma}[{Compatibility of large quotients with compact objects, compare \cite[Thm.\ 2.1]{neeman-smashing}}] \label{compactinquotient}
Let $\DDD\subseteq\CCC$ be a co-continuous inclusion of compactly generated large categories which sends compact objects to compact objects.
The large quotient $\CCC/\DDD$ (by definition the right-orthogonal to $\DDD\subseteq\CCC$) is also a compactly generated large category with co-continuous inclusion into $\CCC$.
The fully faithful inclusion $q^*:\CCC/\DDD\to\CCC$ is right adjoint to a `quotient functor' $q_!:\CCC\to\CCC/\DDD$ whose restriction to compact objects is the corresponding quotient functor on small categories $q:\C\to(\C/\D)^\pi$ (with idempotent-completed target).
\qed
\end{lemma}

If $\C$ is a pre-triangulated dg/$\ainf$ category and $Z$ is a set of
morphisms in $H^0(\C)$, one can form the localization of $\C$ with respect to
$Z$ by taking the quotient
\begin{equation}
    \C[Z^{-1}] := \C/{\cones Z}
\end{equation}
where $\cones Z$ denotes any set of cones of morphisms in $\C$ representing
the elements in $Z$ (regardless of how one chooses such a subset, one
notices that $\cones Z$ is a well-defined full subcategory of $\C$, and in
particular, $\C[Z^{-1}]$ is unaffected by the choice). If $\C$ is not
pre-triangulated, one can still define this localization by taking the
essential image of $\C$ under 
\begin{equation}
    \C \to \Tw \C \to \Tw\C / (\cones Z).
\end{equation}
The tautological localization map $\C \to \C[Z^{-1}]$ possesses a host of nice
properties, simply as a special case of the properties of quotients discussed above; we leave it to the
reader to spell out the details.

\subsection{Proper modules}\label{propsec}

Recall that $\Perf\ZZ \subseteq\Mod\ZZ$ is
the subcategory of perfect $\ZZ$-linear chain complexes, namely those chain complexes
which are quasi-isomorphic to a bounded complex of finite projective
$\ZZ$-modules.

We say a module or bimodule is proper (sometimes called pseudo-perfect in the
literature) if as a functor to $\Mod\ZZ$, it takes values in the full subcategory
$\Perf\ZZ$ (i.e.\ for a module $\M$ if
$\M(X)$ is a perfect chain complex for every $X \in \C$). Denote by 
\begin{equation}
    \Prop\C := \Fun(\C^\op, \Perf\ZZ) \subseteq \Mod \C
\end{equation}
the full subcategory of proper modules.

\subsection{Smooth and proper categories}

We say  a category $\C$ is smooth (sometimes called homologically smooth) if its diagonal bimodule $\C_{\Delta}$ is perfect
(a $(\C,\D)$ bimodule is called perfect if it is split-generated by tensor products of representable bimodules $\hom_\C(-,X)\otimes\hom_\D(Y,-)$).

We say $\C$ is proper
(sometimes called compact) if its diagonal bimodule $\C_{\Delta}$ is proper, or
if equivalently 
$\hom_{\C}(X,Y)$ is a perfect $\ZZ$-module for any two
objects $X,Y \in \C$. Smoothness and properness are Morita-invariant notions;
in particular $\C$ is smooth (resp.\ proper) if and only if $\Perf\C$ is.

In general, the subcategories of modules $\Perf\C$ and $\Prop\C$ do not coincide,\footnote{Rather, they are in some sense `Morita dual' in that $\Prop\C=\Fun(\Perf\C^\op,\Perf\ZZ)$.}
however they are related under the above finiteness assumptions on $\C$: 
\begin{lemma}\label{propperf}
    If $\C$ is proper then $\Perf \C \subseteq \Prop \C$, and if $\C$ is smooth
    then $\Prop \C \subseteq \Perf \C$.
    In particular, if $\C$ is smooth and proper, then $\Prop\C=\Perf\C$.\qed 
\end{lemma}

\begin{lemma} \label{smoothquotient}
    Properness is inherited by full subcategories, and smoothness passes to quotients/localizations.
    \qed 
\end{lemma}

\subsection{Exceptional collections}

We say a (full) subcategory of finitely many objects $\A \subseteq \C$ is an exceptional collection
if there exists a partial ordering of the objects of $\A$ such
that
\begin{align}
    \hom(X,X) &= \ZZ\langle \id_X \rangle, \\
    \hom(X,Y) &= 0\quad\text{unless }X \leq Y.
\end{align}

\begin{lemma}\label{exceptionalidempotentcomplete}
    If $\N \in \Mod\C$ is split-generated by an exceptional collection
    $\A \subseteq \C$, then $\N$ is generated by $\A$ (i.e.\ it is not necessary to add direct
    summands).
\end{lemma}

\begin{proof}
Let $X\in\A$ be any maximal (with respect to the given partial order) object.
We consider the functor
\begin{align}
F_X:\Mod\C&\to\Mod\ZZ\\
\M&\mapsto\M(X).
\end{align}
Certainly if $\M$ is generated by $\A$ (i.e.\ by the Yoneda modules $\hom_\C(-,A)$ for $A\in\A$), then $F_X(\M)\in\Perf\ZZ$ by maximality of $X$, as all of the Yoneda modules except $\hom_{\C}(-,X)$, contribute trivially to $F_X$, and each $\hom_{\C}(-,X)$ contributes a perfect $\ZZ$-module.

There is a tautological map of $\C$ modules $\hom_\C(-,X)\otimes F_X(\M)\to\M(-)$; denote its cone by $\M|_{\A-\{X\}}$.
Now given any maximal object $Y$ of $\A-\{X\}$, we may define a functor
\begin{align}
F_Y:\Mod\C&\to\Mod\ZZ\\
\M&\mapsto\M|_{A-\{X\}}(Y).
\end{align}
Again, if $\M$ is generated by $\A$ then $F_Y(\M)\in\Perf\ZZ$.
To see this, simply note that given a twisted complex $\M$ of objects of $\A$, the object $\M|_{\A-\{X\}}$ is just the same twisted complex but with all instances of $X$ deleted.
We may now similarly define $\M|_{\A-\{X,Y\}}$ to be the cone of $\hom_\C(-,Y)\otimes F_Y(\M)\to\M|_{\A-\{X\}}(-)$.

Iterating this procedure defines a sequence of functors $F_X:\Mod\C\to\Mod\ZZ$ for all $X\in\A$ (in fact, these are independent of the order in which we pick off maximal elements, however we won't use this).
The above arguments show that for any $\M$ generated by $\A$, all $F_X(\M)$ are in $\Perf\ZZ$ (and hence the same holds for $\M$ split-generated by $\A$).
They also show (for arbitrary $\M$) that if all $F_X(\M)$ are in $\Perf\ZZ$, then there exists $\M'\in\Mod\C$ generated by $\A$ and a map $\M'\to\M$ which is an isomorphism in $\Mod\A$.

We may now conclude: if $\N$ is split-generated by $\A$, then $F_X(\N)\in\Perf\ZZ$, so there is $\N'\in\Mod\C$ generated by $\A$ and a map $\N'\to\N$ which is an isomorphism in $\Mod\A$, and since $\N'$ and $\N$ are split-generated by $\A$, an isomorphism in $\Mod\A$ is an isomorphism in $\Mod\C$.
\end{proof}

\begin{lemma}\label{exceptionalpropersmooth}
If $\A$ is an exceptional collection which is proper, then it is smooth.
\end{lemma}

\begin{proof}
    In the case $\A$ has one object, this is true because $\ZZ$ is trivially
    smooth.  Now inductively apply the following assertion: If $\C$ and $\D$ are both smooth, and
    $\E$ denotes the semi-orthogonal gluing of $\C$ with $\D$ along a
    $(\C,\D)$ bimodule $\B$ which is perfect, then $\E$ is smooth as well (see \cite[Prop.\ 3.11]{lunts} and \cite[Thm.\ 3.24]{luntsschnurer} for the dg case, which
    immediately extends to this setting). In the assertion observe it suffices
    that $\B$ be proper, since proper bimodules over smooth categories
    are automatically perfect (by the bimodule version of Lemma \ref{propperf}).
    Hence, one can induct from $\ZZ$ to any proper exceptional collection $\A$.
\end{proof}

\bibliographystyle{amsplain}
\bibliography{wrappedconstructible}
\addcontentsline{toc}{section}{References}

\end{document}